%% file: 3folds.tex
\documentclass{amsart}
\usepackage{amsmath,amssymb,amscd,amsthm}
\usepackage{latexsym,graphics,fullpage}
\usepackage{hyperref}
\usepackage{url}
\usepackage{wasysym}
\usepackage{mathtools}
\usepackage[all]{xy}
\usepackage[toc,page]{appendix}
\usepackage{tikz} 
\usepackage{standalone}

\newtheorem{theorem}{Theorem}[section]
\newtheorem{proposition}[theorem]{Proposition}
\newtheorem{lemma}[theorem]{Lemma}
\newtheorem{corollary}[theorem]{Corollary}

\newtheorem{claim}[theorem]{Claim}

\newtheorem{manualtheoreminner}{Theorem}
\newenvironment{manualtheorem}[1]{%
	\renewcommand\themanualtheoreminner{#1}%
	\manualtheoreminner
}{\endmanualtheoreminner}

\theoremstyle{definition}
\newtheorem{definition}[theorem]{Definition}

\newtheorem{example}[theorem]{Example}

\theoremstyle{remark}
\newtheorem{remark}[theorem]{Remark}

\numberwithin{equation}{section}

\newcommand{\bbC}{\mathbb{C}}
\newcommand{\bbF}{\mathbb{F}}

\newcommand{\bbP}{\mathbb{P}}
\newcommand{\bbQ}{\mathbb{Q}}

\newcommand{\bbZ}{\mathbb{Z}}

\newcommand{\calO}{\mathcal{O}}

\newcommand{\Esx}{\tilde{E}_6}
\newcommand{\Esv}{\tilde{E}_7}
\newcommand{\Ee}{\tilde{E}_8}

\newcommand{\Ab}{\tilde{A}}
\newcommand{\Bb}{\tilde{B}}
\newcommand{\Cb}{\tilde{C}}
\newcommand{\Xb}{\tilde{X}}
\newcommand{\Qb}{\tilde{Q}}

\DeclareMathOperator{\Pic}{Pic}

\DeclareMathOperator{\Def}{Def}
\newcommand{\Mod}[1]{\ (\mathrm{mod}\ #1)}

\begin{document}

\title{On the classification of singular cubic threefolds}
\author{Sasha~Viktorova}
\bibliographystyle{amsalpha}
\maketitle

\vspace*{-2em}
\begin{abstract}
	We classify combinations of isolated singularities that can occur on complex cubic threefolds generalizing analogous results for cubic surfaces (\cite{Schlafli}, \cite{BW}). In addition, we provide concise combinatorial description of the possible configurations of simple singularities: they essentially correspond to subgraphs of a certain graph.
\end{abstract}

\section*{Introduction}

The aim of this paper is to give a complete classification of combinations of isolated singularities appearing on compex cubic threefolds (that is, hypersurfaces of degree three in complex projective four-space). In our work we follow a strategy similar to the one employed by Bruce and Wall \cite{BW} to classify singularities on cubic surfaces. Originally, the classification for cubic surfaces was obtained by Schl\"{a}fli \cite{Schlafli}. Similar classifications for certain classes of singularities on quartic and quintic surfaces were acquired by Urabe \cite{urabe} and Yang \cite{Yang} via lattice theory. More recently, Stegmann \cite{Stegmann} used lattice-theoretic methods to give a partial classification for singularities on cubic fourfolds.

\medskip

The study of singular cubic threefolds goes back to the Italian school, and to Corrado Segre in particular. In \cite{Segre}, he shows that the maximal number of nodes on a cubic threefold is 10, and there is a unique such cubic (up to projective equivalence). The Segre cubic is studied in detail in a paper by Dolgachev \cite{DolgSegre}. Further steps towards classification of singular cubic threefolds include Yokoyama's and Allcock's lists of several maximal combinations of $A_n$ singularities in the context of GIT analysis (\cite{yokoyama}, \cite{allcock1}), and the description of cubic threefolds admitting a  $\bbC$ or $\bbC^*$-action by du~Plessis and Wall \cite{DPW08} (see Appendix~\ref{Sym_Threefolds}). Singular cubic threefolds also play a key role in the study of degenerations of intermediate Jacobians (\cite{CG}, \cite{GrHulek}, \cite{CMGHL}) and a certain construction of an irreducible holomorphic symplectic manifold (\cite{LSV}).

\medskip

Our primary approach to the classification problem is the projection method of \cite{BW}. Specifically, a singular cubic $X\subset\bbP^4$ with isolated singularities which is not a cone is rational via a projection $\pi$ from a singular point $p$. The base locus of the inverse map $\pi^{-1}$ is a $(2,3)$ complete intersection curve $C\subset\bbP^3$. By a theorem of Wall \cite[Theorem 2.1]{W99} (see Theorem \ref{projthm}), there is a close connection between the singularities of $X$ and the singularities of $C$. In addition, we use a deformation theory result that allows us to complete the classification efficiently. Concretely, the global deformations of a cubic with isolated singularities which is not a cone give independent versal deformations of its singularities (\cite[Proposition 4.1]{DPW00}; see Theorem \ref{F:DPW}). In particular, there is a well-defined notion of maximal configurations of isolated singularities on such cubics (Definition \ref{max_comb}). The rest of the possible configurations of singularities can be recovered from the maximal ones using results of Brieskorn, Grothendieck and Slodowy (\cite{Brieskorn}, \cite{slodowy}; see Theorem \ref{Gr} and Table \ref{unimodal} below). We now state the main theorem of the paper:

\smallskip

\begin{manualtheorem}{I}\label{Theorem1}
	The complete list of the combinations of isolated singularities occurring on complex cubic threefolds consists of 204 cases and is contained in Tables \ref{Ta:ca3list}, \ref{Ta:ca2list} and \ref{Ta:anlist}. Moreover, the following configurations are maximal (in the specified classes):
	\begin{enumerate}
		\item The maximal simple ($ADE$) configurations are $E_8+A_2$, $E_7+A_2+A_1$, $E_6+2A_2$, $D_8+A_3$, $D_6+A_3+2A_1$, $D_5+2A_3$, $3D_4$, $A_{11}$, $A_7+A_4$, $2A_5+A_1$, $5A_2$, and $10A_1$.
		\item The maximal constellations of $A_n$ singularities are $A_{11}$, $A_7+A_4$, $2A_5+A_1$, $3A_3+A_1$, $2A_3+A_2+2A_1$, $2A_3+4A_1$, $5A_2$, and $10A_1$.
	\end{enumerate}
\end{manualtheorem}

The paper \cite{BW} includes an elegant description of the configurations of isolated singularities on cubic surfaces via subdiagrams of the $\Esx$ diagram (\cite[Section 4]{BW}; see Theorem \ref{BW}). Inspired by this result, we find a compact way to encode the $ADE$ combinations on cubic threefolds:

\begin{manualtheorem}{II}
	An $ADE$ combination of singularities occurs on a cubic threefold if and only if the union of the corresponding Dynkin diagrams is $10A_1$, $5A_2$, or an induced subgraph of the graph $\Gamma$ (Figure \ref{gamma}).
\end{manualtheorem}

\vspace*{-2em}

\tikzset{every picture/.style={line width=0.75pt}} 

\begin{figure}[htb]
	\begin{center}
		\includestandalone{Gamma}
		\caption{Graph $\Gamma$}
		\label{gamma}
	\end{center}
\end{figure}

\medskip

\noindent
\textbf{Outline of the paper.}
We start with necessary preliminaries (Section \ref{Preliminaries}): introduction to the projection method and Dynkin diagrams of singularities, a brief overview of the results of \cite{BW} for singular cubic surfaces, some deformation theory and Milnor lattice theory. In the following Sections \ref{sect_cr3}, \ref{sect_cr2} and \ref{sect_an} we study the possible combinations of isolated singularities on cubic threefolds. The main result of each of these sections is the list of configurations with a fixed corank of the worst singularity of a threefold (see Definition~\ref{corank}). Theorem \ref{Theorem1} is obtained in the end of Section \ref{sect_an} as a compilation of these results. We prove Theorem \ref{Theorem2} in Section \ref{sect_comb}. The appendices contain basic singularity and lattice theory, a summary of the partial classification of du Plessis and Wall \cite{DPW08}, and tables with the complete list of combinations of isolated singularities appearing on cubic threefolds.

\medskip

\noindent
\textbf{Conventions.} We work over the field of complex numbers. In Sections \ref{sect_cr3}, \ref{sect_cr2} and \ref{sect_an}, we assume all the singularities to be isolated. We will be concerned with the following types of singularities:
\begin{itemize}
	\item $A_n$, $D_n$ ($n\ge4$), $E_6$, $E_7$, and $E_8$ (simple or $ADE$ singularities),
	\smallskip
	\item $P_8=\Esx=T_{333}$, $X_9=\Esv=T_{244}$, and $J_{10}=\Ee=T_{236}$ (parabolic or simply-elliptic),
	\smallskip
	\item $T_{pqr}$ with $\frac{1}{p}+\frac{1}{q}+\frac{1}{r}<1$ (hyperbolic),
	\smallskip
	\item $U_{12}$, $S_{11}$, $Q_{10}$ (exceptional unimodal),
	\smallskip
	\item $O_{16}$ (cone over a smooth cubic surface).
\end{itemize}
The standard local equations of these singularities are given for instance in Chapter 15.1 of \cite{Arnold2012}. Notice that we use the same notation for stably equivalent singularities (Definition \ref{stable}).

\medskip

\noindent
\textbf{Acknowledgements.} First of all, I would like to thank my Ph.D. advisor, Radu Laza, for suggesting to me this topic and for his guidance. I would also like to thank Andrei Ionov, Lisa Marquand, Olivier Martin, Aleksei Pakharev and Grisha Papayanov for helpful discussions and comments on the paper draft. This paper has been completed at KU Leuven where I am supported by Methusalem grant METH/21/03 -- long term structural funding of the Flemish Government.

\section{Preliminaries} \label{Preliminaries}

\subsection{Notation} \label{notation}

Let $X\subset\bbP^n$ be a complex cubic hypersurface ($n\geq 3$). We fix a singular point $p\in X$ and choose coordinates in which $p=[1:0:\ldots:0]$. In these coordinates $X$ is defined by the equation
$$
f(x_0,\ldots,x_n)=x_0f_2(x_1,\ldots,x_n)+f_3(x_1,\ldots,x_n)
$$
where $f_2(x_1,\ldots,x_n)$ and $f_3(x_1,\ldots,x_n)$ are homogeneous polynomials of degree 2 and 3 respectively.

\medskip

Let $N$ be the hyperplane at infinity defined by $x_0=0$, $Q\subset N$ be the quadric hypersurface defined by $f_2=0$, $S\subset N$ be the cubic hypersurface defined by $f_3=0$, and $C$ be the intersection of $Q$ and $S$.

\begin{remark} \label{SmodQ}
	While $Q$ and $C$ are uniquely determined by the singular point $p$, the cubic hypersurface $S$ is only defined modulo $Q$. If we choose a hyperplane $N'$ with the equation $x_0'=x_0-\sum_{i=1}^na_ix_i$ then $f_3'(x_1,\ldots,x_n)=f_3+(\sum_{i=1}^na_ix_i)f_2$. Thus $X$ can be defined by the equation $f'(x_0',x_1,\ldots,x_n)=x_0'f_2+f_3'$.
\end{remark}

\subsection{The projection method} \label{projection}

For the rest of this subsection, we will use the notation above and assume that $X$ is irreducible, not a cone, and has an isolated singularity at $p$. Notice that if $X$ is irreducible, then $f_3\not\equiv 0$, and if $X$ is not a cone, then $f_2\not\equiv 0$.

\begin{proposition}
	\hfill
	\begin{itemize}
		\item[a)] $Q$ and $S$ do not have common components.
		\item[b)] The lines $L\subset X$ passing through $p$ are in $1:1$ correspondence with the points of $C$.
		\item[c)] Let $q\in X$ be a singular point other than $p$. Then the line $L= <p,q>$ is contained in $X$ and the only singular points of $X$ on $L$ are $p$ and $q$.
	\end{itemize}
\end{proposition}

\begin{proof}
	If $Q$ and $S$ have a common component, then $X$ is reducible which contradicts the assumptions of this subsection.
	
	\smallskip
	
	A line passing through $p$ intersects $X$ at $p$ with multiplicity at least $2$. In particular, such a line either meets $X$ at exactly one point other than $p$, intersects $X$ at $p$ with multiplicity $3$ or is contained in $X$. Now, b) follows from the fact that $Q$ is the projectivized tangent cone to $X$ at $p$ and can be interpreted as the locus of lines intersecting $X$ at $p$ with multiplicity at least $3$. Thus if $x\in C(= Q\cap S\subset X)$, the intersection number of the line $L=<p,x>$ and $X$ is at least $4$ (multiplicity $\geq3$ at $p$ and $\geq1$ at $x$), which means $L$ is contained in $X$. The converse holds by a similar argument.
	
	\smallskip
	
	The first part of c) is immediate. For the second part, we note that $Sing(X)$ is cut by quadric hypersurfaces (the partial derivatives of $f$). 
	If there are three points $p,q,r\in Sing(X)$ on a line $L$ then $L$ has to be contained in each of the quadrics. Thus $L\subset Sing(X)$ and the singularity at $p$ is not isolated.
\end{proof}

Let $\epsilon:\Xb\to X$ be the blow-up of $X$ at $p$ with the exceptional divisor $E$. Let $\pi:X\dashrightarrow N\cong\mathbb{P}^{n-1}$ be the projection from $p$ onto $N$.

\begin{corollary} \label{corblowup}
	The projection $\pi:X\dashrightarrow N$ is a birational map, and there is a unique birational morphism $\phi:\Xb\to N$ that fits into the diagram
	$$\xymatrix@R=.25cm@C=.25cm{
		&\ar[dr]^\phi\ar[dl]_\epsilon {\Xb}&\\
		{X}\ar@{-->}[rr]^\pi&&{N}.
	}$$
	Furthermore, the restriction of $\phi$ to the exceptional divisor $E$ gives an isomorphism $\phi\vert_E:E\to Q\subset N$.
\end{corollary}

The following theorem, adapted from \cite{W99} to the case of cubic hypersurfaces, shows how the singularities of $C$ determine the singularities of $X$ away from $p$. This is the result we refer to as the \textit{projection method} or the \textit{projection theorem}. A similar statement for cubic surfaces appears in \cite{BW}.

\begin{theorem}[{\cite{W99}, Theorem 2.1}] \label{projthm}
	Consider a point $q\in C$. If $Q$ and $S$ are both singular at $q$, then $X$ is singular along the line $<p,q>$. Otherwise write $T$ for the type of the singularity of $C$ at $q$ in the (locally) smooth variety $Q$ (or $S$).
	
	\begin{itemize}
		\item[(i)] If $Q$ is smooth at $q$, $X$ has a unique singular point on the line $<p,q>$ other than $p$, and the singularity there has type $T$.
		\item[(ii)] If $Q$ is singular at $q$, the only singular point of $X$ on $<p,q>$ is $p$, and the blow-up $\Xb$ of $X$ at $p$ has a singular point of type $T$ at $\phi\vert_E^{-1}(q)$ where $\phi$ is as in Corollary \ref{corblowup}.
	\end{itemize}
	
\end{theorem}

\subsection{Dynkin diagrams of singularities} \label{DynkinDiagrams}

Consider a holomorphic function $f: \bbC^n\rightarrow\bbC$ which has an isolated critical point at $p=(0,\ldots,0)$. If we pick sufficiently small neighbourhoods $p\in U\subset\bbC^n$ and $f(p)\in T\subset\bbC$, the hypersurface $X_t=f^{-1}(t)\cap U$ is smooth for each $t\in T$, $t\neq f(p)$. One can choose a \textit{distinguished basis} $\delta_1,\ldots,\delta_\mu$ of \textit{vanishing cycles} in $H_{n-1}(X_t,\bbZ)$ (see e.g. Section 1 of \cite{Ebeling2019} for the full definitions and properties of vanishing cycles and distinguished bases). The number $\mu$ of elements in this basis is equal to the Milnor number of $p$ (Definition \ref{Milnor}).

\begin{definition}
	Let $\langle\:,\:\rangle$ be the intersection form on $H_{n-1}(X_t,\bbZ)$. The matrix $(\langle\delta_i,\delta_j\rangle)_{i,j=1,\ldots,\mu}$ is called the \textit{intersection matrix} of the singularity of $f$ at $p$ with respect to the distinguished basis $\delta_1,\ldots,\delta_\mu$.
\end{definition}

\begin{proposition} [{\cite{Ebeling2019}, Proposition 1}]
	A vanishing cycle $\delta$ has the self-intersection number
	\[ \langle\delta,\delta\rangle = (-1)^{n(n-1)/2}(1-(-1)^n) = \begin{cases} \;\;\:0 & \text{for $n$ even,} \\ \;\;\:2 & \text{for $n\equiv1\Mod4$,} \\ -2 & \text{for $n\equiv3\Mod4$.} \end{cases} \]
\end{proposition}

\begin{definition} \label{stable}
	A \textit{stabilization} of $f$ is a function $\tilde{f}:\bbC^{n+k}\rightarrow\bbC$ of the form
	$$
	\tilde{f}=f+z_{n+1}^2+\ldots+z_{n+k}^2.
	$$
	Two function-germs of different number of variables are said to be \textit{stably equivalent} if they admit equivalent stabilizations.
\end{definition}

Stably equivalent function-germs have the same Milnor and Tjurina numbers. While their intersection forms may have different properties (for instance, such form is symmetric for odd $n$ and skew-symmetric for even $n$), they determine one another. Moreover, there are exactly four distinct intersection forms in a class of stably equivalent singularities (\cite{Ebeling2019}, Theorem 13).

\begin{definition}
	We call the symmetric form $(\:,\:)$ on $H_{n-1}(X_t,\bbZ)$ associated with the intersection matrix of the stabilization $\tilde{f}$ of $f$ in $n+k$ variables such that $n+k\equiv1\Mod4$ the \textit{stable intersection form} of the singularity. The group $H_{n-1}(X_t,\bbZ)$ together with $(\:,\:)$ is the \textit{(stable) Milnor lattice} of $p$.
\end{definition}

For the rest of this paper, we choose to work with distinguished bases corresponding to a stabilization with $n+k\equiv1\Mod4$ and stable intersection forms.

\medskip

For a distinguished basis $\delta_1,\ldots,\delta_\mu$, we construct the corresponding \textit{Dynkin diagram}. A vertex labeled $i$ is assigned to each root $\delta_i$; two vertices $i$ and $j$ are connected by a (dashed) edge with index $k$ if $(\delta_i,\delta_j)=k$ and $k<0$ (resp. $k>0$). Since the self-intersection $(\delta_i,\delta_i)=2$ for any distinguished basis element $\delta_i$, the Dynkin diagram completely determines the quadratic form.

\medskip

The following proposition relates the notion of Dynkin diagrams above with the usual notion of $ADE$ Dynkin diagrams coming from Lie theory:

\begin{proposition} [{\cite{AGLV}, Chapter 2.2.5}]
	The quadratic form of a simple singularity of type $L$ is isomorphic to the one given by the Dynkin diagram of type $L$.
\end{proposition}

We can describe all of the adjacencies (see Definition \ref{adjacent}) of $ADE$ singularities using the theorems of Brieskorn, Grothendieck and Slodowy below. Before stating their results, recall the definition of an induced subgraph:

\begin{definition} \label{induced}
	An \textit{induced subgraph} of a graph $G$ is a subset of the vertices of $G$ together with any edges whose endpoints are both in this subset.
\end{definition}

\begin{theorem}[{Brieskorn--Grothendieck, \cite{AGLV}, Chapter 2.5.9}]
	A simple singularity of type $L$ is adjacent to a simple singularity of type $K$ if and only if the Dynkin diagram of $K$ embeds into the Dynkin diagram of $L$.
\end{theorem}

\begin{definition}
	A \textit{combination of singularities} is the unordered set of all singularities on a variety or complex space. If the combination consists of singularities $K_1,\ldots,K_m$, we denote it as $K_1+\ldots +K_m$.  We use the terms combination, \textit{constellation} (in the $A_n$ case) or \textit{configuration of singularities} interchangeably.
\end{definition}

\begin{theorem}[{Brieskorn--Grothendieck--Slodowy, \cite{slodowy}, Chapter 8.10}] \label{Gr}
	A simple singularity of type $L$ is adjacent to a combination of simple singularities $K_1+\ldots+K_m$ if and only if the disjoint union of the Dynkin diagrams of $K_1,\ldots,K_m$ is an induced subgraph of the Dynkin diagram of $L$.
\end{theorem}

In general, any partition of an isolated singularity corresponds to a partition of a corresponding Dynkin graph, however such a Dynkin graph may not be unique.

\subsection{Cubic surfaces} \label{Cubic_surfaces}

A classification of cubic surfaces by their singularities was given by Schl\"{a}fli \cite{Schlafli} over a century ago. In \cite{BW}, Bruce and Wall present Schl\"{a}fli's classification in a more concise and modern way. Their result can be stated in a form similar to Theorem \ref{Gr} from the previous subsection:

\begin{theorem}[{\cite{BW}, Section 4}] \label{BW}
	A cubic surface with only isolated singularities has either one $\Esx$ singularity or a combination of $ADE$ singularities. A combination of $ADE$ singularities can occur on such a surface if and only if
	the disjoint union of the corresponding Dynkin diagrams is an induced subgraph of $\Esx$ (Figure~\ref{Esx}).
\end{theorem}

\begin{figure}[htb]
	\begin{center}
		\includestandalone{E_6_tilde}
		\bigskip
		\caption{$\Esx$ graph (left) decomposes into $4A_1$}
		\label{Esx}
	\end{center}
\end{figure}

\begin{remark}
	The $\Esx$ graph above corresponds to a partial basis of vanishing cycles of an $\Esx$ singularity. However, this graph is not its full Dynkin diagram. Singularities of $\Esx$ type have Milnor number 8, while the $\Esx$ graph has 7 vertices. The notation $\Esx$ comes from extended Dynkin diagrams, and is more appropriate in the context of Theorem \ref{BW} than the equivalent $P_8$ or $T_{333}$ notation.
\end{remark}

For the rest of this subsection, let $X\subset\bbP^3$ be a cubic surface with an isolated singular point $p\in X$. Let $Q$, $S$ and $C$ be as in Section \ref{notation}. Note that if $Q\subsetneq\bbP^2$, then $C$ is a collection of $6$ points (taken with multiplicities). The analytic type of the singularity at $p$ can be determined by the types of singularities of $C$ along the singular locus of $Q$ (see Table \ref{Ta:cubicsurf}). It is convenient to organize the possible cases by the type of conic $Q$:

\begin{proposition}[{\cite{BW}, Section 2}]
	\hfill
	\begin{itemize}
		\item[1)] If $Q=\bbP^2$, then the singularity at $p$ is of type $\Esx$ (a cone over a smooth elliptic curve).
		\item[2)] If $Q$ is a double line, then the singularity at $p$ is of type $D_4$, $D_5$ or $E_6$ depending on the type of intersection of $Q$ and $S$ (three simple points, one simple and one double point, and one triple point respectively).
		\item[3)] If $Q$ is reducible, but reduced, i.e. two distinct lines meeting in $v$, then the singularity at $p$ is of type $A_n$ for $2\le n\le 5$, where $n$ is determined by the intersection multiplicity of $Q$ and $S$ at $v$ (which can be $0$, $2$, $3$ and $4$ respectively).
		\item[4)] If $Q$ is an irreducible (thus smooth) conic, then the singularity at $p$ is $A_1$.   
	\end{itemize}
\end{proposition}

\renewcommand{\arraystretch}{1.25}

\begin{table}[h]
	\begin{center}
		\begin{tabular}{|c | c | c|}\hline
			Singularity at $p$& Singular locus of $Q$&  Singularities of C contained in $Sing(Q)$ \\\hline \hline
			$E_6$& a line &$A_2$ (triple point)\\ \hline
			$D_5$& a line &$A_1$ (double point)\\ \hline
			$D_4$& a line & - \\ \hline
			$A_5$& a point &$A_3$ (point of multiplicity 4)\\ \hline
			$A_4$& a point &$A_2$\\ \hline
			$A_3$& a point &$A_1$\\ \hline
			$A_2$& a point & - \\ \hline
			$A_1$& - & - \\ \hline
		\end{tabular}
		\bigskip
		\caption{Simple singularities of cubic surfaces} \label{Ta:cubicsurf}
		\vspace*{-1em}
	\end{center}
\end{table}

\renewcommand{\arraystretch}{1}

After one has determined the type of singularity at $p$ on a cubic surface $X$, the types of the remaining singularities can be easily established by Theorem \ref{projthm}. One can also use the projection method to construct examples of cubic surfaces with a prescribed combination of singularities. Below we explain this in a bit more detail in the case when $Q$ is smooth.

\begin{example}
	Assume $Q$ is smooth and thus $X$ has an $A_1$ singularity at $p$ (see Table \ref{Ta:cubicsurf}). A point of multiplicity $m$ in $C$ corresponds to an $A_{m-1}$ singularity on $X$ by Theorem \ref{projthm}. For instance, if $C$ consists of three double points, then we get three additional $A_1$ singularities on $X$, and together with $p$ there are $4A_1$ singularities on $X$ in total.
	
	\smallskip
	
	Conversely, for any partition of 6, one can construct $Q$ and $S$ such that the multiplicities of points of $C$ give the same partition. There are 11 different partitions of 6, which give us 11 constellations of singularities on a cubic surface: $A_1$ (corresponds to $1+1+1+1+1+1$), $2A_1$ ($2+1+1+1+1$), $3A_1, 4A_1, A_2+A_1, A_2+2A_1, 2A_2+A_1, A_3+A_1, A_3+2A_1, A_4+A_1,$ and $A_5+A_1$. These are all of the possible combinations of singularities on a cubic surface containing an $A_1$ singularity.
\end{example}

\subsection{Deformations}

Consider the family $\bbP(V)$, $V=H^0(\bbP^n,\calO(d))$ of hypersurfaces of degree $d$ in $\bbP^n$. Let $X\in\bbP(V)$ be a hypersurface with only isolated singular points $p_1,\ldots,p_k$. For each $1\leq i\leq k$, we denote by $\Def(p_i)$ the deformation functor of the singularity of $X$ at $p_i$. The formal neighbourhood of $X$ in $\bbP(V)$ represents the functor that we denote by $\Def(X,\bbP(V))$. As every deformation of $X$ in $\bbP(V)$ induces a deformation of each of its singularities, we have a natural morphism of functors
$$
\Def(X,\bbP(V))\rightarrow\prod_{i=1}^{k}\Def(p_i).
$$
This global-to-local map is not always surjective, i.e. as we move $X$ in the space of degree $d$ hypersurfaces, we do not necessarily get all of the possible combinations of singularities that come from deforming each $p_i$ locally. Nevertheless, it is surjective in some cases by the following theorem of du Plessis and Wall:

\begin{theorem}[{du Plessis--Wall, \cite{DPW00}, Proposition 4.1}] \label{F:DPW}
	Given a complex hypersurface $X$ of degree $d$ in $\mathbb P^n$ with only isolated singularities, the family of hypersurfaces of degree $d$ induces a simultaneous versal deformation of all of the singularities of $X$, provided $\tau(X)<\delta(d)$, where $\tau(X)$ is the total Tjurina number of $X$ (see Definition \ref{Tjurina}), and $\delta(d)=16$, $18$ or $4(d-1)$, for   $d=3$, $4$ or $d\ge 5$, respectively.
\end{theorem}

\begin{remark}
	If $X$ is a cubic threefold with $\mu(X)\leq 15$, the family of cubic threefolds induces a simultaneous versal deformation of all of the singularities of $X$. It follows immediately from the theorem above and the fact that $\tau(X)\leq\mu(X)$ (see Remark \ref{TjM}).
\end{remark}

\begin{example} \label{Cubic_curves}
	It is not hard to check by a direct computation that the possible combinations of isolated singularities on cubic curves are $D_4$ (three concurrent lines), $A_3, A_2, 3A_1, 2A_1$, and $A_1$. By Theorem \ref{Gr}, $D_4$ is adjacent to the rest of the combinations. Theorem \ref{F:DPW} gives us a stronger statement that we can deform three concurrent lines in the family of cubic curves to a plane cubic with any of the listed combinations of singularities: since $\tau(D_4)=\mu(D_4)=4<16$, the global deformations of a cubic curve with a $D_4$ singularity surject onto the local deformations of $D_4$.
\end{example}

As we see in Example \ref{Cubic_curves} and Theorem \ref{BW}, all the information about combinations of singularities on cubic curves and surfaces can be obtained from their worst singularities which are $D_4$ and $\Esx$ respectively. We would like to get a similar description for cubic threefolds as well, especially since there are many more possible combinations of singularities in this case. Just as for curves and surfaces, the worst isolated singularity here is the cone over a smooth cubic hypersurface of one dimension lower:

\begin{proposition} \label{O16adj}
	For any combination of isolated singularities appearing on a cubic threefold, there exists a cubic threefold with an $O_{16}$ singularity adjacent to a cubic threefold with the given combination of singularities.
\end{proposition}

\begin{proof}
	Let $X$ be a cubic threefold with isolated singularities. Since a generic hyperplane section of $X$ is smooth, we can assume that the one given by $x_0=0$ is smooth and thus $X$ is determined by the equation
	$$
	F=cx_0^3+x_0^2f_1(x_1,x_2,x_3,x_4)+x_0f_2(x_1,x_2,x_3,x_4)+f_3(x_1,x_2,x_3,x_4).
	$$
	After scaling $x_0$ by $t$, we get
	$$
	F_t=ct^3x_0^3+t^2x_0^2f_1(x_1,x_2,x_3,x_4)+tx_0f_2(x_1,x_2,x_3,x_4)+f_3(x_1,x_2,x_3,x_4).
	$$
	If $t\neq 0$, then $X=V(F)\simeq V(F_t)$. When $t=0$, we get a cone over a smooth cubic surface $V(f_3)$.
\end{proof}

However, we encounter certain problems. First, currently there exists no complete classification of local deformations of $O_{16}$. Second, even if there was such a classification, we would not be able to apply Theorem~\ref{F:DPW} because the Milnor number of $O_{16}$ is 16 and thus the assumptions of the theorem are not necessarily satisfied. In the following sections, we find a short list of combinations of singularities that satisfy the assumptions of Theorem~\ref{F:DPW} and consists of well-studied singularities. The rest of the combinations can be obtained as local deformations of these maximal ones.

\begin{definition} \label{max_comb}
	We call a combination of singularities on a degree $d$ hypersurface \textit{maximal} in a certain class of combinations of singularities if we cannot get this combination of singularities by deforming a degree $d$ hypersurface with another combination of singularities from this class in $\bbP(V)$.
\end{definition}

\begin{example}
	The maximal combination of singularities on a cubic surface (among all of the combinations) is $\Esx$. It follows from \ref{BW} that the maximal $ADE$ combinations on cubic surfaces are $E_6$, $A_5+A_1$ and $3A_2$.
\end{example}

\subsection{The Milnor lattice of $O_{16}$}

While the results of this subsection are only used in the proof of Proposition \ref{A12} and can be replaced with a geometric argument there, we would still like to provide some lattice-theoretic background and examples since this kind of approach was effectively used in \cite{urabe}, \cite{Yang} and \cite{Stegmann} for studying classification problems similar to ours.

\medskip

Essentially, we would like to study singularities of cubic threefolds in terms of the Milnor lattice of~$O_{16}$:

\begin{proposition}[{\cite{LPZ}, Proposition 2.16}] \label{O16lattice}
	The stable Milnor lattice $T$ associated to a singularity of type $O_{16}$ is isometric to $D_4^{\oplus 3}\oplus U^{\oplus 2}$. In particular, the signature of $T$ is $(14,2)$
\end{proposition}

In general, there is a relation between the Milnor lattice of a singularity and the Milnor lattice of the singularities it is adjacent to:

\begin{theorem}[{\cite{loo}, Proposition 7.13}] \label{Loo}
	If an isolated singularity of type $L$ is adjacent to a combination of singularities $K_1+\ldots+K_m$, then there is a natural embedding of lattices $\iota: M_{K_1}\oplus\ldots\oplus M_{K_m}\hookrightarrow M_L$ where $M_{K_i}$ and $M_L$ are the Milnor lattices of the corresponding singularities.
\end{theorem}

\begin{corollary} \label{ADE_bound}
	A cubic threefold cannot have a combination of $ADE$ singularities with total Milnor number greater or equal to 15.
\end{corollary}

\begin{proof}
	Let $K_1,\ldots,K_m$ be the singularities with the sum of Milnor numbers equal to $\mu$. If $\mu=15$ or 16, then there is no embedding $\iota: M_{K_1}\oplus\ldots\oplus M_{K_m}\hookrightarrow T$ because $M_{K_1}\oplus\ldots\oplus M_{K_m}$ is positive definite, and $T$ has signature $(14,2)$ by Proposition \ref{O16lattice}. If $\mu>16$, then there is no embedding for dimension reasons. Thus the statement follows from Theorem \ref{Loo} above.
\end{proof}

\begin{example}
	Let $R$ be the $A_{12}$ lattice. In this case, we will show that there are no embeddings $R\hookrightarrow T$. If there is such an embedding $\iota$, then $\iota(A_{12})$ coincides with its saturation in $T$ by Corollary \ref{Sat(A_n)}. The discriminant group of $A_{12}$ is $\bbZ/13\bbZ$, the discriminant of $T$ is $(\bbZ/2\bbZ)^{\oplus 6}$ (see Example \ref{DiscrExamples}). Let $S=\iota(A_{12})^{\perp}$. Since $\bbZ/13\bbZ$ does not contain any copies of $\bbZ/2\bbZ$, the discriminant $D_S$ of $S$ should contain $(\bbZ/2\bbZ)^{\oplus 6}$ by Corollary \ref{EmbeddingDiscr}. However, $\dim S=4$ and thus $D_S$ cannot have more than 4 generators. Contradiction. Theorem \ref{Loo} now implies that $A_{12}$ singularities are not possible on cubic threefolds which agrees with the statement of Theorem \ref{Theorem1}. 
\end{example}

\begin{example}
	Let $R$ be the $D_4^{\oplus 3}\oplus A_1$ lattice. It can clearly be embedded into $T=D_4^{\oplus 3}\oplus U^{\oplus 2}$, but a $3D_4+A_1$ combination cannot appear on a cubic threefold by Theorem~\ref{Theorem1}. Thus we see that the converse of Theorem \ref{Loo} does not hold in this setting. We believe that the relation between possible embeddings of lattices into $T$ and singularities of cubic threefolds can be further explored using the results of \cite{LPZ} which link $T$ to the study of cubic fourfolds with Eckardt points.
\end{example}

\section{Singularities of corank 3} \label{sect_cr3}

Let $X\subset\bbP^4$ be a cubic threefold with a singular point $p=[1:0:0:0:0]$ of corank 3. Let $Q$ and $S$ be as in Section \ref{notation}. In this case, $Q\subset\bbP^3$ is a double plane and can be defined by the equation $x_1^2=0$. Let $C$ be the intersection of $S$ and the plane $x_1=0$.

\medskip

We will use the geometry of $C$ to describe possible singularities on $X$. The plane cubic $C$ is one of the following: three concurrent lines, a conic and its tangent, a triangle, a cuspidal cubic, a conic and its secant, a nodal cubic or a smooth cubic. Notice that $C$ cannot contain a double line, because otherwise $Sing(Q)\cap Sing(S)=Q\cap Sing(S)\neq\O$ and $p$ is not isolated by \ref{projthm}.

\begin{proposition}
	If $C$ is the union of three concurrent lines, then $p$ is a $U_{12}$ singularity.
\end{proposition}

\begin{proof}
	In the affine chart given by $x_0\neq 0$, $X$ is defined by the equation $x_1^2+f_3(x_1,x_2,x_3,x_4)=0$. Let $h_3(x_2,x_3,x_4)=f_3(0,x_2,x_3,x_4)$. Since $C$ is the union of three concurrent lines, we can choose coordinates in which $h_3(x_2,x_3,x_4)=x_2^3+x_2x_3^2$. Then 
	$$
	f_3(x_1,x_2,x_3,x_4)=x_2^3+x_2x_3^2+x_1h_2(x_2,x_3,x_4)+x_1^2h_1(x_2,x_3,x_4)+cx_1^3
	$$
	where $h_1$ and $h_2$ are homogeneous polynomials of degree 1 and 2 respectively and $c\in\bbC$. We have
	$$
	df_2=2x_1dx_1,
	$$
	$$
	df_3=(3x_2^2+x_3^2)dx_2+2x_2x_3dx_3+h_2dx_1+x_1dh_2+2x_1h_1dx_1+x_1^2dh_1+3cx_1^2dx_1.
	$$
	If $df_2(x_1,x_2,x_3,x_4)=0$, then $x_1=0$; if $df_3(0,x_2,x_3,x_4)=0$, then $x_2=x_3=0$. We define $b\in\bbC$ as follows:
	$$
	df_3(0,0,0,x_4)=h_2(0,0,x_4)dx_1=bx_4^2dx_1.
	$$
	In particular, $b\neq 0$, otherwise $[0:0:0:1]$ is a singular point on both $Q$ and $S$ and $p$ is not isolated by \ref{projthm}.
	
	\medskip
	
	We will use theorems from \cite{Arnold2012} to determine the type of singularity at $p$. In order to be able to apply them, we will do a sequence of coordinate changes. First, we set
	$$
	x_1=y_1-\frac{1}{2}h_2-\frac{1}{2}y_1h_1
	$$
	which gives
	$$
	f_2+f_3=y_1^2+x_2^3+x_2x_3^2-\frac{1}{4}h_2^2-y_1h_1h_2-\frac{3}{4}y_1^2h_1^2+O_5(x_1,x_2,x_3,x_4)
	$$
	where $O_k$ stands for any series with terms of degree $k$ and higher. Next, we can find a series $\psi_3$ with terms of degree 3 and higher so that if $y_1=z_1+\psi_3$, then
	$$
	f_2+f_3=z_1^2+x_2^3+x_2x_3^2-\frac{1}{4}h_2^2+O_5(x_2,x_3,x_4).
	$$
	Finally, we set $x_2=z_2$, $x_3=z_3$, $x_4=\frac{2\sqrt{-1}}{b}z_4$ and get
	$$
	f_2+f_3=z_1^2+z_2^3+z_2z_3^2+z_4^4+z_4^3g_1(z_2,z_3)+z_4^2g_2(z_2,z_3)+z_4g_3(z_2,z_3)+g_4(z_2,z_3)+O_5(z_2,z_3,z_4)
	$$
	where each $g_i$ is a homogeneous polynomial of degree $i$.
	
	\medskip
	
	Now we will use the determinator of singularities from Chapter 16 of \cite{Arnold2012}. The third jet of $f_2+f_3$ equals $z_2^3+z_2z_3^2$ which leads us to case $83_1$ of the determinator. The quasijet $j^*_{z_4^4}(f_2+f_3)$ equals $z_2^3+z_2z_3^2+z_4^4$ where $j^*_{z_4^4}(z_2^{m_2}z_3^{m_3}z_4^{m_4})$ is defined to be $z_2^{m_2}z_3^{m_3}z_4^{m_4}$ if $\frac{m_2}{3}+\frac{m_3}{3}+\frac{m_4}{4}\leq 1$ and 0 otherwise. Thus we are directed to case $84_1$ which means that $p$ is a $U_{12}$ singularity.
\end{proof}

\begin{proposition}
	A singular point $p\in X$ as above has one of the following types: $U_{12}$, $S_{11}$, $Q_{10}$, $T_{444}$, $T_{344}$, $T_{334}$ or $P_8$. Furthermore $p$ is the only singularity of $X$.
\end{proposition}

\renewcommand{\arraystretch}{1.25}

\begin{table}[h]
	\begin{center}
		\begin{tabular}{| c | c | c | c |} \hline
			$\mu $&Singularity at $p$& Geometry of $C$& Singularities of $C$ \\ \hline \hline
			12 & $U_{12}$ & three concurrent lines & $D_4$ \\ \hline
			11 & $S_{11}$ & conic and tangent & $A_3$ \\ \hline
			11 & $T_{444}$ & triangle & $3A_1$ \\ \hline
			10 & $Q_{10}$ & cuspidal cubic & $A_2$ \\ \hline
			10 & $T_{344}$ &conic and secant& $2A_1$ \\ \hline
			9 & $T_{334}$ & nodal cubic & $A_1$ \\ \hline
			8 & $P_8=T_{333}$ & smooth cubic & - \\ \hline
		\end{tabular}
		\bigskip
		\caption{Singularities of corank 3}
		\label{crk3}
		\vspace*{-1em}
	\end{center}
\end{table}

\renewcommand{\arraystretch}{1}

\begin{proof}
	On the one hand, a $U_{12}$ singularity can be deformed to $S_{11}$, $Q_{10}$, $T_{444}$, $T_{344}$, $T_{334}$ and $P_8$ (see Table \ref{unimodal}), and thus these singularities can appear on a cubic threefold by Theorem \ref{F:DPW}. On the other hand, there are 7 possible diffeomorphism types of $C$ which means that there are at most 7 possible singularity types of $p$ by Theorem \ref{projthm}. Thus $U_{12}$, $S_{11}$, $Q_{10}$, $T_{444}$, $T_{344}$, $T_{334}$ and $P_8$ are all of the possibilities.
	
	\smallskip
	
	The correspondence between singularity types of $p$ and diffeomorphism types of $C$ is shown on Table \ref{crk3}. It can be established by comparing the adjacencies from Table \ref{unimodal} and the way different types of plane cubics deform to each other.
	
	\smallskip
	
	By Theorem \ref{projthm}, singularities of $X$ other than $p$ correspond to singularities of $Q\cap S$ contained in the smooth locus of $Q$ which is empty in the corank 3 case.
\end{proof}

\section{Singularities of corank 2} \label{sect_cr2}

Let $X\subset\bbP^4$ be a cubic threefold with a singular point $p\in X$ of corank 2. Let $Q$, $S$, $C$ be as in Section \ref{notation}. Since $p$ is of corank 2, $Q=P_1\cup P_2$ where $P_i$ are distinct planes. Let $L=Sing(Q)=P_1\cap P_2$.  By Theorem~\ref{projthm}, the singularities of $X$ away from $p$ correspond to the singularities of $C$ away from $L$. The singularity type of $p$ depends on the singularities of $C$ along $L$.

\begin{claim} \label{crk2jet}
	The curve $C$ contains $L$ if and only if the third (stable) jet of $f_2+f_3$ vanishes.
\end{claim}

\begin{proof}
	We can choose coordinates such that $Q$ is defined by the equation $x_1x_2=0$. Consider the third stable jet of $x_1x_2+f_3$, i.e. the third jet of a singularity in $\bbC^2$ stably equivalent to the one given by $x_1x_2+f_3$. It vanishes if and only if $f_3=x_1g_2(x_1,x_2,x_3,x_4)+x_2h_2(x_1,x_2,x_3,x_4)$ where $g_2$ are $h_2$ homogeneous polynomials of degree 2. The polynomial $f_3$ is of the form $x_1g_2+x_2h_2$ if and only if $L\subset C$.
\end{proof}

\subsection{Singularities with vanishing third jet}

By Claim \ref{crk2jet}, if the third jet of $f_2+f_3$ vanishes then $C=L\cup C_1\cup C_2$ where $C_1\subset P_1$ and $C_2\subset P_2$ are plane conics. We can choose coordinates in which $P_1$ and $P_2$ are defined by $x_1=0$ and $x_2=0$ respectively.

\begin{claim} \label{crk2doubleline}
	\hfill
	\begin{itemize}
		\item[a)] If $C_1$ or $C_2$ is a double line, then $X$ has a non-isolated singularity.
		\item[b)] If $C_1\cap C_2\neq\O$, then $X$ has a non-isolated singularity.
	\end{itemize}
\end{claim}

\begin{proof}
	\hfill
	\begin{itemize}
		\item[a)] Assume $C_1$ is a double line defined by $x_3^2=0$. Then $f_3=x_2x_3^2+x_1g_2(x_1,x_2,x_3,x_4)$ where $g_2$ is a homogeneous polynomial of degree 2. In the affine chart given by $x_0\neq 0$, $X$ is defined by the equation $f_2+f_3=x_1x_2+x_2x_3^2+x_1g_2=0$. We have
		$$
		d(f_2+f_3)=(x_2+g_2)dx_1+(x_1+x_3^2)dx_2+2x_2x_3dx_3+x_1dg_2.
		$$
		Both $f_2+f_3$ and $d(f_2+f_3)$ vanish along the curve $K\subset X\subset\bbP^4$ defined by $x_1=0$, $x_3=0$, $x_0x_2+g_2=0$. Thus the singular point $p\in K\subset X$ is not isolated.
		
		\smallskip
		
		\item[b)] Assume there exists a point $q\in C_1\cap C_2\subset L$. Then $f_3=x_1g_2(x_1,x_2,x_3,x_4)+x_2h_2(x_1,x_2,x_3,x_4)$ where $g_2$ and $h_2$ are homogeneous polynomials of degree 2 such that $g_2(q)=h_2(q)=0$. The differential $df_3=g_2dx_1+h_2dx_2+x_1dg_2+x_2dh_2$ vanishes at $q$ which means that $q\in Sing(S)$. Thus $q\in Sing(Q)\cap Sing(S)$ and the singularity at $p$ is not isolated by \ref{projthm}. \qedhere
	\end{itemize}
\end{proof}

There are 10 possible geometric configurations of $C$. Typical pictures are shown in Figure \ref{F:X9}. We will see that these pictures correspond to combinations of type $X_9+2A_1$ (on the left) and $X_9$ (on the right).

\medskip

\begin{figure}
	\begin{center}
		\includestandalone{X9}
		\vspace*{-1em}
		\caption{The curve $C$ for cubic threefolds with $X_9+2A_1$ and $X_9$ singularities}
		\label{F:X9}
	\end{center}
\end{figure}     

All of the configurations of $C$ together with the corresponding combinations of singularities are given schematically in Figure \ref{crk2j0}. This correspondence is going to be explained in Proposition \ref{crk2j0final}. Two configurations in Figure \ref{crk2j0} are connected by an arrow if one be can deformed to the other by a small perturbation.

\medskip

\begin{figure}[htb]
	\begin{center}
		\scalebox{.25}{\includegraphics{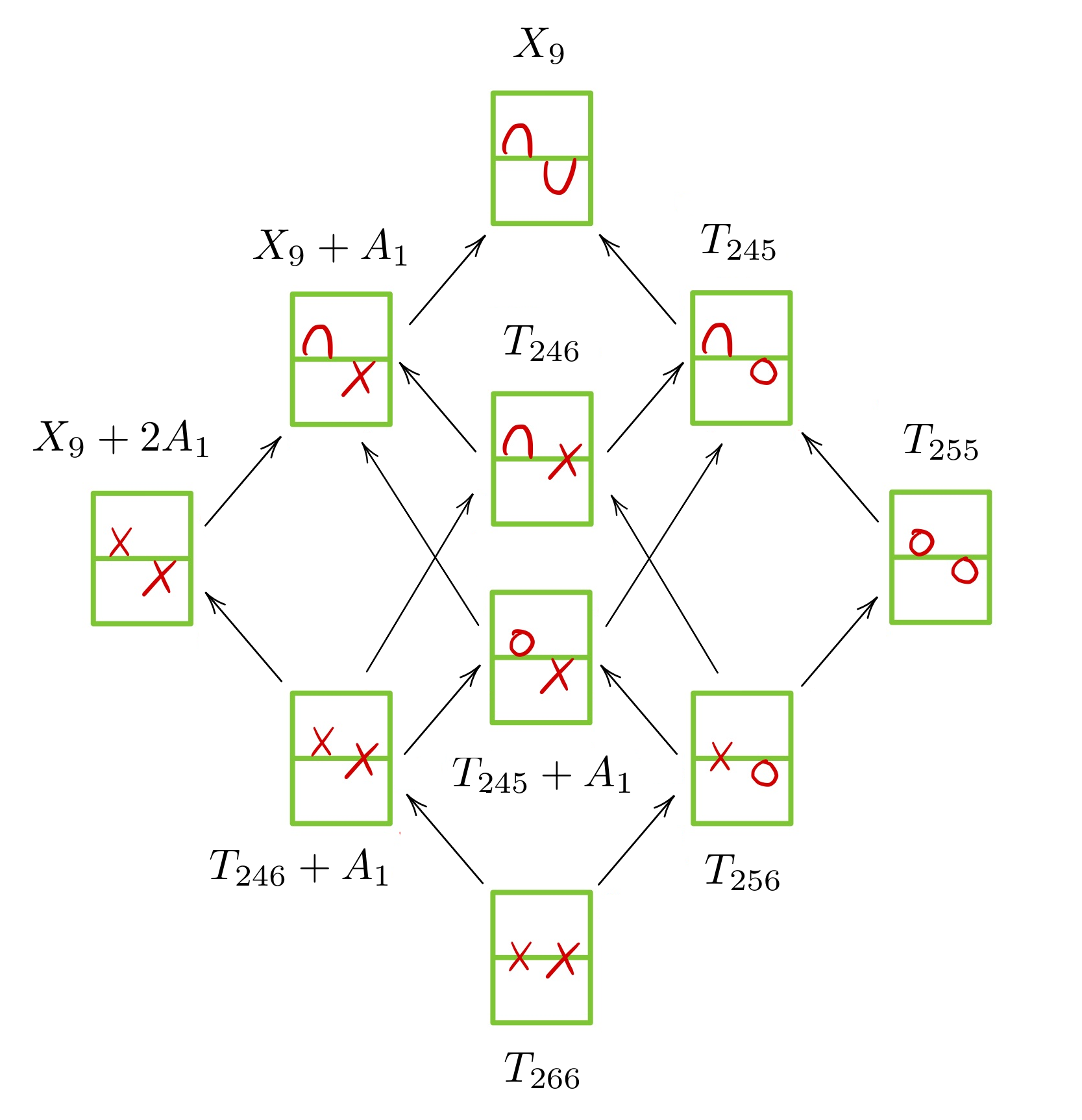}}
	\end{center}
	\caption{Adjancencies of corank 2 singularities with vanishing third jet}
	\label{crk2j0}
\end{figure}

We will first consider the case where each $C_i$ is a pair of distinct lines with the intersection point on $L$. We can choose coordinates such that $f_2=x_1x_2$, $f_3(0,x_2,x_3,x_4)=ax_2x_3(x_2-x_3)$ and $f_3(x_1,0,x_3,x_4)=bx_1x_4(x_1-x_4)$ where $a,b\in\bbC$, $ab\neq 0$. In these coordinates $f_3=ax_2x_3(x_2-x_3)+bx_1x_4(x_1-x_4)+x_1x_2g_1(x_1,x_2,x_3,x_4)$ where $g_1$ is a homogeneous polynomial of degree 1. We can get rid of the $x_1x_2g_1$ term by a linear coordinate change in $\bbP^5$  (see Remark \ref{SmodQ}). After another linear change we can assume that $a=b$. Thus we have $f_3=ax_2x_3(x_2-x_3)+ax_1x_4(x_1-x_4)$. In the affine chart given by $x_0\neq 0$, $X$ is defined by $x_1x_2+a(x_2^2x_3-x_2x_3^2+x_1^2x_4-x_1x_4^2)=0$. 

\begin{proposition} \label{t266}
	If $f_2+f_3=x_1x_2+a(x_2^2x_3-x_2x_3^2+x_1^2x_4-x_1x_4^2)$, then the singularity at $p$ is of type $T_{266}$. Moreover $p$ is the unique singularity of $X$.
\end{proposition}

\begin{proof}
	We will do a sequence of coordinate changes. First we will set $x_1=y_1-ay_2x_3+ax_3^2$, $x_2=y_2-ay_1x_4+ax_4^2$.
	After some term cancellation the function $f_2+f_3$ has the form
	$$
	y_1y_2-a^2x_3^2x_4^2+a^3(x_3x_4^4+x_3^4x_4)+2a^2(y_1x_3^2x_4+y_2x_3x_4^2)-2a^3(y_1x_3x_4^3+y_2x_3^3x_4)-3y_1y_2x_3x_4+a^3(y_1^2x_3x_4^2+y_2^2x_3^2x_4).
	$$
	The next coordinate change is as follows:
	$$
	y_1=z_1-2a^2x_3x_4^2+2a^3x_3^3x_4+\frac{3}{2}z_1x_3x_4-a^3z_2x_3^2x_4,
	$$
	$$
	y_2=z_2-2a^2x_3^2x_4+2a^3x_3x_4^3+\frac{3}{2}z_2x_3x_4-a^3z_1x_3x_4^2;
	$$
	$$
	f_2+f_3=z_1z_2-a^2x_3^2x_4^2+a^3(x_3x_4^4+x_3^4x_4)-8a^4x_3^3x_4^3+z_1O_5+z_2O_5+O_7
	$$
	
	\smallskip
	
	\noindent where $O_k$ stands for any series with terms of degree $k$ and higher. Now we set
	$$
	x_3=y_3+\frac{a}{2}y_4^2,\;x_4=y_4+\frac{a}{2}y_3^2.
	$$
	We get
	$$
	f_2+f_3=z_1z_2-a^2(y_3^2y_4^2+\frac{a^2}{4}y_3^6+\frac{a^2}{4}y_4^6)+y_3^2y_4O_3+y_3y_4^2O_3+z_1O_5+z_2O_5+O_7.
	$$
	We can make a final coordinate change such that
	$$
	f_2+f_3=z_1z_2+bz_3^2z_4^2+z_3^6+z_4^6+O(7),\;b\neq 0
	$$
	which is a local description of a $T_{266}$ singularity (see case $16$ in Chapter 16 of \cite{Arnold2012}).
	
	\medskip
	
	By Theorem \ref{projthm}, singularities of $X$ other than $p$ correspond to singularities of $C$ contained in the smooth locus of $Q$. In our case $C$ is smooth outside of $Sing(Q)=L$.
\end{proof}

\begin{corollary} \label{t2pq}
	Any $T_{2pq}$ singularity with $4\leq p,q\leq 6$ can appear on a cubic threefold.
\end{corollary}

\begin{proof}
	A $T_{266}$ singularity is possible by Proposition \ref{t266} and is adjacent to $T_{2pq}$ singularities with $p,q\leq 6$ (see Table \ref{unimodal}). Thus a $T_{2pq}$ with $4\leq p,q\leq 6$ singularity can appear on a cubic threefold by Theorem \ref{F:DPW}.
\end{proof}

\begin{proposition} \label{crk2j0final}
	The possible combinations of singularities on a cubic threefold $X$ containing a singular point $p$ of corank 2 with vanishing third jet are $T_{266}$, $T_{256}$, $T_{246}+A_1$, $T_{245}+A_1$, $T_{255}$, $X_9+2A_1=T_{244}+2A_1$, $T_{246}$, $T_{245}$, $X_9+A_1$ and $X_9$. The adjacencies of these combinations are shown in Figure \ref{crk2j0}.
\end{proposition}

\begin{proof}
	By Theorem \ref{projthm}, the geometry of $C$ determines the type of singularity at $p$. Thus we need to assign a combination of singularities to each configuration in Figure \ref{crk2j0}. The singularity at $p$ is unique if and only if the corresponding curve $C$ does not contain a pair of distinct lines with the intersection point outside of $L$, otherwise we get additional $A_1$ singularities. By Corollary \ref{t2pq}, the six configurations without such pairs of lines should correspond to $T_{2pq}$ with $4\leq p,q\leq 6$. Taking into account the fact that $T_{246}+A_1$ can appear on a cubic threefold by \ref{unipotent}, there is only one way to assign combinations of singularities to the configurations of $C$ in Figure \ref{crk2j0}.
\end{proof}

\subsection{Singularities with nonvanishing third jet} \label{ADEsection}

By Claim \ref{crk2jet}, if $p\in X$ has a nonvanishing third jet then $L\not\subset C$, $C=C_1\cup C_2$ where $C_1\subset P_1$ and $C_2\subset P_2$ are plane cubics. We can choose coordinates in which $P_1$ and $P_2$ are defined by $x_1=0$ and $x_2=0$ respectively.

\begin{claim}
	If $C_1$ or $C_2$ contains a double line, then $X$ has a non-isolated singularity.
\end{claim}

\begin{proof}
	Analogous to Claim \ref{crk2doubleline}.
\end{proof}

\begin{claim}
	The intersection $C\cap L$ is either a triple point, a double point and a simple point or three simple points. In these cases, the third jet of $f_2+f_3$ is equal to $x_3^3$, $x_3^2x_4$ or $x_3^2x_4+x_4^3$ respectively.
\end{claim}

\begin{proof}
	Follows from case 3 in Chapter 16 of \cite{Arnold2012}.
\end{proof}

We will consider these three cases separately starting with the triple point case.

\begin{claim} \label{crk2c1}
	Let $C\cap L$ be a triple point $q$. We can assume that $C_2$ is non-singular at $q$ and $C_1$ has $D_4$, $A_2$, $A_1$ or no singularity at $q$.
\end{claim}

\begin{proof}
	A cubic curve can only have $D_4$, $A_3$, $A_2$ and $A_1$ singularities. If $q\in C_1$ is an $A_3$ singularity then $C_1$ is a union of a conic and its tangent line. Since $q$ is a triple point, this tangent line should coincide with $L$ but $L\not\subset C_1$.
	
	\smallskip
	
	If $C_1$ and $C_2$ are both singular at $q$ then $S$ is singular at $q$ and $p$ is not isolated by \ref{projthm}.
\end{proof}

\begin{claim} \label{plcubicsing}
	Let $C\cap L$ be a triple point. If $C_i$ is smooth at $q$ then it can only have one $A_2$ or $A_1$ singularity away from $q$. If $C_i$ has a $D_4$ or $A_2$ singularity at $q$ then it does not have any other singularities. If $C_i$ has an $A_1$ singularity at $q$ then it can have another $A_1$ singularity.
\end{claim}

\begin{proof}
	If $C_i$ is smooth at $q$ then it should be irreducible because otherwise it will have more than one intersection point with $L$. An irreducible cubic curve can have at most one $A_2$ or $A_1$ singularity. If $C_i$ has an $A_1$ singularity at $q$ and two additional $A_1$ singularities then it is a union of three lines and thus intersects $L$ at two points.
\end{proof}

\begin{claim} \label{crk2p}
	Singularities of types $J_{10}$, $E_8$, $E_7$ and $E_6$ can appear on a cubic threefold.
\end{claim}

\begin{proof}
	A $T_{246}$ singularity is possible by Corollary \ref{t2pq} and is adjacent to $J_{10}$ (see Table \ref{unimodal}). A $J_{10}$ singularity is in turn adjacent to $E_8$, $E_7$ and $E_6$ singularities. The claim now follows from Theorem \ref{F:DPW}.
\end{proof}

\begin{proposition}
	All of the possible singularities on $X$ with the third jet equal to $x_3^3$ are $J_{10}$, $E_8$, $E_7$ and $E_6$. The maximal configurations among the ones containing these singularities are $J_{10}+A_2$, $E_7+A_2+A_1$ and $E_6+2A_2$.
\end{proposition}

\begin{proof}
	By Theorem \ref{projthm}, the singularity type of $p$ depends on the geometry of $C_1$ and $C_2$. To prove the proposition, we can combine claims \ref{crk2c1} and \ref{crk2p}: if $C_2$ is smooth then there are 4 possibilities for the geometry of $C_1$ at $q$, and we already know 4 singularity types with the third jet equal to $x_3^3$ that can appear on a cubic threefold.
	
	\smallskip
	
	The second statement of the proposition follows from Claim \ref{plcubicsing} and Theorem \ref{F:DPW}.
\end{proof}

\renewcommand{\arraystretch}{1.25}

\begin{table}[h]
	\begin{center}
		\begin{tabular}{|c | c | c | c|}\hline
			$\mu$&Singularity at $p$& Singularity of $C_1$ along $L$& Other singularities of $C_1$ \\ \hline \hline
			$10$&$J_{10}$&$D_4$& - \\ \hline
			$8$ &$E_{8}$ &$A_2$& - \\ \hline
			$7$ &$E_{7}$ &$A_1$& $A_1$ or none \\ \hline
			$6$ &$E_{6}$ & -   & $A_2$, $A_1$ or none \\ \hline\hline
			$8$ &$D_{8}$ &$A_3$& - \\ \hline
			$7$ &$D_{7}$ &$A_2$& - \\ \hline
			$6$ &$D_{6}$ &$A_1$& $2A_1$, $A_1$ or none \\ \hline
			$5$ &$D_{5}$ & -   & $A_3$, $A_2$, $2A_1$, $A_1$ or none \\ \hline\hline
			$4$ &$D_{4}$ & -   & $D_4$, $A_3$, $A_2$, $3A_1$, $2A_1$, $A_1$ or none \\ \hline
		\end{tabular}
		\bigskip
		\caption{Singularities of corank 2 with nonvanishing first jet}
		\label{crk2jnot0}
		\vspace*{-1em}
	\end{center}
\end{table}

\renewcommand{\arraystretch}{1}

\begin{claim} \label{crk2c1'}
	Let $C\cap L$ be a simple point $q_1$ and a double point $q_2$. Then both $C_1$ and $C_2$ are smooth at $q_1$. We can assume that $C_2$ is non-singular at $q_2$ and $C_1$ has an $A_3$, $A_2$, $A_1$ or no singularity at $q_2$.
\end{claim}

\begin{proof}
	Analogous to Claim \ref{crk2c1}.
\end{proof}

\begin{claim} \label{plcubicsing'}
	Let $C\cap L$ be a simple point $q_1$ and a double point $q_2$. If $C_i$ is smooth at $q_2$ then it can have $A_3$, $A_2$, $2A_1$ or $A_1$ combinations of singularities away from $q_2$. If $C_i$ has an $A_3$ or $A_2$ singularity at $q_2$ then it does not have any other singularities. If $C_i$ has an $A_1$ singularity at $q_2$ then it can have two more $A_1$ singularities.
\end{claim}

\begin{proof}
	Analogous to Claim \ref{plcubicsing}.
\end{proof}

\begin{proposition} \label{Dn_max}
	All of the possible singularities on $X$ with the third jet equal to $x_3^2x_4$ are $D_k$, $5\leq k\leq 8$. The maximal cases among the ones containing such $D_k$ are $D_8+A_3$, $D_6+A_3+2A_1$ and $D_5+2A_3$.
\end{proposition}

\begin{proof}
	By case 5 in Chapter 16 of \cite{Arnold2012}, if the third jet of $p$ is equal to $x_3^2x_4+x_4^3$ then $p$ is a $D_k$ singularity with $k\geq 5$.
	
	\smallskip
	
	By Theorem \ref{projthm}, the singularity type of $p$ depends on the geometry of $C_1$ and $C_2$. Claim \ref{crk2c1'} says that if $C_2$ is smooth then there are 4 possibilities for the geometry of $C_1$ at $q_2$. It now follows from Theorem \ref{F:DPW} that $p$ is of type $D_k$ such that $5\leq k\leq 8$.
	
	\smallskip
	
	The second statement of the proposition follows from Claim \ref{plcubicsing'} and Theorem \ref{F:DPW}.
\end{proof}

By case 4 in Chapter 16 of \cite{Arnold2012}, if the third jet of $p$ is equal to $x_3^2x_4+x_4^3$, then $p$ is a $D_4$ singularity.

\begin{proposition} \label{3D4sing}
	The maximal configuration of singularities on $X$ among the ones containing $D_4$ is $3D_4$.
\end{proposition}

\begin{proof}
	If $p$ is a $D_4$ singularity, then $C\cap L$ is three simple points and both $C_1$ and $C_2$ are smooth at these points. Each $C_i$ can have any of the possible plane cubic combinations of singularities: $D_4$, $A_3$, $A_2$, $3A_1$, $2A_1$, $A_1$ or no singularities. By Theorem \ref{F:DPW}, $3D_4$ is the maximal combination containing $D_4$. The curve $C$ for a cubic threefold with $3D_4$ singularities is shown in Figure \ref{F:3D4}.
\end{proof}

\begin{figure}
	\begin{center}
		\includestandalone{D4curve}
		\bigskip
		\caption{The curve $C$ for cubic threefolds with $3D_4$ singularities}
		\label{F:3D4}
	\end{center}
\end{figure}     

The list of all of the possible combinations of singularities of corank 2 with nonvanishing third jet is given in Table \ref{crk2jnot0}.

\section{Constellations of $A_n$ singularities} \label{sect_an}

Let $X\subset\bbP^4$ be a cubic threefold with $A_n$ singularities only. Let $Q$, $S$, $C$ be as in Section \ref{notation}. We have $\deg C=6$ and $g_a(C)=4$ since $C$ is a complete intersection curve of multidegree $(2,3)$.

\medskip

If $X$ has only $A_1$ singularities, then it is classically known that it contains at most 10 of them. The Segre cubic (\cite{Segre}) has exactly 10 nodes, and thus any number less than 10 is also possible by Theorem \ref{F:DPW}. 

\begin{proposition} \label{a1max}
	The maximal configuration of $A_1$ singularities on a cubic threefold is $10A_1$.  
\end{proposition}

In this section, we will describe possible constellations of $A_n$ singularities on $X$ using the projection method introduced in Section \ref{projection}. We will be projecting from the worst singular point $p\in X$. From now on, we will assume that $p$ is of $A_n$ type with $n>1$. By Claim \ref{crk1}, this is equivalent to $Q\subset\bbP^3$ being of corank 1, which means that after a coordinate change $Q$ can be defined by the equation $x_1^2+x_2^2+x_3^2=0$. The blow-up $\Qb$ of $Q$ at $v=[0:0:0:1]$ is thus the Hirzebruch surface $\bbF_2$. Denote by $\Cb\in\bbF_2$ the strict transform of $C\in Q$.

\begin{lemma} [{\cite{beauville}, Proposition IV.1}]
	The Picard group $\Pic\bbF_2$ is isomorphic to $\bbZ\sigma\oplus\bbZ f$ where $\sigma$ is the class of the unique irreducible curve with negative self-intersection and $f$ is the class of a fibre. We have $\sigma^2=-2$, $\sigma.f=1$, $f^2=0$, $K_{\bbF_2}=-2\sigma-4f$.
\end{lemma}

\begin{remark} \label{H}
	The class $H\in\Pic\bbF_2$ of the preimage of a hyperplane section of $Q$ not intersecting $v$ is equal to $\sigma+2f$.
\end{remark}

\begin{lemma} \label{classofc}
	If $p$ is of $A_n$ type with $n>2$, then $C$ passes through $v$, $g_a(\Cb)=3$, and $[\Cb]=2\sigma+6f$.
\end{lemma}

\begin{proof}
	By Theorem \ref{projthm} and Claim \ref{Anblowup}, $v$ is a singularity of $A_{n-2}$ type on $C$, and in particular $C$ passes through $v$ since $n-2>0$. After blowing up, we get a singularity of $A_{n-4}$ type on $\Cb$, and thus $g_a(\Cb)=g_a(C)-1=3$ by Proposition \ref{gdrop}.
	
	\smallskip
	
	Let $[\Cb]=a\sigma+bf$ in $\Pic\bbF_2$. Since $\deg C=6$, $\Cb.H=6$ and $b=6$ where $H$ is as in Remark \ref{H}.
	
	\smallskip
	
	By the genus formula, $g_a(\Cb)=\frac{1}{2}\Cb.(\Cb+K_{\bbF_2})+1=\frac{1}{2}(a\sigma+6f).(a\sigma+6f-2\sigma-4f)+1=-a^2+6a-5$. Solving $-a^2+6a-5=3$, we get $a=2$ or $a=4$. If $a=4$ then $\Cb.\sigma=-2<0$ which is impossible because $\Cb$ is the strict transform of $C$ and $\sigma$ is the class of the exceptional divisor of the blow-up. Thus $a=2$ and $[\Cb]=2\sigma+6f$.
\end{proof}

\begin{lemma} \label{a2}
	If $p$ is of $A_2$ type, then $C$ does not pass through $v$, $g_a(\Cb)=4$, and $[\Cb]=3\sigma+6f$.
\end{lemma}

\begin{proof}
	Assume that $C$ passes through $v$. Then $C$ is singular at $v$ and the singularity of $X$ at $p$ should be worse than $A_2$ by Theorem \ref{projthm}. The rest of the proof is analogous to Lemma \ref{classofc}.
\end{proof}

In the two propositions below, we find the maximal combinations in the case when $p$ is of $A_n$ type with $n\geq 3$. In the first one we assume that $n$ is even which implies that the singularity at $v$ on $C$ comes from one irreducible component of $C$ (unibranched case). In the second one we assume that $n$ is odd, and thus the singularity at $v$ can come either from one component of $C$ or from the intersection of two components. 

\begin{proposition} \label{A12}
	\hfill
	\begin{itemize}
		\item[(12)] A cubic threefold cannot have singularities of $A_n$ type with $n\geq 12$.
		\item[(10)] If $p$ is of $A_{10}$ type, then it is the only singularity on $X$.
		\item[(8)] If $p$ is of $A_8$ type, then the corresponding maximal configuration on $X$ is $A_8+A_2$.
		\item[(6)] If $p$ is of $A_6$ type, then the corresponding maximal configuration on $X$ is $A_6+A_4$.
		\item[(4)] If $p$ is of $A_4$ type, then the corresponding maximal configurations on $X$ are $2A_4+A_2$ and $A_4+2A_3$.
	\end{itemize}
\end{proposition}

\begin{proof}
	By Corollary \ref{ADE_bound}, a cubic threefold cannot have an $A_n$ singularity with $n\geq 15$. If it has an $A_n$ singularity with $15>n>12$, then it deforms to a cubic threefold with an $A_{12}$ singularity by Theorems \ref{Gr} and \ref{F:DPW}. Now assume $X$ has an $A_{12}$ singularity. By Theorem \ref{projthm}, an $A_{12}$ singularity at $p$ on $X$ gives an $A_{10}$ singularity at $v$ on $C$. Denote the irreducible component of $C$ containing $v$ by $A$. By Proposition \ref{gdrop}, the arithmetic genus of $A$ should be at least 5. If $A$ is the only irreducible component of $C$, then we get a contradiction since $g_a(C)=4$. If $C=A\cup B$, then we have $g_a(C)=g_a(A)+g_a(B)+A.B-1$. Consider the strict transforms $\Ab$ and $\Bb$ of $A$ and $B$. By Lemma \ref{classofc}, $\Cb.\sigma=(2\sigma+6f).\sigma=2$ and we get $\Ab.\sigma=2$, $\Bb.\sigma=0$ since $v\notin B$. There are two possibilities: either $[\Ab]=2f$ and $[\Bb]=2\sigma+4f$, or $[\Ab]=\sigma+4f$ and $[\Bb]=\sigma+2f$. In the first case, $\Ab$ is reducible. In the second case, $g_a(\Ab)=0$ and $g_a(A)=g_a(\Ab)+1<5$. Contradiction. Notice that instead of using Corollary \ref{ADE_bound} one can generalize the geometric argument for $A_{12}$ to $A_n$ with $n>12$.
		
	\smallskip
	
	We can use the same reasoning as above to show that $C$ cannot be reducible in parts (10), (8), and (6) of the proposition. By Theorem \ref{projthm} and Proposition \ref{gdrop}, we get that the corresponding maximal configurations on $X$ are $A_{10}$, $A_8+A_2$, $A_6+A_4$ and $A_6+2A_2$. However, a constellation of $A_6+2A_2$ type is not possible on a cubic threefold because it deforms to $A_3+3A_2$ by Theorems \ref{Gr} and \ref{F:DPW}, and $A_3+3A_2$ is not possible by Proposition \ref{a33a2}. The rest of the combinations appear on cubic threefold since $A_{10}$, $A_8+A_2$ and $A_6+A_4$ are deformations of $A_{11}$ which is possible by Example \ref{A11}.
	
	\smallskip
	
	Similarly, if $p$ is of $A_4$ type and $C$ is irreducible, we get that the corresponding maximal configuration on $X$ is $2A_4+A_2$. Combinations of type $A_7+A_4$ which appear on cubic threefolds by Theorem \ref{1-sym-ss} can be deformed to $2A_4+A_2$.
	
	\smallskip
	
	If $p$ is of $A_4$ type and $C$ is reducible, we have $[\Ab]=\sigma+4f$, $[\Bb]=\sigma+2f$, $g_a(A)=g_a(\Ab)+1=1$, $g_a(B)=1$. In this case, both $A$ and $B$ are irreducible, $A$ has one $A_2$ singularity and $B$ can possibly have an $A_2$ or an $A_1$ singularity by Proposition \ref{gdrop}. Since $A.B=(\sigma+4f).(\sigma+2f)=4$, we can get $2A_3$, $A_3+2A_1$ or $4A_1$ singularities by intersecting $A$ and $B$. The corresponding maximal case on $X$ is $A_4+2A_3$ which is also a deformation of $A_7+A_4$.
\end{proof}

\begin{proposition} \label{A_odd}
	\hfill
	\begin{itemize}
		\item[(11)] If $p$ is of $A_{11}$ type then it is the only singularity on $X$.
		\item[(9)] If $p$ is of $A_{9}$ type then the corresponding maximal configuration on $X$ is $A_9+A_1$.
		\item[(7)] If $p$ is of $A_7$ type then the corresponding maximal configuration on $X$ is $A_7+A_4$.
		\item[(5)] If $p$ is of $A_5$ type then the corresponding maximal configuration on $X$ is $2A_5+A_1$.
	\end{itemize}
\end{proposition}

\begin{proof}
	\hfill
	\begin{itemize}
	\item[(11)] By Theorem \ref{projthm}, an $A_{11}$ singularity at $p$ on $X$ gives an $A_9$ singularity at $v$ on $C$. A singularity of $A_{9}$ type can either come from one irreducible component of $C$ or from intersecting two components. Using the same argument as in the $A_{12}$ case in Proposition \ref{A12}, we can show that $A_9$ cannot come from one component.
	
	\smallskip
	
	If $A_9$ comes from the intersection of two components $A$ and $B$, we have $\Ab.\sigma=1$ and $\Bb.\sigma=1$. Then either $[\Ab]=f$ and $[\Bb]=2\sigma+5f$, or $[\Ab]=\sigma+3f$ and $[\Bb]=\sigma+3f$. In the first case we have $\Ab.\Bb=2$, in the second case we have $\Ab.\Bb=4$. However $\Ab.\Bb$ should be at least 4 since we have an $A_7$ singularity on $\Cb$ and thus the first case is not possible.
	
	\smallskip
	
	If $\Ab$ is reducible, then $\Ab=\Ab'\cup\Ab''$, $[\Ab']=f$ and $[\Ab'']=\sigma+2f$. This situation is impossible because it is the same as the first case from the previous paragraph after relabeling the components. For the same reason, $\Bb$ cannot be reducible. By the genus formula, $g_a(\sigma+3f)=0$. Since $g_a(A)=g_a(\Ab)=0$ and $g_a(B)=g_a(\Bb)=0$, both $A$ and $B$ are smooth and $v$ is the only singularity on $C$, which implies that $p$ is the only singularity on $X$.
	
	\smallskip
	
	Example \ref{A11} shows that $A_{11}$ appears on a cubic threefold.
	
	\smallskip
	
	\item[(9)] By Theorem \ref{projthm}, an $A_9$ singularity at $p$ on $X$ gives an $A_7$ singularity at $v$ on $C$. A singularity of $A_7$ type can either come from one irreducible component of $C$ or from intersecting two components. By Proposition \ref{gdrop}, if $C$ is irreducible then $v$ is its only singularity. Now assume that $C$ is reducible and $A_7$ comes from an irreducible component $A$ of $C$. In this case we have $[\Ab]=\sigma+4f$ and $g_a(\Ab)=0$. Then $g_a(A)=g_a(\Ab)+1=1$ and $A$ cannot have an $A_7$ singularity by Proposition \ref{gdrop}.
	
	\smallskip
	
	If $A_7$ comes from the intersection of two components $A$ and $B$, then, similarly to part (11), $[\Ab]=[\Bb]=\sigma+3f$ and both $A$ and $B$ are smooth. Since $\Ab.\Bb=4$ and an $A_7$ singularity at $v$ on $C$ corresponds to a triple intersection on $\Cb$, we have an extra $A_1$ singularity on $C$. Thus we get an $A_7+A_1$ configuration of singularities on $C$ and an $A_9+A_1$ configuration on $X$ by Theorem \ref{projthm}.
	
	\smallskip
	
	An $A_{11}$ singularity can be deformed to $A_9+A_1$, and thus $A_9+A_1$ appears on cubic threefolds.
	
	\smallskip
	
	\item[(7)] By Theorem \ref{projthm}, an $A_7$ singularity at $p$ on $X$ gives an $A_5$ singularity at $v$ on $C$. A singularity of $A_5$ type can either come from one irreducible component of $C$ or from intersecting two components. By Proposition \ref{gdrop}, if $C$ is irreducible, then we have an $A_7$ and possibly one additional $A_2$ or $A_1$ singularity on $X$. Similarly to part (9), $A_5$ cannot come from one irreducible component if $C$ is reducible.
	
	\smallskip
	
	Now assume that $A_5$ comes from the intersection of two components $A$ and $B$. Then either $[\Ab]=f$ and $[\Bb]=2\sigma+5f$, or $[\Ab]=\sigma+3f$ and $[\Bb]=\sigma+3f$.
	
	\smallskip
	
	If $[\Ab]=f$, $[\Bb]=2\sigma+5f$, then $g_a(\Ab)=0$, $g_a(\Bb)=2$. If $\Bb$ is irreducible, then we can get an additional $A_i$ singularity with $i\leq 4$ on $\Bb$ or one of the configurations $2A_2$, $A_2+A_1$, $2A_1$. However, an $A_7+2A_2$ configuration cannot appear on $X$ because it deforms to $A_6+2A_2$ by Theorem \ref{Gr}, and $A_6+2A_2$ is not possible by Proposition \ref{A12}. If $\Bb$ is reducible, then $[\Bb]$ splits as $(\sigma+3f)+(\sigma+2f)$ or $f+(2\sigma+4f)$. In the first case we get $\Ab.(\sigma+3f)=1$, in the second case we get $\Ab.f=0$. Since an $A_5$ singularity at $v$ on $C$ corresponds to a double intersection on $\Cb$, we get a contradiction.
	
	\smallskip
	
	If $[\Ab]=\sigma+3f$, $[\Bb]=\sigma+3f$ and both $A$ and $B$ are irreducible, then, analogously to parts (11) and (9), we have either an $A_7+A_3$ or an $A_7+2A_1$ configuration on $X$. If $\Ab$ is reducible, then we can relabel the components and get the $[\Ab]=f$, $[\Bb]=2\sigma+5f$ case which we have already considered.
	
	\smallskip
	
	Combinations of type $A_7+A_4$ appear on cubic threefolds by Theorem \ref{1-sym-ss}.
	
	\smallskip
	
	\item[(5)] By Theorem \ref{projthm}, an $A_5$ singularity at $p$ on $X$ gives an $A_3$ singularity at $v$ on $C$. A singularity of $A_3$ type can either come from one irreducible component of $C$ or from intersecting two components. By Proposition \ref{gdrop}, if $C$ is irreducible then the corresponding maximal configurations on $X$ are $A_5+A_4$ and $A_5+2A_2$. Similarly to parts (7) and (9), $A_3$ cannot come from one irreducible component if $C$ is reducible.
	
	\smallskip
	
	Now assume that $A_3$ comes from the intersection of two components $A$ and $B$. Then either $[\Ab]=f$ and $[\Bb]=2\sigma+5f$, or $[\Ab]=\sigma+3f$ and $[\Bb]=\sigma+3f$.
	
	\smallskip
	
	If $[\Ab]=f$, $[\Bb]=2\sigma+5f$, $g_a(\Ab)=0$, $g_a(\Bb)=2$ and $\Bb$ is irreducible, then we can get and additional $A_i$ singularity with $i\leq 4$ on $\Bb$ or one of the configurations $2A_2$, $A_2+A_1$, $2A_1$. We also get an $A_1$ singularity from intersecting $A$ and $B$. Thus the corresponding maximal configurations on $X$ are $A_5+A_4+A_1$ and $A_5+2A_2+A_1$. If $\Bb$ is reducible, then $[\Bb]$ splits as $(\sigma+3f)+(\sigma+2f)$ or $f+(2\sigma+4f)$. In the first case we have $\Ab.(\sigma+3f)=1$, in the second case we have $\Ab.f=0$. Since the $A_3$ singularity at $v$ on $C$ corresponds to an $A_1$ singularity on $\Cb$, $[\Bb]=(\sigma+3f)+(\sigma+2f)$. We get an $A_3$ singularity on $C$ from intersecting $\Ab$ and $\sigma+3f$ and an $A_1$ singularity from intersecting $\Ab$ and $\sigma+2f$. The intersection number of the two components of $\Bb$ is $(\sigma+3f).(\sigma+2f)=3$ which means that there is an additional $A_5$ singularity on $\Bb$ or one of the configurations $A_3+A_1$, $3A_1$. Thus the corresponding maximal configuration on $X$ is $2A_5+A_1$.
	
	\smallskip
	
	If $[\Ab]=\sigma+3f$, $[\Bb]=\sigma+3f$ and both $A$ and $B$ are irreducible, then, analogously to the previous parts, we have one of the configurations $2A_5$, $A_5+A_3+A_1$, $A_5+3A_1$ on $X$. If $\Ab$ is reducible, then we can relabel the components and get the $[\Ab]=f$, $[\Bb]=2\sigma+5f$ case which we have already considered.
	
	\smallskip
	
	Combinations of type $2A_5+A_1$ appear on cubic threefolds by Theorem \ref{1-sym-ss}. \qedhere
    \end{itemize}
\end{proof}

\begin{proposition} \label{a3max}
	If $p$ is of $A_{3}$ type, then the corresponding maximal configurations on $X$ are $3A_3+A_1$, $2A_3+A_2+2A_1$ and $2A_3+4A_1$.  
\end{proposition}

\begin{proof}
	By Theorem \ref{projthm}, an $A_{3}$ singularity at $p$ on $X$ gives an $A_1$ singularity at $v$ on $C$. By Proposition \ref{gdrop}, if $C$ is irreducible, then the possible maximal combinations on $X$ are $A_3+3A_2$ and $2A_3+A_2$. However, an $A_3+3A_2$ configuration cannot appear on $X$ by Proposition \ref{a33a2}. Thus the the corresponding maximal combinations on $X$ are $A_3+2A_2+A_1$ and $2A_3+A_2$.
	
	\smallskip
	
	Now assume $C$ is reducible. If there is an irreducible component $A$ such that $\Ab.\sigma=2$, then $[\Ab]=\sigma+4f$, $[\Bb]=\sigma+2f$, $g_a(A)=g_a(\Ab)+1=1$, $g_a(B)=g_a(\Bb)=0$, and $B$ is irreducible. We can get $2A_3$, $A_3+2A_1$ or $4A_1$ singularities from intersecting $A$ and $B$. The corresponding maximal configuration on $X$ is $3A_3$.
	
	\smallskip
	
	If $C$ is a union of two components $A$ and $B$ such that $A$ is irreducible, $\Ab.\sigma=1$ and $\Bb.\sigma=1$, then either $[\Ab]=f$ and $[\Bb]=2\sigma+5f$, or $[\Ab]=\sigma+3f$ and $\Bb=\sigma+3f$.

	\smallskip
	
	If $[\Ab]=f$, $[\Bb]=2\sigma+5f$ and $B$ is irreducible, we get an $A_1$ singularity at $v$ and an additional $A_3$ or $2A_1$ combination of singularities from intersecting $A$ and $B$. By Proposition \ref{gdrop}, the possible maximal combinations on $B$ are $2A_2$ and $A_3$. This gives us a $2A_3+2A_2$ or $3A_3$ configuration on $X$.
	However $2A_3+2A_2$ deforms to $A_3+3A_2$ which cannot appear on $X$ by Proposition \ref{a33a2}. Thus the maximal combinations on $X$ we get in this case are $3A_3$, $2A_3 + A_2 + A_1$, and $A_3+2A_2+2A_1$.
	
	\smallskip
	
	If $B$ is reducible, then we have the following options:
	\begin{itemize}
		\item[(i)] $\Ab=f$, $\Bb=f+(2\sigma+4f)$;
		\item[(ii)] $\Ab=f$, $\Bb=(\sigma+2f)+(\sigma+3f)$;
		\item[(iii)] $\Ab=f$, $\Bb=f+(\sigma+2f)+(\sigma+2f)$.
	\end{itemize}
	The tentative maximal configuration on $X$ corresponding to case (i) is $3A_3+A_2$. Indeed, we get can get an $A_3$ from intersecting each $f$ with $2\sigma+4f$, and at worst $A_2$ on $2\sigma+4f$. However, $3A_3+A_2$ deforms to $A_3+3A_2$, thus the remaining maximal cases are $2A_3+A_2+2A_1$ and $3A_3+A_1$.
	
	\smallskip
	
	In case (ii), we have an $A_1$ singularity from intersecting $f$ and $\sigma+2f$ components, another $A_1$ singularity from intersecting $f$ and $\sigma+3f$ and an $A_3+A_1$ or $3A_1$ configuration from intersecting $\sigma+2f$ and $\sigma+3f$. The corresponding maximal combination on $X$ is $2A_3+3A_1$.
	
	\smallskip
	
	In case (iii), we have an $A_1$ singularity from intersecting each $f$ with each of the $\sigma+2f$ components. Additionally, we get an $A_3$ or $2A_1$ from intersecting the two $\sigma+2f$ components. Thus the maximal configuration is $2A_3+4A_1$.
	
	\smallskip
	
	If $[\Ab]=\sigma+3f$, $[\Bb]=\sigma+3f$ and $B$ is irreducible, we get a $2A_3$, $A_3+2A_1$ or $4A_1$ configuration from intersecting $\Ab$ and $\Bb$. The corresponding maximal combination on $X$ is $3A_3$. If $B$ is reducible then we can relabel the components and get the $\Ab=f$, $\Bb=2\sigma+5f$ case which we have already considered.
	
	\smallskip
	
	The combinations from the statement of the proposition can be obtained by deforming an $D_5+2A_3$ configuration (types $3A_3+A_1$ and $2A_3+A_2+2A_1$) or a  $D_6+A_3+2A_1$ configuration (type $2A_3+4A_1$).
\end{proof}

\begin{proposition} \label{a2max}
	If $p$ is of $A_{2}$ type, then the corresponding maximal configurations on $X$ are $5A_2$, $2A_2+4A_1$ and $A_2+6A_1$.
\end{proposition}

\begin{proof}
	By Lemma \ref{a2}, $C$ does not pass through $v$. By Proposition \ref{gdrop}, if $C$ is irreducible then the corresponding maximal configuration on $X$ is $5A_2$ (see Example \ref{5A2}).
	
	\smallskip
	
	Now assume that $C=A\cup B$ and $A$ is irreducible. Then $[\Ab]=\sigma+2f=H$, $[\Bb]=2H$ where $H$ is as in Remark \ref{H}. If $B$ is irreducible, we have $g_a(\Ab)=0$, $g_a(\Bb)=1$, $\Ab.\Bb=4$. Thus we get $4A_1$ singularities from intersecting $\Ab$ and $\Bb$, and possibly one $A_2$ or $A_1$ singularity on $\Bb$. The corresponding maximal configuration on $X$ is $2A_2+4A_1$. If $B$ is reducible, we have $[\Bb]=H+H$, $g_a(H)=0$. In this case all of the components of $C$ are smooth, and we get $6A_1$ singularities from intersecting these components. The corresponding maximal configuration on $X$ is $A_2+6A_1$.
	
	\smallskip
	
	One can get a cubic threefold with $2A_2+4A_1$ or $A_2+6A_1$ combinations by deforming a threefold with a $2A_3+A_2+2A_1$ combination.  
\end{proof}

\begin{proposition} \label{a33a2}
	A constellation of $A_3+3A_2$ type cannot appear on a cubic threefold.
\end{proposition}

\begin{proof}
	Assume it is possible and let $p\in X$ be the $A_3$ singularity. It follows from the proof of Proposition \ref{a3max} that the corresponding curve $C$ is irreducible. Consider a smooth hyperplane section $L$ of the quadric cone $Q$. It is isomorphic to $\bbP^1$. Projecting from $v\in C$, we get a branched double cover $\eta: \Cb\to L$. By Theorem \ref{projthm}, $C$ has $3A_2+A_1$ singularities and $\Cb$ has $3A_2$ singularities. Consider the normalization $\nu: C^{\nu}\to\Cb$. The curve $C^\nu$ is isomorphic to $\bbP^1$ by Lemma \ref{classofc} and Proposition \ref{gdrop}. The double cover $\eta\circ\nu: C^{\nu}\to L$ has at least 3 branch points corresponding to the $A_2$ singularities. Thus, by the Riemann-Hurwitz formula, $\chi(C^{\nu})\leq 2\chi(L)-3$ and we get $2\leq 4-3=1$. Contradiction.
\end{proof}

\begin{example} \label{A11}
	By an example of Allcock (\cite{allcock1}), there are cubic threefolds with one $A_{11}$ singularity. For instance, such a threefold can be given by the following equation: 
	$$
	f=x_0(x_3^2-x_2x_4)+x_2^3+x_1^2x_4-2x_1x_2x_3+x_4^3.
	$$
\end{example}

\begin{example} \label{5A2}
	By Theorem 3.7 in \cite{Moe}, there exists a curve $\tilde{C}$ on the Hirzebruch surface $\bbF_2$ with four $A_2$ singularities such that $[\tilde{C}]=3\sigma+6f$. After blowing down, the corresponding curve $C$ on the quadric cone $Q$ becomes a $(2,3)$ complete intersection curve in $\bbP^3$ with four cusps. Thus $C$ and $Q$ define a cubic threefold with $5A_2$ singularities.
\end{example}

After combining the statements of Propositions \ref{a1max}, \ref{A12}, \ref{A_odd}, \ref{a3max}, \ref{a2max} and applying Theorems \ref{Gr} and \ref{F:DPW}, we get the following theorem:

\begin{theorem} \label{An_max}
	Among the configurations of singularities only containing $A_n$ singularities, the maximal ones are $A_{11}$, $A_7+A_4$, $2A_5+A_1$, $3A_3+A_1$, $2A_3+A_2+2A_1$, $2A_3+4A_1$, $5A_2$ and $10A_1$. The list of all of the possible constellations of $A_n$ singularities consists of 109 cases and is given in Table \ref{Ta:anlist}.
\end{theorem}

The theorem above is part (2) of Theorem \ref{Theorem1}. Together with the results of Section \ref{ADEsection} (see Table \ref{crk2jnot0}) as well as Theorems \ref{Gr} and \ref{F:DPW}, it implies part (1) of Theorem \ref{Theorem1}.

\section{Combinatorial description of the $ADE$ configurations} \label{sect_comb}

In this section, we present a graph $\Gamma$ whose induced subgraphs (see Definition \ref{induced}) correspond to combinations of $ADE$ singularities on cubic threefolds. We construct $\Gamma$ by adding vertices and edges to the union of three $D_4$ Dynkin diagram. First, we get a graph $\Delta$ by adding one vertex and three edges to $3D_4$ (Figure~\ref{3D4}). However, $\Delta$ does not contain $A_{11}$ or $E_8$ as subgraphs, and the corresponding singularities occur on cubic threefolds. After adding two more vertices and six more edges to $\Delta$, we obtain $\Gamma$ (Figure \ref{gamma1}). Table~\ref{Maximal configurations} shows which vertices we need to remove from $\Gamma$ to get most of the maximal $ADE$ combinations. The only two maximal combinations that we do not get from $\Gamma$ are $5A_2$ and $10A_1$ (but we do get the combinations $4A_2+A_1$ and $9A_1$).

\begin{figure}
	\begin{center}
		\includestandalone{3D4}
		\bigskip
		\caption{Graph $\Delta$}
		\label{3D4}
	\end{center}
\end{figure}      

\begin{figure}[htb]
	\begin{center}
		\includestandalone{Gamma_labeled}
		\caption{Graph $\Gamma$} \label{gamma1}
	\end{center}
\end{figure}

\renewcommand{\arraystretch}{1.25}

\begin{table}[h]
	\begin{center}
		\begin{tabular}{|c | c | c|}\hline
			Maximal configuration & Points to remove from $\Gamma$ \\ \hline \hline
			$E_8+A_2$& 3, 8, 9, 11, 15 \\ \hline
			$E_7+A_2+A_1$& 3, 4, 8, 9, 15 \\ \hline
			$E_6+2A_2$& 3, 8, 9, 14, 15 \\ \hline
			$D_8+A_3$& 1, 3, 11, 15 \\ \hline
			$D_6+A_3+2A_1$& 1, 3, 4, 15 \\ \hline
			$D_5+2A_3$& 1, 3, 14, 15 \\ \hline
			$3D_4$& 1, 3, 5  \\ \hline
			$A_{11}$& 3, 8, 12, 13 \\ \hline
			$A_7+A_4$& 3, 9, 13, 14  \\ \hline
			$2A_5+A_1$& 2, 3, 9, 14 \\ \hline
		\end{tabular}
		\bigskip
		\caption{Maximal $ADE$ configurations and the corresponding induced subgraphs of $\Gamma$} \label{Maximal configurations}
		\vspace*{-1em}
	\end{center}
\end{table}

\renewcommand{\arraystretch}{1}

\begin{manualtheorem}{II} \label{Theorem2}
	An $ADE$ combination of singularities occurs on a cubic threefold if and only if the union of the corresponding Dynkin diagrams is $10A_1$, $5A_2$, or an induced subgraph of the graph $\Gamma$ (Figure \ref{gamma1}).
\end{manualtheorem}

\begin{remark} \label{14vertices}
	Notice that we remove vertex 3 in each of the cases in Table \ref{Maximal configurations}. It means that instead of $\Gamma$ we can consider a graph $\Gamma'$ that is obtained from $\Gamma$ by removing 3. We choose to consider $\Gamma$ because it is more symmetric. In particular, $\Gamma$ is a subdivided version of the $K_{3,3}$ graph. The six vertices 1, 2, 3, 4, 5, 6 have degree three in the graph $\Gamma$, and the remaining nine vertices $7,\ldots,15$ have degree two in $\Gamma$. Each of the degree three vertices can be moved to any other degree three vertex by a symmetry of $\Gamma$, and the same holds for the degree two vertices.
\end{remark}

\begin{lemma} \label{3D4_graph}
	If an induced subgraph of $\Delta$ is a union of $ADE$ graphs, then the corresponding combination of $ADE$ singularities appears on a cubic threefold.
\end{lemma}

\begin{proof}
	First notice that if we remove the central vertex of $\Delta$, we get the $3D_4$ graph. There exists a cubic threefold with $3D_4$ singularities (Proposition \ref{3D4sing}), and thus the induced subgraphs of $3D_4$ correspond to configurations of singularities on cubic threefolds by Theorems \ref{Gr} and \ref{F:DPW}. Now assume that we do not remove the center, and consider the following cases:
	\begin{enumerate}
		\item we remove a vertex that is one edge away from the center;
		\item we remove a vertex that is two edges away from the center and do not remove vertices that are one edge away;
		\item we only remove vertices that are three edges away from the center.
	\end{enumerate}
	In case $(3)$, a resulting graph cannot be of $ADE$ type because it contains the $\Esx$ graph. The possibilities in case $(1)$ are shown in Figure \ref{3D4proof}. We get the $D_5+2A_3$, $D_6+A_3+2A_1$ and $D_8+A_3$ subgraphs, and there are cubic threefolds with the corresponding combinations of singularities (Proposition \ref{Dn_max}). The analysis in case $(2)$ is similar. 
\end{proof}

\tikzset{every picture/.style={line width=0.75pt}}

\begin{figure}
	\begin{center}
		\includestandalone{3D4proof}
		\bigskip
		\caption{$D_5+2A_3$, $D_6+A_3+2A_1$ and $D_8+A_3$ as induced subgraphs of $\Delta$}
		\label{3D4proof}
	\end{center}
\end{figure}

\begin{proof} [Proof of Theorem \ref{Theorem2}]	
	One direction follows from Table \ref{Maximal configurations}. For the other direction, we apply the same strategy as in the proof of Lemma \ref{3D4_graph}. Namely, we consider the possible ways to remove the first few vertices, and after the resulting subgraph is of $ADE$ type, we check that it appears in the classification of singularities. By Theorems \ref{Gr} and \ref{F:DPW}, all of the smaller subgraphs correspond to possible configurations as well. We split the proof into three cases: when we remove at least 2 degree three points, exactly 1 degree three point or no degree three points. The labeling of vertices is as in Figure \ref{gamma1}.
	
	\smallskip
	
	\begin{itemize}
		\item[(1)] Suppose we remove 2 degree three vertices. Without loss of generality, we can say that these vertices are either 1 and 3 or 2 and 3. If we remove 1 and 3, we get the graph $\Delta$ which only contains induced $ADE$ subdiagrams that come from combinations of singularities on cubic threefolds by Lemma \ref{3D4_graph}.
		
		\smallskip
		
		Now assume that we remove vertices 2 and 3. If we remove an additional degree three vertex, then the graph we get is a subgraph of $\Delta$, and we can use Lemma \ref{3D4_graph} again. If we only remove 2 and 3, the resulting graph contains the cycle $(1,8,4,14,5,15,6,9)$. Thus we need to remove $8,9,14$ or $15$ which have degree two in $\Gamma$. Since the picture is symmetric, we can choose to remove 8. Figure \ref{gamma_proof_1} shows the graph we get (the points that have degree three in $\Gamma$ are marked with green). The case by case analysis of this graph is straightforward and can be done similarly to the analysis in the proof of Lemma \ref{3D4_graph}.
		
		\tikzset{every picture/.style={line width=0.75pt}}
		
		\begin{figure}[htb]
			\begin{center}
				\includestandalone{Gamma_proof_1}
				\caption{Theorem \ref{Theorem2}, Part (1) of the proof} \label{gamma_proof_1}
			\end{center}
		\end{figure}
	
		\smallskip

		\item[(2)] Suppose we remove exactly 1 vertex of degree three, say 3. The remaining graph contains three cycles: $(1,7,2,13,5,14,4,8)$, $(1,7,2,13,5,15,6,9)$, and $(1,8,4,14,5,15,6,9)$. To get rid of them, we need to remove at least two vertices. By symmetry, we can assume these vertices are either 7 and 8 or 7 and 14. Both possibilities are shown in Figure \ref{gamma_proof_2}, and their analysis is straightforward (notice that we are not allowed to remove the vertices marked with green here because they have degree three in $\Gamma$).
		
		\tikzset{every picture/.style={line width=0.75pt}}    
		
		\begin{figure}[htb]
			\begin{center}
				\includestandalone{Gamma_proof_2}
				\caption{Theorem \ref{Theorem2}, Part (2) of the proof} \label{gamma_proof_2}
			\end{center}
		\end{figure}
	
		\smallskip
		
		\item[(3)] First we make a few observations. Let $\gamma$ be an induced subgraph of $\Gamma$ which contains the vertices $1,\ldots,6$. Notice that if 1 has degree three in $\gamma$ (i.e. $7,8,9\in\gamma$), then $\gamma$ contains an $\Esx$ subgraph and thus $\gamma$ is not a union of $ADE$ graphs. Same holds if 2, 3, 4, 5 or 6 have degree three in $\gamma$. Thus $\gamma$ contains only vertices of degree two or less. If one of the vertices $7,\ldots,15$ is in $\gamma$, then it is of degree two in $\gamma$ because $1,\ldots,6$ are in $\gamma$. If $1,\ldots,6$ are all of degree two in $\gamma$, then $\gamma$ contains a cycle and is not a union of $ADE$ graphs. Thus there is a vertex, say 1, of degree one. We can assume that $7\in\gamma$ and $8,9\notin\gamma$.
		
		\smallskip
		
		If vertex 2 is of degree one, then 10 and 13 are not in $\gamma$. It implies that $\gamma$ is an induced subgraph of the union of the $A_3$ diagram containing 1, 7, 2 and the cycle $(3,11,4,14,5,15,6,12)$ of length 8. After breaking the cycle, we get an $A_7$ diagram. An $A_7+A_3$ combination of singularities can appear on a cubic threefold.
		
		\smallskip
		
		If 2 is of degree two, without loss of generality, $10\in\gamma$. If 3 is of degree one, then $\gamma$ is contained in the $2A_5$ diagram containing the vertices 1, 7, 2, 10, 3 and 1, 14, 5, 15, 6. A combination of two $A_5$ singularities is possible on cubic threefolds.
		
		\smallskip
		
		If 3 is of degree two, we can assume that $11\in\gamma$. If 4 is of degree one, then $\gamma$ is an induced subgraph of the $A_7+A_3$ diagram containing the vertices 1, 7, 2, 10, 3, 11, 4 and 5, 15, 6.
		
		\smallskip
		
		If 4 is of degree two, then $14,5\in\gamma$. If $15\notin\gamma$, then $\gamma$ is a union of the $A_9$ diagram containing the vertices 1, 7, 2, 10, 3, 11, 4, 14, 5 and the $A_1$ diagram $\{6\}$. If $15\in\gamma$, then $\gamma$ is the $A_{11}$ diagram containing the vertices 1, 7, 2, 10, 3, 11, 4, 14, 5, 15, 6. Both $A_{11}$ and $A_9+A_1$ combinations appear on cubic threefolds. \qedhere
	\end{itemize}
\end{proof}

\begin{remark}
	We would like to point out that one can construct (somewhat artificially) a graph that contains 16 vertices such that its $ADE$ induced subgraphs are in one-to-one correspondence with $ADE$ combinations on cubic threefolds. One proceeds as follows. First consider $\Gamma'$ as in Remark \ref{14vertices}, add one vertex to it and connect this vertex to all vertices of $\Gamma'$ that have degree 3 in $\Gamma$ by a dashed edge. Then add another vertex and connect it to vertices 2, 8, 9, 14, 15 by a dashed edge and to vertex 10 by a regular edge. Since $ADE$ diagrams do not have dashed edges, we only get $10A_1$ and $5A_2$ diagrams as new induced subgraphs. While we do not have a proof, we do not expect that there is a graph with 16 or less vertices without dashed edges such that its $ADE$ induced subgraphs are in one-to-one correspondence with $ADE$ combinations on cubic threefolds.
\end{remark}

\appendix

\section{Singularity theory} \label{SingTheory}

In this paper, we study hypersurface singularities. In particular, it means that locally these singularities are critical points of germs of holomorphic functions.

\medskip

Let $\calO_n$ be the set of holomorphic function-germs in $n$ variables at the point $0\in\bbC^n$:
$$f: (\bbC^n,0)\rightarrow (\bbC,0).$$

\begin{definition}
	A function-germ $f\in\calO_n$ has an \textit{isolated} critical point at zero if there exists a neighbourhood $U\subset \bbC^n$ of zero and a function $g: U\rightarrow \bbC$ representing $f$ such that zero is the unique critical point of $g$.
\end{definition}

\begin{definition} [{\cite{AGLV}, Chapter 1.1.2}]
	Two function-germs in $\calO_n$ are said to be \textit{equivalent} if one is taken to the other by a biholomorphic change of coordinates that keeps zero fixed. We call the equivalence class of a singularity its \textit{singularity type}.
\end{definition}

We now introduce three important invariants of hypersurface singularities (all of which are well-defined for singularity types):

\begin{definition} [{\cite{AGLV}, Chapter 1.1.1}] \label{corank}
	The \textit{corank} of a critical point of a function is the dimension of the kernel of its second differential at the critical point.
\end{definition}

\begin{definition} [{\cite{AGLV}, Chapter 1.1.4}] \label{Milnor}
	Consider the \textit{gradient ideal} $I_{\nabla_f}$ generated by partial derivatives of a function-germ $f$ and let $Q_f=\calO_n/I_{\nabla_f}$ The \textit{Milnor number} $\mu(f)$ is defined as $\dim_\bbC Q_f$. If $X$ is a hypersurface with several isolated singularities, then the \textit{total Milnor number} $\mu(X)$ is the sum of Milnor numbers of all of the singularities.
\end{definition}

\begin{definition} \label{Tjurina}
	With $I_{\nabla_f}$ as in the definition above, let $T_f=\calO_n/(f,I_{\nabla_f})$. The \textit{Tjurina number} $\tau(f)$ is defined as $\dim_\bbC T_f$. If $X$ is a hypersurface with several isolated singularities, then the \textit{total Tjurina number} $\tau(X)$ is the sum of Tjurina numbers of all of the singularities.
\end{definition}

\begin{remark} \label{TjM}
	It follows immediately from the definitions that $\tau(f)\leq\mu(f)$ (and $\tau(X)\leq\mu(X)$). For instance, the equality holds in the $ADE$ case.
\end{remark}

Some of the main results we use in this work (Theorems \ref{Gr} and \ref{F:DPW}) are related to the notion of deformation of singularities:

\begin{definition} [{\cite{AGLV}, Chapter 1.1.9}]
	A \textit{deformation} with base $\bbC^l$ of a function-germ $f(x)$ is the germ at zero of the smooth map $F: (\bbC^n\times\bbC^l)\rightarrow (\bbC,0)$ such that $F(x,0)\equiv f(x)$.
	
	\smallskip
	
	A deformation $F(x,\lambda)$ of a germ $f(x)$ is said to be \textit{versal} if every deformation $F'$ of $f(x)$ can be represented in the form
	$$
	F'(x,\lambda')=F(g(x,\lambda'),\theta(\lambda'))
	$$
	such that $g(x,0)\equiv x$, $\theta(0)=0$.
\end{definition}

\begin{definition} ([{\cite{AGLV}, Chapter 1.2.7}]) \label{adjacent}
	A class of singularities $L$ is said to be \textit{adjacent} to a class $K$, and one writes $L\rightarrow K$, if every function $f\in L$ can be deformed to a function  of class $K$ by an arbitrarily small perturbation.
\end{definition}

\begin{table}[htb]
	$$\xymatrix@R=.25cm@C=.25cm{
		{E_6}&& \ar@{>}[ll]{E_7}&& \ar@{>}[ll]{E_8}&& && &&\\
		&&&&&&&&&&\\
		\ar@{>}[uu]{P_8}&& \ar@{>}[uu]{X_9}&& \ar@{>}[uu]{J_{10}}&& && &&\\
		&&&&&&&&&&\\
		&&\ar@{>}[uull]{T_{334}}&&\ar@{>}[uull]{T_{245}}&&
		\ar@{>}[uull]\ar@{>}[ll]{T_{255}}&& &&\\
		&&&&&&&&&&\\
		&& &&\ar@{>}[uuuull]\ar@{>}[uull]{T_{344}}&&\ar@{>}[uull]\ar@{>}[uuuull]{T_{246}}&& &&\\
		&&&&&&&&&&\\
		&& && &&\ar@{>}[uull]{T_{444}}&&\ar@{>}[uull]\ar@{>}[uuuull]{T_{256}}&&\\
		&& &&\ar@{>}[uuuuull]{Q_{10}}&& && &&\ar@{>}[ull]{T_{266}}\\
		&& && &&\ar@{>}[ull]\ar@{>}[uuuull]\ar@{>}[uuuuuull]{S_{11}}&& && \\
		&& && && &&\ar@{>}[ull]\ar@{>}[uuull]\ar@{>}[uuuuuuull]{U_{12}}&& 
	}$$
	\smallskip
	\caption{Adjacencies of unimodal and $E_n$ singularities}
	\label{unimodal}
	\vspace*{-1em}
\end{table}

Notice that adjacencies of $ADE$ singularities are described by Theorem \ref{Gr} and adjacencies of unimodal singularities are classified by Brieskorn \cite{Brieskorn}. We show some adjacencies of unimodal and $E_n$ singularities in Table \ref{unimodal}.

\medskip

Finally, we will need the following results concerning the corank of singularities on cubic threefolds and $A_n$ singularities:

\begin{claim}
	Le $X\subset \bbP^4$ be a cubic threefold given by an equation $f=x_0f_2(x_1,x_2,x_3,x_4)+f_3(x_1,x_2,x_3,x_4)$. Then the corank of $p=[1:0:0:0:0]\in X$ is equal to the corank of $f_2$. 
\end{claim}

\begin{claim} \label{crk1}
	The corank of $f_2$ is equal to one if and only if $p$ is of $A_r$ type with $r>1$. The corank of $f_2$ is equal to zero if and only if $p$ is an $A_1$ singularity.
\end{claim}

\begin{claim} \label{Anblowup}
	Let $X$ be a variety with an $A_n$ singularity at $p$. Then the blow-up $\tilde{X}$ of $X$ at $p$ has an $A_{n-2}$ singularity (or no singularity if $n\leq 2$) which is the unique singularity of $\tilde{X}$ contained in the exceptional divisor.
\end{claim}

\begin{proposition} \label{gdrop}
	If the combination of singularities on a curve $C$ is $a_1A_1+...+a_nA_n$, then $$g_a(C)\geq \sum_{i=1}^n a_i\lceil \frac{i}{2}\rceil.$$
\end{proposition}

\begin{proof}
	The inequality follows from Claim \ref{Anblowup} by induction. The base of induction for $A_1$ and $A_2$ can be proven by taking the short exact sequence of sheaves for normalization.
\end{proof}

\section{Cubic threefolds with one-parameter symmetry groups} \label{Sym_Threefolds}

There are several papers by bu Plessis and Wall (including \cite{DPWI} and \cite{DPW08}) where they study quasi-smooth projective hypersurfaces with symmetry. In particular, they give a complete classification of singularities on quasi-smooth 1-symmetric cubic threefolds in \cite{DPW08}.

\begin{definition}
	A variety $X\subset\bbP^n$ is \textit{quasi-smooth} if it has only isolated singularities and is not a cone.
\end{definition}

For the rest of this section, we will assume that $X\subset\bbP^n$ is a quasi-smooth hypersurface of degree $d$.

\begin{definition}
	We say that $X$ is \textit{$k$-symmetric} if it admits a $k$-dimensional algebraic subgroup $G$ of $PGL_n(\bbC)$ as automorphism group.
\end{definition}
	
\begin{proposition}[{\cite{DPW08}, Corollary 2.6}]
	Suppose $X$ is 1-symmetric. Then the Tjurina number $\tau(X)\leq (d-1)^{n-2}(d^2-3d+3)$.
\end{proposition}

\begin{definition}
	We say that $X$ is \textit{oversymmetric} if it is 1-symmetric and $\tau(X)=(d-1)^{n-2}(d^2-3d+3)$.
\end{definition}

\begin{proposition}[{\cite{DPW08}, Corollary 2.7}]
	The hypersurface $X$ cannot be 3-symmetric; it is 2-symmetric if and only if it is oversymmetric and $d=3$.
\end{proposition}

\begin{corollary}
	If $X$ is a cubic threefold then $\tau(X)\leq 12$. The Tjurina number $\tau(X)=12$ if and only if $X$ is 2-symmetric.
\end{corollary}

In {\cite{DPW08}, du Plessis and Wall describe singularities of 1-symmetric quasi-smooth hypersurfaces. If $X$ is 1-symmetric, there are two possibilities for $G$: a linear algebraic one-parameter group is isomorphic either to the multiplicative group $\bbC^*$ (semisimple case) or to the additive group $\bbC$ (unipotent case).

\medskip

In the semisimple case, singularities of $X$ are determined by the weights of the corresponding $\bbC^*$-action. General methods are introduced in Sections 3 and 5 of \cite{DPWI}. The cubic threefold case is considered in Section 5 of \cite{DPW08}.

\begin{theorem} [{\cite{DPW08}, Section 5}] \label{1-sym-ss}
	Let $X\subset\bbP^4$ be a 1-symmetric quasi-smooth cubic threefold with $G$ semisimple. The following table has a list of possible combinations of singularities on such a threefold together with the weights of the corresponding $\bbC^*$-action. In the two cases marked with $*$ additional singularities can appear: $A_1$ for $[-2, -1, 0, 1, 2]$ and $A_1$, $2A_1$, $A_2$, $3A_1$, $A_3$ or $D_4$ for $[-1, 0, 0, 0, 1]$.
\end{theorem}

\renewcommand{\arraystretch}{1.25}

\begin{table}[h]
	\begin{center}
		\begin{tabular}{c c c | c c c}
			$\mu$ & Weights & Singularities & $\mu$ & Weights & Singularities \\ \hline
			11 & $[-8, -2, 1, 4, 16]$ & $S_{11}$ & 11 & $[-10, -4, 2, 5, 8]$ & $A_7+A_4$ \\
			11 & $[-8, -2, 1, 4, 4]$ & $D_5+2A_3$ & 11 & $[-8, -2, 1, 4, 7]$ & $D_8+A_3$ \\
			11 & $[-5, -2, 1, 1, 4]$ & $D_5+A_3+2A_1$ & 11 & $[-2, -2, 1, 1, 4]$ & $X_9+2A_1$ \\
			10 & $[-4, -1, 0, 2, 8]$ & $Q_{10}$ & 10 & $[-4, -1, 0, 2, 5]$ & $E_8+A_2$ \\
			10 & $[-4, -1, 0, 2, 4]$ & $D_7+A_3$ & 10 & $[-2, -2, 0, 1, 4]$ & $E_6+2A_2$ \\
			10 & $[-2,-1,0,1,3]$ & $E_7+A_2+A_1$ & $10*$ & $[-2, -1, 0, 1, 2]$ & $2A_5$ \\
			12 & $[-2, 0, 0, 1, 4]$ & $U_{12}$ & 12 & $[-2, 0, 0, 1, 2]$ & $J_{10}+A_2$ \\
			8 & $[-1, 0, 0, 0, 2]$ & $P_8$ & $8*$ & $[-1, 0, 0, 0, 1]$ & $2D_4$ \\
		\end{tabular}
		\bigskip
		\caption{Singularities of 1-symmetric cubic threefolds (semisimple case)}
	\end{center}
	\vspace*{-1em}
\end{table}

\renewcommand{\arraystretch}{1}

The unipotent case is handled in Section 4 of \cite{DPW08}. In the cubic threefold case, the classification is given in the theorem below:

\begin{theorem}[{\cite{DPW08}, Section 5}] \label{unipotent}
	Let $X\subset\bbP^4$ be a 1-symmetric quasi-smooth cubic threefold with $G$ unipotent. Then possible combinations of singularities on $X$ are $A_{11}$, $U_{12}$, $T_{266}$, $T_{256}$, $T_{246}+A_1$, $J_{10}+A_2$ and $J_{10}+A_1$.
\end{theorem}

\begin{corollary}
	Let $X\subset\bbP^4$ be a 2-symmetric quasi-smooth cubic threefold. Then $\tau(X)=12$ and $X$ has one of the following combinations of singularities: $U_{12}$, $T_{266}$, $J_{10}+A_2$, $3D_4$.
\end{corollary}

\begin{remark}
	The Tjurina number of a $U_{12}$ singularity can be equal to 11 or 12 (see Section 5 of \cite{DPW08}).
\end{remark}

\section{Lattice theory} \label{LatticeTheory}

In this appendix, we review necessary definitions and results from lattice theory following the papers of Nikulin \cite{Nikulin} and Dolgachev \cite{Dolgachev}.

\begin{definition}
	A \textit{lattice} of rank $r$ is a free finitely generated $\bbZ$-module $S\cong\bbZ^r$ together with a non-degenerate integral symmetric bilinear form $(\:,\:): S\otimes S\rightarrow\bbZ$. The \textit{signature} of a lattice is a pair $(n_+,n_-)$ where $n_+$ and $n_-$ are the numbers of $\pm1$ in the diagonalization of the bilinear form. A lattice $S$ is called \textit{even} if $(s,s)\in 2\bbZ$ for each $s\in S$.
\end{definition}

\begin{definition}
	The finite group $A_S=S^*/S$ is the \textit{discriminant group} of $S$. It can be equipped with a symmetric bilinear form $b_S: A_S\otimes A_S\rightarrow\bbQ/\bbZ$ induced by the bilinear form on $S$. If $S$ is even, then there is an induced quadratic form $q_S: A_S\rightarrow\bbQ/2\bbZ$. Lattices with trivial discriminant group are called \textit{unimodular}.
\end{definition}

\begin{example}
	The lattice $U$ is an even unimodular lattice of rank 2. Its bilinear form is isomorphic to the one given by the matrix $\begin{psmallmatrix} 0 & 1 \\ 1 & 0 \end{psmallmatrix}$.
\end{example}

\begin{example} \label{DiscrExamples}
	As mentioned in Section \ref{DynkinDiagrams}, Dynkin diagrams of $ADE$ types determine bilinear forms which correspond to $ADE$ lattices. The discriminant groups of these lattices are as follows (for instance, see \cite{Mongardi}, Examples 2.1.11-2.1.14):
	\begin{itemize}
		\item the discriminant group of $A_n$ is $\bbZ/(n+1)\bbZ$;
		\item the discriminant of $D_n$ is $\bbZ/4\bbZ$ when $n$ is odd, and $\bbZ/2\bbZ\times\bbZ/2\bbZ$ when $n$ is even;
		\item $E_8$ is unimodular, $A_{E_7}\cong A_{A_1}$, and $A_{E_6}\cong A_{A_2}$.
	\end{itemize}
\end{example}

\begin{definition}
	Let $S\subset L$ be a sublattice. The lattice $Sat(S)=(S\otimes_{\bbZ}\bbQ)\cap L$ is called the \textit{saturation} of~$S$. If $S=Sat(S)$, then $S$ is \textit{saturated} in $L$. An embedding of lattices $\iota: S\rightarrow L$ is called \textit{primitive} if $\iota(S)\cong S$ is saturated in $L$.
\end{definition}

\begin{definition}
	Let $\iota: S\rightarrow L$ be an embedding of lattices. If the group $H_L=L/S$ is finite, then $\iota$ is an \textit{overlattice} of $S$. If $L$ is even, then $\iota$ is an \textit{even} overlattice. Notice that $H_L\subset L^*/S\subset S^*/S\subset A_S$ since there is a chain of embeddings $S\hookrightarrow L\hookrightarrow L^*\hookrightarrow S^*$.
\end{definition}

\begin{proposition} [{\cite{Nikulin}, Proposition 1.4.1}] \label{Saturation}
	Let $S$ be even. The correspondence $L\mapsto H_L$ determines a bijection between even overlattices of $S$ and isotropic subgroups of $A_S$ (isotropic means $q_S|_{H_L}=0$). Moreover, $H_L^\perp=L^*/S\subset A_S$ and $(q_S|_{H_L^\perp})/H_L=q_L$
\end{proposition}

\begin{corollary} \label{Sat(A_n)}
	Any embedding of the lattice $A_{p-1}$ with $p$ prime is primitive.
\end{corollary}

\begin{proof}
	The discriminant $\bbZ/p\bbZ$ of $A_{p-1}$ does not have nontrivial subgroups.
\end{proof} 

\begin{proposition} [{\cite{Nikulin}, Proposition 1.5.1}] \label{EmbeddingQ}
	A primitive embedding of an even lattice $S$ into an even lattice $L$ with discriminant form $q$ such that $S^{\perp}\subset L$ is isomorphic to $K$ is determined by a pair $(H,\gamma)$, where $H\subset A_S$ is a subgroup and $\gamma:H\hookrightarrow A_K$ is a group monomorphism, while $q_K\circ\gamma=-q_S|_H$ and
	$$
	(q_S\oplus q_K|_{\Gamma_\gamma^\perp})/\Gamma_\gamma\simeq q,
	$$
	where $\Gamma_\gamma$ is the graph of $\gamma$ in $A_S\oplus A_K$.
\end{proposition}

\begin{corollary} \label{EmbeddingDiscr}
	Under the assumptions of Proposition \ref{EmbeddingQ}, we have $A_L=\Gamma_\gamma^\perp/\Gamma_\gamma\subset A_S\oplus A_K/\Gamma_\gamma$.
\end{corollary}

\clearpage

\section{The list of combinations of singularities on cubic threefolds} \label{Tables}

\begin{table}[htb]
	\begin{center}
		\begin{tabular*}{5.90in}{@{\extracolsep{\fill}}|c | c | c| p{0.95in}|| c| c| c| 
				p{0.95in}|| c| c| c| p{0.95in}|}\hline
			T &$\mu$&k&type&   T &$\mu$&k&type    &        T  &$\mu$&k&type \\\hline\hline
			$1$&$16$&$1$&$O_{16}      $&$ 7$&$9$ &$1$&$T_{334}    $&$13$&$11$&$1$&$T_{246}    $\\ \hline
			$2$&$12$&$1$&$U_{12} $&$ 8$&$ 8$&$1$&$P_{8}      $&$14$&$11$&$2$&$T_{245}+A_1$\\ \hline
			$3$&$11$&$1$&$S_{11} $&$ 9$&$13$&$1$&$T_{266}    $&$15$&$11$&$3$&$X_{9}+2A_1 $\\ \hline
			$4$&$11$&$1$&$T_{444}$&$10$&$12$&$1$&$T_{256}    $&$16$&$10$&$1$&$T_{245}    $\\ \hline 
			$5$&$10$&$1$&$Q_{10} $&$11$&$12$&$2$&$T_{246}+A_1$&$17$&$10$&$2$&$X_{9}+A_1  $\\ \hline
			$6$&$10$&$1$&$T_{344}$&$12$&$11$&$1$&$T_{255}    $&$18$&$9 $&$2$&$X_9        $\\\hline
		\end{tabular*}
		
		\bigskip
		
		\caption{Threefolds with a corank $\geq3$ singularity or a corank 2, $j_3=0$ singularity}\label{Ta:ca3list}
	\end{center}
\end{table}

\begin{table}[htb]
	\begin{center}
		\begin{tabular*}{5.90in}{@{\extracolsep{\fill}}|c | c | c| p{0.95in}|| c| c| c| 
				p{0.95in}|| c| c| c| p{0.95in}|}\hline
			$T$ &$\mu$&$k$&type        &$T$ &$\mu$&$k$&type        &$T$&$\mu$&$k$&type  \\\hline\hline
			$19$&$12$&$2$&$J_{10}+A_2 $&$45$&$ 7$&$1$&$D_7         $&$71$&$12$&$3$&$3D_4        $\\ \hline
			$20$&$11$&$2$&$J_{10}+A_1 $&$46$&$11$&$4$&$D_6+A_3+2A_1$&$72$&$11$&$3$&$2D_4+A_3    $\\ \hline
			$21$&$10$&$1$&$J_{10}     $&$47$&$10$&$3$&$D_6+A_3+A_1 $&$73$&$10$&$3$&$2D_4+A_2    $\\ \hline
			$22$&$10$&$2$&$E_{8} +A_2 $&$48$&$ 9$&$2$&$D_6+A_3     $&$74$&$11$&$5$&$2D_4+3A_1   $\\ \hline
			$23$&$ 9$&$2$&$E_{8} +A_1 $&$49$&$10$&$4$&$D_6+A_2+2A_1$&$75$&$10$&$4$&$2D_4+2A_1   $\\ \hline
			$24$&$ 8$&$1$&$E_{8}      $&$50$&$ 9$&$3$&$D_6+A_2+A_1 $&$76$&$ 9$&$3$&$2D_4+A_1    $\\ \hline
			$25$&$10$&$3$&$E_7+A_2+A_1$&$51$&$ 8$&$2$&$D_6+A_2     $&$77$&$ 8$&$2$&$2D_4        $\\ \hline
			$26$&$ 9$&$2$&$E_7+A_2    $&$52$&$10$&$5$&$D_6+4A_1    $&$78$&$10$&$3$&$D_4+2A_3    $\\ \hline
			$27$&$ 9$&$3$&$E_7+2A_1   $&$53$&$ 9$&$4$&$D_6+3A_1    $&$79$&$ 9$&$3$&$D_4+A_3+A_2 $\\ \hline
			$28$&$ 8$&$2$&$E_7+A_1    $&$54$&$ 8$&$3$&$D_6+2A_1    $&$80$&$10$&$5$&$D_4+A_3+3A_1$\\ \hline
			$29$&$ 7$&$1$&$E_7        $&$55$&$ 7$&$2$&$D_6+A_1     $&$81$&$ 9$&$4$&$D_4+A_3+2A_1$\\ \hline
			$30$&$10$&$3$&$E_6+2A_2   $&$56$&$ 6$&$1$&$D_6         $&$82$&$ 8$&$3$&$D_4+A_3+A_1 $\\ \hline
			$31$&$ 9$&$3$&$E_6+A_2+A_1$&$57$&$11$&$3$&$D_5+2A_3    $&$83$&$ 7$&$2$&$D_4+A_3     $\\ \hline
			$32$&$ 8$&$2$&$E_6+A_2    $&$58$&$10$&$3$&$D_5+A_3+A_2 $&$84$&$ 8$&$3$&$D_4+2A_2    $\\ \hline
			$33$&$ 8$&$3$&$E_6+2A_1   $&$59$&$10$&$4$&$D_5+A_3+2A_1$&$85$&$ 9$&$5$&$D_4+A_2+3A_1$\\ \hline
			$34$&$ 7$&$2$&$E_6+A_1    $&$60$&$ 9$&$3$&$D_5+A_3+A_1 $&$86$&$ 8$&$4$&$D_4+A_2+2A_1$\\ \hline
			$35$&$ 6$&$1$&$E_6        $&$61$&$ 8$&$2$&$D_5+A_3     $&$87$&$ 7$&$3$&$D_4+A_2+A_1 $\\ \hline
			$36$&$11$&$2$&$D_8+A_3    $&$62$&$ 9$&$3$&$D_5+2A_2    $&$88$&$ 6$&$2$&$D_4+A_2     $\\ \hline
			$37$&$10$&$2$&$D_8+A_2    $&$63$&$ 9$&$4$&$D_5+A_2+2A_1$&$89$&$10$&$7$&$D_4+6A_1    $\\ \hline
			$38$&$10$&$3$&$D_8+2A_1   $&$64$&$ 8$&$3$&$D_5+A_2+A_1 $&$90$&$ 9$&$6$&$D_4+5A_1    $\\ \hline
			$39$&$ 9$&$2$&$D_8+A_1    $&$65$&$ 7$&$2$&$D_5+A_2     $&$91$&$ 8$&$5$&$D_4+4A_1    $\\ \hline
			$40$&$ 8$&$2$&$D_8        $&$66$&$ 9$&$5$&$D_5+4A_1    $&$92$&$ 7$&$4$&$D_4+3A_1    $\\ \hline
			$41$&$10$&$2$&$D_7+A_3    $&$67$&$ 8$&$4$&$D_5+3A_1    $&$93$&$ 6$&$3$&$D_4+2A_1    $\\ \hline
			$42$&$ 9$&$2$&$D_7+A_2    $&$68$&$ 7$&$3$&$D_5+2A_1    $&$94$&$ 5$&$2$&$D_4+A_1     $\\ \hline
			$43$&$ 9$&$3$&$D_7+2A_1   $&$69$&$ 6$&$2$&$D_5+A_1     $&$95$&$ 4$&$1$&$D_4         $\\ \hline
			$44$&$ 8$&$2$&$D_7+A_1    $&$70$&$ 5$&$1$&$D_5         $&\multicolumn{4}{c}{}\\ \cline{1-8}
		\end{tabular*}
		
		\bigskip
		
		\caption{Threefolds with a corank 2, $j_3\neq 0$ singularity} \label{Ta:ca2list}
	\end{center}
\end{table}

\begin{table}[htb]
	\begin{center}
		\begin{tabular*}{5.90in}{@{\extracolsep{\fill}}|c | c | c| p{0.95in}|| c| c| c| p{0.95in}|| c| c| c| p{0.95in}|}\hline
			$T$ &$\mu$&$k$&type          &$T$ &$\mu$&$k$&type           &$T$&$\mu$&$k$&type            \\\hline\hline
			$ 96$&$11$&$1$&$A_{11}      $&$133$&$ 5$&$1$&$A_5          $&$170$&$ 6$&$2$&$2A_3         $\\ \hline
			$ 97$&$10$&$1$&$A_{10}      $&$134$&$10$&$3$&$2A_4+A_2     $&$171$&$ 6$&$3$&$A_3+A_2+A_1  $\\ \hline
			$ 98$&$10$&$2$&$A_9+A_1     $&$135$&$10$&$3$&$A_4+2A_3     $&$172$&$ 6$&$4$&$A_3+3A_1     $\\ \hline
			$ 99$&$ 9$&$1$&$A_9         $&$136$&$ 9$&$3$&$2A_4+A_1     $&$173$&$ 5$&$2$&$A_3+A_2      $\\ \hline
			$100$&$10$&$2$&$A_{8} +A_2  $&$137$&$ 9$&$3$&$A_4+A_3+A_2  $&$174$&$ 5$&$3$&$A_3+2A_1     $\\ \hline
			$101$&$ 9$&$2$&$A_{8}+A_1   $&$138$&$ 9$&$5$&$A_4+2A_2+A_1 $&$175$&$ 4$&$2$&$A_3+A_1      $\\ \hline
			$102$&$ 8$&$1$&$A_8         $&$139$&$ 9$&$5$&$A_4+A_3+2A_1 $&$176$&$ 3$&$1$&$A_3          $\\ \hline
			$103$&$11$&$2$&$A_7+A_4     $&$140$&$ 8$&$2$&$2A_4         $&$177$&$10$&$5$&$5A_2         $\\ \hline
			$104$&$10$&$2$&$A_7+A_3     $&$141$&$ 8$&$3$&$A_4+A_3+A_1  $&$178$&$ 9$&$5$&$4A_2+A_1     $\\ \hline
			$105$&$10$&$3$&$A_7+A_2+A_1 $&$142$&$ 8$&$3$&$A_4+2A_2     $&$179$&$ 8$&$4$&$4A_2         $\\ \hline
			$106$&$ 9$&$2$&$A_7+A_2     $&$143$&$ 8$&$4$&$A_4+A_2+2A_1 $&$180$&$ 8$&$5$&$3A_2+2A_1    $\\ \hline
			$107$&$ 9$&$3$&$A_7+2A_1    $&$144$&$ 8$&$5$&$A_4+4A_1     $&$181$&$ 8$&$6$&$2A_2+4A_1    $\\ \hline
			$108$&$ 8$&$2$&$A_7+A_1     $&$145$&$ 7$&$2$&$A_4+A_3      $&$182$&$ 8$&$7$&$A_2+6A_1     $\\ \hline
			$109$&$ 7$&$1$&$A_7         $&$146$&$ 7$&$3$&$A_4+ A_2+A_1 $&$183$&$ 7$&$4$&$3A_2+A_1     $\\ \hline
			$110$&$10$&$2$&$A_6+A_4     $&$147$&$ 7$&$4$&$A_4+3A_1     $&$184$&$ 7$&$5$&$2A_2+3A_1    $\\ \hline
			$111$&$ 9$&$2$&$A_6+A_3     $&$148$&$ 6$&$2$&$A_4+A_2      $&$185$&$ 7$&$6$&$A_2+5A_1     $\\ \hline
			$112$&$ 9$&$3$&$A_6+A_2+A_1 $&$149$&$ 6$&$3$&$A_4+2A_1     $&$186$&$ 6$&$3$&$3A_2         $\\ \hline
			$113$&$ 8$&$2$&$A_6+A_2     $&$150$&$ 5$&$2$&$A_4+A_1      $&$187$&$ 6$&$2$&$2A_2+2A_1    $\\ \hline
			$114$&$ 8$&$3$&$A_6+2A_1    $&$151$&$ 4$&$1$&$A_4          $&$188$&$ 6$&$5$&$A_2+4A_1     $\\ \hline
			$115$&$ 7$&$2$&$A_6+A_1     $&$152$&$11$&$4$&$3A_3+A_1     $&$189$&$ 5$&$3$&$2A_2+A_1     $\\ \hline
			$116$&$ 6$&$1$&$A_6         $&$153$&$10$&$5$&$2A_3+A_2+2A_1$&$190$&$ 5$&$4$&$A_2+3A_1     $\\ \hline
			$117$&$11$&$3$&$2A_5+A_1    $&$154$&$10$&$6$&$2A_3+4A_1    $&$191$&$ 4$&$2$&$2A_2         $\\ \hline
			$118$&$10$&$2$&$2A_5        $&$155$&$ 9$&$3$&$3A_3         $&$192$&$ 4$&$3$&$A_2+2A_1     $\\ \hline
			$119$&$10$&$3$&$A_5+A_4+ A_1$&$156$&$ 9$&$4$&$2A_3+A_2+A_1 $&$193$&$ 3$&$2$&$A_2+A_1      $\\ \hline
			$120$&$10$&$4$&$A_5+A_3+2A_1$&$157$&$ 9$&$5$&$A_3+2A_2+2A_1$&$193$&$ 3$&$2$&$A_2+A_1      $\\ \hline
			$121$&$10$&$4$&$A_5+2A_2+A_1$&$158$&$ 9$&$5$&$2A_3+3A_1    $&$194$&$ 2$&$1$&$A_2          $\\ \hline
			$122$&$ 9$&$2$&$A_5+A_4     $&$159$&$ 9$&$6$&$A_3+A_2+4A_1 $&$195$&$10$&$10$&$10A_1       $\\ \hline
			$123$&$ 9$&$3$&$A_5+A_3+A_1 $&$160$&$ 9$&$7$&$A_3+6A_1     $&$196$&$ 9$&$ 9$&$9A_1        $\\ \hline
			$124$&$ 9$&$3$&$A_5+2A_2    $&$161$&$ 8$&$3$&$2A_3+A_2     $&$197$&$ 8$&$ 8$&$8A_1        $\\ \hline
			$125$&$ 9$&$4$&$A_5+A_2+2A_1$&$162$&$ 8$&$4$&$A_3+2A_2+A_1 $&$198$&$ 7$&$ 7$&$7A_1        $\\ \hline
			$126$&$ 9$&$5$&$A_5+4A_1    $&$163$&$ 8$&$4$&$2A_3+2A_1    $&$199$&$ 6$&$ 6$&$6A_1        $\\ \hline
			$127$&$ 8$&$2$&$A_5+ A_3    $&$164$&$ 8$&$5$&$A_3+A_2+3A_1 $&$200$&$ 5$&$ 5$&$5A_1        $\\ \hline
			$128$&$ 8$&$3$&$A_5+A_2+A_1 $&$165$&$ 8$&$6$&$A_3+5A_1     $&$201$&$ 4$&$ 4$&$4A_1        $\\ \hline
			$129$&$ 8$&$4$&$A_5+3A_1    $&$166$&$ 7$&$3$&$2A_3+A_1     $&$202$&$ 3$&$ 3$&$3A_1        $\\ \hline
			$130$&$ 7$&$2$&$A_5+A_2     $&$167$&$ 7$&$3$&$A_3+2A_2     $&$203$&$ 2$&$ 2$&$2A_1        $\\ \hline
			$131$&$ 7$&$3$&$A_5+2A_1    $&$168$&$ 7$&$4$&$A_3+A_2+2A_1 $&$204$&$ 1$&$ 1$&$A_1         $\\ \hline
			$132$&$ 6$&$2$&$A_5+A_1     $&$169$&$ 7$&$5$&$A_3+4A_1     $&\multicolumn{4}{c}{}          \\ 
			\cline{1-8}
		\end{tabular*}
		
		\bigskip
		
		\caption{Threefolds with constellations of $A_n$ singularities}\label{Ta:anlist}
	\end{center}
\end{table}

\clearpage

\bibliography{3folds.bib}

\end{document}

%% file: Gamma.tex
\tikzset{every picture/.style={line width=0.75pt}} 

\begin{tikzpicture}[x=0.75pt,y=0.75pt,yscale=-1,xscale=1]
	
	\draw [color={rgb, 255:red, 60; green, 211; blue, 33 }  ,draw opacity=1 ]   (145.56,141.68) -- (255.93,92.32) ;
	\draw [shift={(255.93,92.32)}, rotate = 335.91] [color={rgb, 255:red, 60; green, 211; blue, 33 }  ,draw opacity=1 ][fill={rgb, 255:red, 60; green, 211; blue, 33 }  ,fill opacity=1 ][line width=0.75]      (0, 0) circle [x radius= 3.35, y radius= 3.35]   ;
	\draw [shift={(200.75,117)}, rotate = 335.91] [color={rgb, 255:red, 60; green, 211; blue, 33 }  ,draw opacity=1 ][fill={rgb, 255:red, 60; green, 211; blue, 33 }  ,fill opacity=1 ][line width=0.75]      (0, 0) circle [x radius= 3.35, y radius= 3.35]   ;
	\draw [color={rgb, 255:red, 60; green, 211; blue, 33 }  ,draw opacity=1 ]   (145.23,21.48) .. controls (276.84,-42.9) and (305.45,178.1) .. (145.89,261.87) ;
	\draw [shift={(145.89,261.87)}, rotate = 152.3] [color={rgb, 255:red, 60; green, 211; blue, 33 }  ,draw opacity=1 ][fill={rgb, 255:red, 60; green, 211; blue, 33 }  ,fill opacity=1 ][line width=0.75]      (0, 0) circle [x radius= 3.35, y radius= 3.35]   ;
	\draw [shift={(145.23,21.48)}, rotate = 333.93] [color={rgb, 255:red, 60; green, 211; blue, 33 }  ,draw opacity=1 ][fill={rgb, 255:red, 60; green, 211; blue, 33 }  ,fill opacity=1 ][line width=0.75]      (0, 0) circle [x radius= 3.35, y radius= 3.35]   ;
	\draw [color={rgb, 255:red, 0; green, 0; blue, 0 }  ,draw opacity=1 ]   (145.89,261.87) .. controls (-47.79,217.82) and (-5.98,-6.18) .. (145.23,21.48) ;
	\draw [shift={(145.23,21.48)}, rotate = 10.37] [color={rgb, 255:red, 0; green, 0; blue, 0 }  ,draw opacity=1 ][fill={rgb, 255:red, 0; green, 0; blue, 0 }  ,fill opacity=1 ][line width=0.75]      (0, 0) circle [x radius= 3.35, y radius= 3.35]   ;
	\draw [shift={(145.89,261.87)}, rotate = 192.81] [color={rgb, 255:red, 0; green, 0; blue, 0 }  ,draw opacity=1 ][fill={rgb, 255:red, 0; green, 0; blue, 0 }  ,fill opacity=1 ][line width=0.75]      (0, 0) circle [x radius= 3.35, y radius= 3.35]   ;
	\draw [color={rgb, 255:red, 0; green, 0; blue, 0 }  ,draw opacity=1 ]   (54.55,41.72) ;
	\draw [shift={(54.55,41.72)}, rotate = 0] [color={rgb, 255:red, 0; green, 0; blue, 0 }  ,draw opacity=1 ][fill={rgb, 255:red, 0; green, 0; blue, 0 }  ,fill opacity=1 ][line width=0.75]      (0, 0) circle [x radius= 3.35, y radius= 3.35]   ;
	\draw [color={rgb, 255:red, 60; green, 211; blue, 33 }  ,draw opacity=1 ]   (216.76,16.67) ;
	\draw [shift={(216.76,16.67)}, rotate = 0] [color={rgb, 255:red, 60; green, 211; blue, 33 }  ,draw opacity=1 ][fill={rgb, 255:red, 60; green, 211; blue, 33 }  ,fill opacity=1 ][line width=0.75]      (0, 0) circle [x radius= 3.35, y radius= 3.35]   ;
	\draw [color={rgb, 255:red, 208; green, 2; blue, 27 }  ,draw opacity=1 ]   (177.8,90.32) ;
	\draw [shift={(177.8,90.32)}, rotate = 0] [color={rgb, 255:red, 208; green, 2; blue, 27 }  ,draw opacity=1 ][fill={rgb, 255:red, 208; green, 2; blue, 27 }  ,fill opacity=1 ][line width=0.75]      (0, 0) circle [x radius= 3.35, y radius= 3.35]   ;
	\draw [color={rgb, 255:red, 208; green, 2; blue, 27 }  ,draw opacity=1 ]   (188.15,264.67) ;
	\draw [shift={(188.15,264.67)}, rotate = 0] [color={rgb, 255:red, 208; green, 2; blue, 27 }  ,draw opacity=1 ][fill={rgb, 255:red, 208; green, 2; blue, 27 }  ,fill opacity=1 ][line width=0.75]      (0, 0) circle [x radius= 3.35, y radius= 3.35]   ;
	\draw [color={rgb, 255:red, 60; green, 211; blue, 33 }  ,draw opacity=1 ]   (237.45,168.67) ;
	\draw [shift={(237.45,168.67)}, rotate = 0] [color={rgb, 255:red, 60; green, 211; blue, 33 }  ,draw opacity=1 ][fill={rgb, 255:red, 60; green, 211; blue, 33 }  ,fill opacity=1 ][line width=0.75]      (0, 0) circle [x radius= 3.35, y radius= 3.35]   ;
	\draw [color={rgb, 255:red, 0; green, 0; blue, 0 }  ,draw opacity=1 ]   (53.01,213.67) ;
	\draw [shift={(53.01,213.67)}, rotate = 0] [color={rgb, 255:red, 0; green, 0; blue, 0 }  ,draw opacity=1 ][fill={rgb, 255:red, 0; green, 0; blue, 0 }  ,fill opacity=1 ][line width=0.75]      (0, 0) circle [x radius= 3.35, y radius= 3.35]   ;
	\draw    (145.89,261.87) ;
	\draw [shift={(145.89,261.87)}, rotate = 0] [color={rgb, 255:red, 0; green, 0; blue, 0 }  ][fill={rgb, 255:red, 0; green, 0; blue, 0 }  ][line width=0.75]      (0, 0) circle [x radius= 3.35, y radius= 3.35]   ;
	\draw [color={rgb, 255:red, 0; green, 0; blue, 0 }  ,draw opacity=1 ]   (14.93,129.72) ;
	\draw [shift={(14.93,129.72)}, rotate = 0] [color={rgb, 255:red, 0; green, 0; blue, 0 }  ,draw opacity=1 ][fill={rgb, 255:red, 0; green, 0; blue, 0 }  ,fill opacity=1 ][line width=0.75]      (0, 0) circle [x radius= 3.35, y radius= 3.35]   ;
	\draw    (145.23,21.48) ;
	\draw [shift={(145.23,21.48)}, rotate = 0] [color={rgb, 255:red, 0; green, 0; blue, 0 }  ][fill={rgb, 255:red, 0; green, 0; blue, 0 }  ][line width=0.75]      (0, 0) circle [x radius= 3.35, y radius= 3.35]   ;
	\draw [color={rgb, 255:red, 60; green, 211; blue, 33 }  ,draw opacity=1 ]   (255.93,92.32) ;
	\draw [shift={(255.93,92.32)}, rotate = 0] [color={rgb, 255:red, 60; green, 211; blue, 33 }  ,draw opacity=1 ][fill={rgb, 255:red, 60; green, 211; blue, 33 }  ,fill opacity=1 ][line width=0.75]      (0, 0) circle [x radius= 3.35, y radius= 3.35]   ;
	\draw [color={rgb, 255:red, 60; green, 211; blue, 33 }  ,draw opacity=1 ]   (200.75,117) ;
	\draw [shift={(200.75,117)}, rotate = 0] [color={rgb, 255:red, 60; green, 211; blue, 33 }  ,draw opacity=1 ][fill={rgb, 255:red, 60; green, 211; blue, 33 }  ,fill opacity=1 ][line width=0.75]      (0, 0) circle [x radius= 3.35, y radius= 3.35]   ;
	\draw [color={rgb, 255:red, 208; green, 2; blue, 27 }  ,draw opacity=1 ]   (199.81,202.32) ;
	\draw [shift={(199.81,202.32)}, rotate = 0] [color={rgb, 255:red, 208; green, 2; blue, 27 }  ,draw opacity=1 ][fill={rgb, 255:red, 208; green, 2; blue, 27 }  ,fill opacity=1 ][line width=0.75]      (0, 0) circle [x radius= 3.35, y radius= 3.35]   ;
	\draw [color={rgb, 255:red, 255; green, 255; blue, 255 }  ,draw opacity=1 ]   (187,123) ;
	\draw [shift={(187,123)}, rotate = 0] [color={rgb, 255:red, 255; green, 255; blue, 255 }  ,draw opacity=1 ][fill={rgb, 255:red, 255; green, 255; blue, 255 }  ,fill opacity=1 ][line width=0.75]      (0, 0) circle [x radius= 3.35, y radius= 3.35]   ;
	\draw [color={rgb, 255:red, 255; green, 255; blue, 255 }  ,draw opacity=1 ]   (200,223) ;
	\draw [shift={(200,223)}, rotate = 0] [color={rgb, 255:red, 255; green, 255; blue, 255 }  ,draw opacity=1 ][fill={rgb, 255:red, 255; green, 255; blue, 255 }  ,fill opacity=1 ][line width=0.75]      (0, 0) circle [x radius= 3.35, y radius= 3.35]   ;
	\draw [color={rgb, 255:red, 208; green, 2; blue, 27 }  ,draw opacity=1 ]   (145.23,21.48) .. controls (158.77,31.5) and (175.92,75.24) .. (187.29,123.44) .. controls (206.96,206.79) and (209.34,303.44) .. (145.89,261.87) ;
	\draw [color={rgb, 255:red, 0; green, 0; blue, 0 }  ,draw opacity=1 ]   (14.93,129.72) -- (145.56,141.68) ;
	\draw [shift={(145.56,141.68)}, rotate = 5.23] [color={rgb, 255:red, 0; green, 0; blue, 0 }  ,draw opacity=1 ][fill={rgb, 255:red, 0; green, 0; blue, 0 }  ,fill opacity=1 ][line width=0.75]      (0, 0) circle [x radius= 3.35, y radius= 3.35]   ;
	\draw [shift={(80.25,135.7)}, rotate = 5.23] [color={rgb, 255:red, 0; green, 0; blue, 0 }  ,draw opacity=1 ][fill={rgb, 255:red, 0; green, 0; blue, 0 }  ,fill opacity=1 ][line width=0.75]      (0, 0) circle [x radius= 3.35, y radius= 3.35]   ;
	\draw [shift={(14.93,129.72)}, rotate = 5.23] [color={rgb, 255:red, 0; green, 0; blue, 0 }  ,draw opacity=1 ][fill={rgb, 255:red, 0; green, 0; blue, 0 }  ,fill opacity=1 ][line width=0.75]      (0, 0) circle [x radius= 3.35, y radius= 3.35]   ;
	\draw [color={rgb, 255:red, 208; green, 2; blue, 27 }  ,draw opacity=1 ]   (145.56,141.68) -- (199.81,202.32) ;
	\draw [shift={(199.81,202.32)}, rotate = 48.18] [color={rgb, 255:red, 208; green, 2; blue, 27 }  ,draw opacity=1 ][fill={rgb, 255:red, 208; green, 2; blue, 27 }  ,fill opacity=1 ][line width=0.75]      (0, 0) circle [x radius= 3.35, y radius= 3.35]   ;
	\draw [shift={(172.68,172)}, rotate = 48.18] [color={rgb, 255:red, 208; green, 2; blue, 27 }  ,draw opacity=1 ][fill={rgb, 255:red, 208; green, 2; blue, 27 }  ,fill opacity=1 ][line width=0.75]      (0, 0) circle [x radius= 3.35, y radius= 3.35]   ;
\end{tikzpicture}

%% file: E_6_tilde.tex
\tikzset{every picture/.style={line width=0.75pt}} 

\begin{tikzpicture}[x=0.75pt,y=0.75pt,yscale=-1,xscale=1]
	
	\draw  [dash pattern={on 0.84pt off 2.51pt}]  (549.8,11.5) -- (549.6,130.8) ;
	\draw [shift={(549.6,130.8)}, rotate = 90.1] [color={rgb, 255:red, 0; green, 0; blue, 0 }  ][fill={rgb, 255:red, 0; green, 0; blue, 0 }  ][line width=0.75]      (0, 0) circle [x radius= 3.35, y radius= 3.35]   ;
	\draw  [dash pattern={on 0.84pt off 2.51pt}]  (430,11.2) -- (669.6,11.8) ;
	\draw [shift={(669.6,11.8)}, rotate = 0.14] [color={rgb, 255:red, 0; green, 0; blue, 0 }  ][fill={rgb, 255:red, 0; green, 0; blue, 0 }  ][line width=0.75]      (0, 0) circle [x radius= 3.35, y radius= 3.35]   ;
	\draw [shift={(549.8,11.5)}, rotate = 0.14] [color={rgb, 255:red, 0; green, 0; blue, 0 }  ][fill={rgb, 255:red, 0; green, 0; blue, 0 }  ][line width=0.75]      (0, 0) circle [x radius= 3.35, y radius= 3.35]   ;
	\draw [shift={(430,11.2)}, rotate = 0.14] [color={rgb, 255:red, 0; green, 0; blue, 0 }  ][fill={rgb, 255:red, 0; green, 0; blue, 0 }  ][line width=0.75]      (0, 0) circle [x radius= 3.35, y radius= 3.35]   ;
	\draw  [color={rgb, 255:red, 255; green, 255; blue, 255 }  ,draw opacity=1 ][line width=3] [line join = round][line cap = round] (549.6,67.6) .. controls (549.6,71.74) and (551.21,70.6) .. (555.6,70.6) .. controls (558.27,70.6) and (550.6,69.6) .. (552.6,71.6) .. controls (554.93,73.93) and (553.6,62.6) .. (553.6,67.6) .. controls (553.6,69.27) and (553.14,71) .. (553.6,72.6) .. controls (554.06,74.2) and (556.85,77.09) .. (557.6,75.6) .. controls (558.86,73.07) and (555.96,70.6) .. (554.6,70.6) ;
	\draw  [color={rgb, 255:red, 255; green, 255; blue, 255 }  ,draw opacity=1 ][line width=3] [line join = round][line cap = round] (605.6,7.6) .. controls (605.6,11.74) and (607.21,10.6) .. (611.6,10.6) .. controls (614.27,10.6) and (606.6,9.6) .. (608.6,11.6) .. controls (610.93,13.93) and (609.6,2.6) .. (609.6,7.6) .. controls (609.6,9.27) and (609.14,11) .. (609.6,12.6) .. controls (610.06,14.2) and (612.85,17.09) .. (613.6,15.6) .. controls (614.86,13.07) and (611.96,10.6) .. (610.6,10.6) ;
	\draw  [color={rgb, 255:red, 255; green, 255; blue, 255 }  ,draw opacity=1 ][line width=3] [line join = round][line cap = round] (486.6,9.6) .. controls (486.6,13.74) and (488.21,12.6) .. (492.6,12.6) .. controls (495.27,12.6) and (487.6,11.6) .. (489.6,13.6) .. controls (491.93,15.93) and (490.6,4.6) .. (490.6,9.6) .. controls (490.6,11.27) and (490.14,13) .. (490.6,14.6) .. controls (491.06,16.2) and (493.85,19.09) .. (494.6,17.6) .. controls (495.86,15.07) and (492.96,12.6) .. (491.6,12.6) ;
	\draw    (52,11) -- (291.6,11.6) ;
	\draw [shift={(291.6,11.6)}, rotate = 0.14] [color={rgb, 255:red, 0; green, 0; blue, 0 }  ][fill={rgb, 255:red, 0; green, 0; blue, 0 }  ][line width=0.75]      (0, 0) circle [x radius= 3.35, y radius= 3.35]   ;
	\draw [shift={(171.8,11.3)}, rotate = 0.14] [color={rgb, 255:red, 0; green, 0; blue, 0 }  ][fill={rgb, 255:red, 0; green, 0; blue, 0 }  ][line width=0.75]      (0, 0) circle [x radius= 3.35, y radius= 3.35]   ;
	\draw [shift={(52,11)}, rotate = 0.14] [color={rgb, 255:red, 0; green, 0; blue, 0 }  ][fill={rgb, 255:red, 0; green, 0; blue, 0 }  ][line width=0.75]      (0, 0) circle [x radius= 3.35, y radius= 3.35]   ;
	\draw    (112,11) ;
	\draw [shift={(112,11)}, rotate = 0] [color={rgb, 255:red, 0; green, 0; blue, 0 }  ][fill={rgb, 255:red, 0; green, 0; blue, 0 }  ][line width=0.75]      (0, 0) circle [x radius= 3.35, y radius= 3.35]   ;
	\draw    (232,11) ;
	\draw [shift={(232,11)}, rotate = 0] [color={rgb, 255:red, 0; green, 0; blue, 0 }  ][fill={rgb, 255:red, 0; green, 0; blue, 0 }  ][line width=0.75]      (0, 0) circle [x radius= 3.35, y radius= 3.35]   ;
	\draw    (171.8,11.3) -- (171.6,130.6) ;
	\draw [shift={(171.6,130.6)}, rotate = 90.1] [color={rgb, 255:red, 0; green, 0; blue, 0 }  ][fill={rgb, 255:red, 0; green, 0; blue, 0 }  ][line width=0.75]      (0, 0) circle [x radius= 3.35, y radius= 3.35]   ;
	\draw [shift={(171.7,70.95)}, rotate = 90.1] [color={rgb, 255:red, 0; green, 0; blue, 0 }  ][fill={rgb, 255:red, 0; green, 0; blue, 0 }  ][line width=0.75]      (0, 0) circle [x radius= 3.35, y radius= 3.35]   ;
	\draw    (490,11.2) ;
	\draw [shift={(490,11.2)}, rotate = 0] [color={rgb, 255:red, 0; green, 0; blue, 0 }  ][line width=0.75]      (0, 0) circle [x radius= 3.35, y radius= 3.35]   ;
	\draw    (610,11.2) ;
	\draw [shift={(610,11.2)}, rotate = 0] [color={rgb, 255:red, 0; green, 0; blue, 0 }  ][line width=0.75]      (0, 0) circle [x radius= 3.35, y radius= 3.35]   ;
	\draw    (330,70.4) -- (387.1,70.59) ;
	\draw [shift={(390.1,70.6)}, rotate = 180.19] [fill={rgb, 255:red, 0; green, 0; blue, 0 }  ][line width=0.08]  [draw opacity=0] (10.72,-5.15) -- (0,0) -- (10.72,5.15) -- (7.12,0) -- cycle    ;
	\draw    (549.7,71.15) ;
	\draw [shift={(549.7,71.15)}, rotate = 0] [color={rgb, 255:red, 0; green, 0; blue, 0 }  ][line width=0.75]      (0, 0) circle [x radius= 3.35, y radius= 3.35]   ;
	
\end{tikzpicture}

%% file: X9.tex
\tikzset{every picture/.style={line width=0.75pt}} 

\begin{tikzpicture}[x=0.75pt,y=0.75pt,yscale=-1,xscale=1]
	
	\draw   (10,10.6) -- (207.6,10.6) -- (207.6,124) -- (10,124) -- cycle ;
	\draw   (10,124) -- (207.59,124) -- (278.53,185.6) -- (80.94,185.6) -- cycle ;
	\draw    (98,42) -- (184.6,123.6) ;
	\draw    (100,124) -- (149.6,41.6) ;
	\draw    (56,124) -- (161.6,170.6) ;
	\draw    (144.6,123.6) -- (97.6,170.6) ;
	\draw    (131,73) ;
	\draw [shift={(131,73)}, rotate = 0] [color={rgb, 255:red, 0; green, 0; blue, 0 }  ][fill={rgb, 255:red, 0; green, 0; blue, 0 }  ][line width=0.75]      (0, 0) circle [x radius= 3.35, y radius= 3.35]   ;
	\draw    (117,151) ;
	\draw [shift={(117,151)}, rotate = 0] [color={rgb, 255:red, 0; green, 0; blue, 0 }  ][fill={rgb, 255:red, 0; green, 0; blue, 0 }  ][line width=0.75]      (0, 0) circle [x radius= 3.35, y radius= 3.35]   ;
	\draw    (427.6,123.6) .. controls (447.2,14.2) and (491.6,16.6) .. (511.6,123.6) ;
	\draw    (383,124) .. controls (501.6,219.6) and (470.6,146.6) .. (471.6,123.6) ;
	\draw   (337,10.6) -- (534.6,10.6) -- (534.6,124) -- (337,124) -- cycle ;
	\draw   (337,124) -- (534.59,124) -- (605.53,185.6) -- (407.94,185.6) -- cycle ;
	
	\draw (183,19.4) node [anchor=north west][inner sep=0.75pt]    {$P_{1}$};
	\draw (229,164.4) node [anchor=north west][inner sep=0.75pt]    {$P_{2}$};
	\draw (24,100.4) node [anchor=north west][inner sep=0.75pt]    {$L$};
	\draw (69,32.4) node [anchor=north west][inner sep=0.75pt]    {$C_{1}$};
	\draw (166,142.4) node [anchor=north west][inner sep=0.75pt]    {$C_{2}$};
	\draw (102,64.4) node [anchor=north west][inner sep=0.75pt]    {$q_{1}$};
	\draw (118,160.4) node [anchor=north west][inner sep=0.75pt]    {$q_{2}$};
	\draw (510,19.4) node [anchor=north west][inner sep=0.75pt]    {$P_{1}$};
	\draw (556,164.4) node [anchor=north west][inner sep=0.75pt]    {$P_{2}$};
	\draw (351,100.4) node [anchor=north west][inner sep=0.75pt]    {$L$};
	\draw (416,32.4) node [anchor=north west][inner sep=0.75pt]    {$C_{1}$};
	\draw (491,142.4) node [anchor=north west][inner sep=0.75pt]    {$C_{2}$};

\end{tikzpicture}

%% file: D4curve.tex
\tikzset{every picture/.style={line width=0.75pt}} 

\begin{tikzpicture}[x=0.75pt,y=0.75pt,yscale=-1,xscale=1]
	
	\draw   (10,10.6) -- (207.6,10.6) -- (207.6,124) -- (10,124) -- cycle ;
	\draw   (10,124) -- (207.59,124) -- (278.53,185.6) -- (80.94,185.6) -- cycle ;
	\draw    (81.6,45.6) -- (186.6,124.6) ;
	\draw    (56,124) -- (160.6,46.6) ;
	\draw    (56,124) -- (191.6,169.6) ;
	\draw    (186.6,124.6) -- (126.6,172.6) ;
	\draw    (120.8,75.3) ;
	\draw [shift={(120.8,75.3)}, rotate = 0] [color={rgb, 255:red, 0; green, 0; blue, 0 }  ][fill={rgb, 255:red, 0; green, 0; blue, 0 }  ][line width=0.75]      (0, 0) circle [x radius= 3.35, y radius= 3.35]   ;
	\draw    (149.1,155.1) ;
	\draw [shift={(149.1,155.1)}, rotate = 0] [color={rgb, 255:red, 0; green, 0; blue, 0 }  ][fill={rgb, 255:red, 0; green, 0; blue, 0 }  ][line width=0.75]      (0, 0) circle [x radius= 3.35, y radius= 3.35]   ;
	\draw    (123.6,41.6) -- (117.6,123.6) ;
	\draw    (117.6,123.6) -- (164.6,171.6) ;
	
	\draw (183,19.4) node [anchor=north west][inner sep=0.75pt]    {$P_{1}$};
	\draw (229,164.4) node [anchor=north west][inner sep=0.75pt]    {$P_{2}$};
	\draw (24,100.4) node [anchor=north west][inner sep=0.75pt]    {$L$};
	\draw (37,32.4) node [anchor=north west][inner sep=0.75pt]    {$C_{1}$};
	\draw (84,156.4) node [anchor=north west][inner sep=0.75pt]    {$C_{2}$};
	\draw (137,66.4) node [anchor=north west][inner sep=0.75pt]    {$q_{1}$};
	\draw (170,141.4) node [anchor=north west][inner sep=0.75pt]    {$q_{2}$};

\end{tikzpicture}

%% file: 3D4.tex
\tikzset{every picture/.style={line width=0.75pt}} 

\begin{tikzpicture}[x=0.75pt,y=0.75pt,yscale=-1,xscale=1]
	
	\draw    (57,52) -- (296.6,52.6) ;
	\draw [shift={(296.6,52.6)}, rotate = 0.14] [color={rgb, 255:red, 0; green, 0; blue, 0 }  ][fill={rgb, 255:red, 0; green, 0; blue, 0 }  ][line width=0.75]      (0, 0) circle [x radius= 3.35, y radius= 3.35]   ;
	\draw [shift={(176.8,52.3)}, rotate = 0.14] [color={rgb, 255:red, 0; green, 0; blue, 0 }  ][fill={rgb, 255:red, 0; green, 0; blue, 0 }  ][line width=0.75]      (0, 0) circle [x radius= 3.35, y radius= 3.35]   ;
	\draw [shift={(57,52)}, rotate = 0.14] [color={rgb, 255:red, 0; green, 0; blue, 0 }  ][fill={rgb, 255:red, 0; green, 0; blue, 0 }  ][line width=0.75]      (0, 0) circle [x radius= 3.35, y radius= 3.35]   ;
	\draw    (117,52) ;
	\draw [shift={(117,52)}, rotate = 0] [color={rgb, 255:red, 0; green, 0; blue, 0 }  ][fill={rgb, 255:red, 0; green, 0; blue, 0 }  ][line width=0.75]      (0, 0) circle [x radius= 3.35, y radius= 3.35]   ;
	\draw    (237,52) ;
	\draw [shift={(237,52)}, rotate = 0] [color={rgb, 255:red, 0; green, 0; blue, 0 }  ][fill={rgb, 255:red, 0; green, 0; blue, 0 }  ][line width=0.75]      (0, 0) circle [x radius= 3.35, y radius= 3.35]   ;
	\draw    (176.8,52.3) -- (176.6,171.6) ;
	\draw [shift={(176.6,171.6)}, rotate = 90.1] [color={rgb, 255:red, 0; green, 0; blue, 0 }  ][fill={rgb, 255:red, 0; green, 0; blue, 0 }  ][line width=0.75]      (0, 0) circle [x radius= 3.35, y radius= 3.35]   ;
	\draw [shift={(176.7,111.95)}, rotate = 90.1] [color={rgb, 255:red, 0; green, 0; blue, 0 }  ][fill={rgb, 255:red, 0; green, 0; blue, 0 }  ][line width=0.75]      (0, 0) circle [x radius= 3.35, y radius= 3.35]   ;
	\draw    (57,52) -- (13.1,12.6) ;
	\draw [shift={(13.1,12.6)}, rotate = 221.91] [color={rgb, 255:red, 0; green, 0; blue, 0 }  ][fill={rgb, 255:red, 0; green, 0; blue, 0 }  ][line width=0.75]      (0, 0) circle [x radius= 3.35, y radius= 3.35]   ;
	\draw    (57,52) -- (13.1,91.4) ;
	\draw [shift={(13.1,91.4)}, rotate = 138.09] [color={rgb, 255:red, 0; green, 0; blue, 0 }  ][fill={rgb, 255:red, 0; green, 0; blue, 0 }  ][line width=0.75]      (0, 0) circle [x radius= 3.35, y radius= 3.35]   ;
	\draw    (176.6,171.6) -- (137.2,215.5) ;
	\draw [shift={(137.2,215.5)}, rotate = 131.91] [color={rgb, 255:red, 0; green, 0; blue, 0 }  ][fill={rgb, 255:red, 0; green, 0; blue, 0 }  ][line width=0.75]      (0, 0) circle [x radius= 3.35, y radius= 3.35]   ;
	\draw    (176.6,169.6) -- (216,213.5) ;
	\draw [shift={(216,213.5)}, rotate = 48.09] [color={rgb, 255:red, 0; green, 0; blue, 0 }  ][fill={rgb, 255:red, 0; green, 0; blue, 0 }  ][line width=0.75]      (0, 0) circle [x radius= 3.35, y radius= 3.35]   ;
	\draw    (296.6,52.6) -- (340.5,92) ;
	\draw [shift={(340.5,92)}, rotate = 41.91] [color={rgb, 255:red, 0; green, 0; blue, 0 }  ][fill={rgb, 255:red, 0; green, 0; blue, 0 }  ][line width=0.75]      (0, 0) circle [x radius= 3.35, y radius= 3.35]   ;
	\draw    (296.6,52.6) -- (340.5,13.2) ;
	\draw [shift={(340.5,13.2)}, rotate = 318.09] [color={rgb, 255:red, 0; green, 0; blue, 0 }  ][fill={rgb, 255:red, 0; green, 0; blue, 0 }  ][line width=0.75]      (0, 0) circle [x radius= 3.35, y radius= 3.35]   ;
\end{tikzpicture}

%% file: Gamma_labeled.tex
\tikzset{every picture/.style={line width=0.75pt}} 

\begin{tikzpicture}[x=0.75pt,y=0.75pt,yscale=-1,xscale=1]
	
	\draw [color={rgb, 255:red, 60; green, 211; blue, 33 }  ,draw opacity=1 ]   (155.56,153.68) -- (265.93,104.32) ;
	\draw [shift={(265.93,104.32)}, rotate = 335.91] [color={rgb, 255:red, 60; green, 211; blue, 33 }  ,draw opacity=1 ][fill={rgb, 255:red, 60; green, 211; blue, 33 }  ,fill opacity=1 ][line width=0.75]      (0, 0) circle [x radius= 3.35, y radius= 3.35]   ;
	\draw [shift={(210.75,129)}, rotate = 335.91] [color={rgb, 255:red, 60; green, 211; blue, 33 }  ,draw opacity=1 ][fill={rgb, 255:red, 60; green, 211; blue, 33 }  ,fill opacity=1 ][line width=0.75]      (0, 0) circle [x radius= 3.35, y radius= 3.35]   ;
	\draw [color={rgb, 255:red, 60; green, 211; blue, 33 }  ,draw opacity=1 ]   (155.23,33.48) .. controls (286.84,-30.9) and (315.45,190.1) .. (155.89,273.87) ;
	\draw [shift={(155.89,273.87)}, rotate = 152.3] [color={rgb, 255:red, 60; green, 211; blue, 33 }  ,draw opacity=1 ][fill={rgb, 255:red, 60; green, 211; blue, 33 }  ,fill opacity=1 ][line width=0.75]      (0, 0) circle [x radius= 3.35, y radius= 3.35]   ;
	\draw [shift={(155.23,33.48)}, rotate = 333.93] [color={rgb, 255:red, 60; green, 211; blue, 33 }  ,draw opacity=1 ][fill={rgb, 255:red, 60; green, 211; blue, 33 }  ,fill opacity=1 ][line width=0.75]      (0, 0) circle [x radius= 3.35, y radius= 3.35]   ;
	\draw [color={rgb, 255:red, 0; green, 0; blue, 0 }  ,draw opacity=1 ]   (155.89,273.87) .. controls (-6.23,236.99) and (-3.35,74.04) .. (90.57,38.14) .. controls (108.85,31.15) and (130.59,28.98) .. (155.23,33.48) ;
	\draw [shift={(155.23,33.48)}, rotate = 10.37] [color={rgb, 255:red, 0; green, 0; blue, 0 }  ,draw opacity=1 ][fill={rgb, 255:red, 0; green, 0; blue, 0 }  ,fill opacity=1 ][line width=0.75]      (0, 0) circle [x radius= 3.35, y radius= 3.35]   ;
	\draw [shift={(155.89,273.87)}, rotate = 192.81] [color={rgb, 255:red, 0; green, 0; blue, 0 }  ,draw opacity=1 ][fill={rgb, 255:red, 0; green, 0; blue, 0 }  ,fill opacity=1 ][line width=0.75]      (0, 0) circle [x radius= 3.35, y radius= 3.35]   ;
	\draw [color={rgb, 255:red, 0; green, 0; blue, 0 }  ,draw opacity=1 ]   (64.55,53.72) ;
	\draw [shift={(64.55,53.72)}, rotate = 0] [color={rgb, 255:red, 0; green, 0; blue, 0 }  ,draw opacity=1 ][fill={rgb, 255:red, 0; green, 0; blue, 0 }  ,fill opacity=1 ][line width=0.75]      (0, 0) circle [x radius= 3.35, y radius= 3.35]   ;
	\draw [color={rgb, 255:red, 60; green, 211; blue, 33 }  ,draw opacity=1 ]   (226.76,28.67) ;
	\draw [shift={(226.76,28.67)}, rotate = 0] [color={rgb, 255:red, 60; green, 211; blue, 33 }  ,draw opacity=1 ][fill={rgb, 255:red, 60; green, 211; blue, 33 }  ,fill opacity=1 ][line width=0.75]      (0, 0) circle [x radius= 3.35, y radius= 3.35]   ;
	\draw [color={rgb, 255:red, 208; green, 2; blue, 27 }  ,draw opacity=1 ]   (187.8,102.32) ;
	\draw [shift={(187.8,102.32)}, rotate = 0] [color={rgb, 255:red, 208; green, 2; blue, 27 }  ,draw opacity=1 ][fill={rgb, 255:red, 208; green, 2; blue, 27 }  ,fill opacity=1 ][line width=0.75]      (0, 0) circle [x radius= 3.35, y radius= 3.35]   ;
	\draw [color={rgb, 255:red, 208; green, 2; blue, 27 }  ,draw opacity=1 ]   (198.15,276.67) ;
	\draw [shift={(198.15,276.67)}, rotate = 0] [color={rgb, 255:red, 208; green, 2; blue, 27 }  ,draw opacity=1 ][fill={rgb, 255:red, 208; green, 2; blue, 27 }  ,fill opacity=1 ][line width=0.75]      (0, 0) circle [x radius= 3.35, y radius= 3.35]   ;
	\draw [color={rgb, 255:red, 60; green, 211; blue, 33 }  ,draw opacity=1 ]   (247.45,180.67) ;
	\draw [shift={(247.45,180.67)}, rotate = 0] [color={rgb, 255:red, 60; green, 211; blue, 33 }  ,draw opacity=1 ][fill={rgb, 255:red, 60; green, 211; blue, 33 }  ,fill opacity=1 ][line width=0.75]      (0, 0) circle [x radius= 3.35, y radius= 3.35]   ;
	\draw [color={rgb, 255:red, 0; green, 0; blue, 0 }  ,draw opacity=1 ]   (63.01,225.67) ;
	\draw [shift={(63.01,225.67)}, rotate = 0] [color={rgb, 255:red, 0; green, 0; blue, 0 }  ,draw opacity=1 ][fill={rgb, 255:red, 0; green, 0; blue, 0 }  ,fill opacity=1 ][line width=0.75]      (0, 0) circle [x radius= 3.35, y radius= 3.35]   ;
	\draw [color={rgb, 255:red, 0; green, 0; blue, 0 }  ,draw opacity=1 ]   (155.89,273.87) ;
	\draw [shift={(155.89,273.87)}, rotate = 0] [color={rgb, 255:red, 0; green, 0; blue, 0 }  ,draw opacity=1 ][fill={rgb, 255:red, 0; green, 0; blue, 0 }  ,fill opacity=1 ][line width=0.75]      (0, 0) circle [x radius= 3.35, y radius= 3.35]   ;
	\draw [color={rgb, 255:red, 0; green, 0; blue, 0 }  ,draw opacity=1 ]   (24.93,141.72) ;
	\draw [shift={(24.93,141.72)}, rotate = 0] [color={rgb, 255:red, 0; green, 0; blue, 0 }  ,draw opacity=1 ][fill={rgb, 255:red, 0; green, 0; blue, 0 }  ,fill opacity=1 ][line width=0.75]      (0, 0) circle [x radius= 3.35, y radius= 3.35]   ;
	\draw [color={rgb, 255:red, 0; green, 0; blue, 0 }  ,draw opacity=1 ]   (155.23,33.48) ;
	\draw [shift={(155.23,33.48)}, rotate = 0] [color={rgb, 255:red, 0; green, 0; blue, 0 }  ,draw opacity=1 ][fill={rgb, 255:red, 0; green, 0; blue, 0 }  ,fill opacity=1 ][line width=0.75]      (0, 0) circle [x radius= 3.35, y radius= 3.35]   ;
	\draw [color={rgb, 255:red, 60; green, 211; blue, 33 }  ,draw opacity=1 ]   (265.93,104.32) ;
	\draw [shift={(265.93,104.32)}, rotate = 0] [color={rgb, 255:red, 60; green, 211; blue, 33 }  ,draw opacity=1 ][fill={rgb, 255:red, 60; green, 211; blue, 33 }  ,fill opacity=1 ][line width=0.75]      (0, 0) circle [x radius= 3.35, y radius= 3.35]   ;
	\draw [color={rgb, 255:red, 60; green, 211; blue, 33 }  ,draw opacity=1 ]   (210.75,129) ;
	\draw [shift={(210.75,129)}, rotate = 0] [color={rgb, 255:red, 60; green, 211; blue, 33 }  ,draw opacity=1 ][fill={rgb, 255:red, 60; green, 211; blue, 33 }  ,fill opacity=1 ][line width=0.75]      (0, 0) circle [x radius= 3.35, y radius= 3.35]   ;
	\draw [color={rgb, 255:red, 208; green, 2; blue, 27 }  ,draw opacity=1 ]   (209.81,214.32) ;
	\draw [shift={(209.81,214.32)}, rotate = 0] [color={rgb, 255:red, 208; green, 2; blue, 27 }  ,draw opacity=1 ][fill={rgb, 255:red, 208; green, 2; blue, 27 }  ,fill opacity=1 ][line width=0.75]      (0, 0) circle [x radius= 3.35, y radius= 3.35]   ;
	\draw [color={rgb, 255:red, 255; green, 255; blue, 255 }  ,draw opacity=1 ]   (197,135) ;
	\draw [shift={(197,135)}, rotate = 0] [color={rgb, 255:red, 255; green, 255; blue, 255 }  ,draw opacity=1 ][fill={rgb, 255:red, 255; green, 255; blue, 255 }  ,fill opacity=1 ][line width=0.75]      (0, 0) circle [x radius= 3.35, y radius= 3.35]   ;
	\draw [color={rgb, 255:red, 255; green, 255; blue, 255 }  ,draw opacity=1 ]   (210,235) ;
	\draw [shift={(210,235)}, rotate = 0] [color={rgb, 255:red, 255; green, 255; blue, 255 }  ,draw opacity=1 ][fill={rgb, 255:red, 255; green, 255; blue, 255 }  ,fill opacity=1 ][line width=0.75]      (0, 0) circle [x radius= 3.35, y radius= 3.35]   ;
	\draw [color={rgb, 255:red, 208; green, 2; blue, 27 }  ,draw opacity=1 ]   (155.23,33.48) .. controls (168.77,43.5) and (185.92,87.24) .. (197.29,135.44) .. controls (216.96,218.79) and (219.34,315.44) .. (155.89,273.87) ;
	\draw [color={rgb, 255:red, 0; green, 0; blue, 0 }  ,draw opacity=1 ]   (24.93,141.72) -- (155.56,153.68) ;
	\draw [shift={(155.56,153.68)}, rotate = 5.23] [color={rgb, 255:red, 0; green, 0; blue, 0 }  ,draw opacity=1 ][fill={rgb, 255:red, 0; green, 0; blue, 0 }  ,fill opacity=1 ][line width=0.75]      (0, 0) circle [x radius= 3.35, y radius= 3.35]   ;
	\draw [shift={(90.25,147.7)}, rotate = 5.23] [color={rgb, 255:red, 0; green, 0; blue, 0 }  ,draw opacity=1 ][fill={rgb, 255:red, 0; green, 0; blue, 0 }  ,fill opacity=1 ][line width=0.75]      (0, 0) circle [x radius= 3.35, y radius= 3.35]   ;
	\draw [shift={(24.93,141.72)}, rotate = 5.23] [color={rgb, 255:red, 0; green, 0; blue, 0 }  ,draw opacity=1 ][fill={rgb, 255:red, 0; green, 0; blue, 0 }  ,fill opacity=1 ][line width=0.75]      (0, 0) circle [x radius= 3.35, y radius= 3.35]   ;
	\draw [color={rgb, 255:red, 208; green, 2; blue, 27 }  ,draw opacity=1 ]   (155.56,153.68) -- (209.81,214.32) ;
	\draw [shift={(209.81,214.32)}, rotate = 48.18] [color={rgb, 255:red, 208; green, 2; blue, 27 }  ,draw opacity=1 ][fill={rgb, 255:red, 208; green, 2; blue, 27 }  ,fill opacity=1 ][line width=0.75]      (0, 0) circle [x radius= 3.35, y radius= 3.35]   ;
	\draw [shift={(182.68,184)}, rotate = 48.18] [color={rgb, 255:red, 208; green, 2; blue, 27 }  ,draw opacity=1 ][fill={rgb, 255:red, 208; green, 2; blue, 27 }  ,fill opacity=1 ][line width=0.75]      (0, 0) circle [x radius= 3.35, y radius= 3.35]   ;
	
	\draw (151,12) node [anchor=north west][inner sep=0.75pt]  [font=\footnotesize] [align=left] {1};
	\draw (7,134) node [anchor=north west][inner sep=0.75pt]  [font=\footnotesize] [align=left] {2};
	\draw (152,132) node [anchor=north west][inner sep=0.75pt]  [font=\footnotesize] [align=left] {3};
	\draw (274,92) node [anchor=north west][inner sep=0.75pt]  [font=\footnotesize] [align=left] {4};
	\draw (152,252) node [anchor=north west][inner sep=0.75pt]  [font=\footnotesize] [align=left] {5};
	\draw (216,200) node [anchor=north west][inner sep=0.75pt]  [font=\footnotesize] [align=left] {6};
	\draw (72,58) node [anchor=north west][inner sep=0.75pt]  [font=\footnotesize] [align=left] {7};
	\draw (214,35) node [anchor=north west][inner sep=0.75pt]  [font=\footnotesize] [align=left] {8};
	\draw (171,96) node [anchor=north west][inner sep=0.75pt]  [font=\footnotesize] [align=left] {9};
	\draw (84,128) node [anchor=north west][inner sep=0.75pt]  [font=\footnotesize] [align=left] {10};
	\draw (210,107) node [anchor=north west][inner sep=0.75pt]  [font=\footnotesize] [align=left] {11};
	\draw (165,190) node [anchor=north west][inner sep=0.75pt]  [font=\footnotesize] [align=left] {12};
	\draw (66,211) node [anchor=north west][inner sep=0.75pt]  [font=\footnotesize] [align=left] {13};
	\draw (254,183) node [anchor=north west][inner sep=0.75pt]  [font=\footnotesize] [align=left] {14};
	\draw (202,279) node [anchor=north west][inner sep=0.75pt]  [font=\footnotesize] [align=left] {15};
\end{tikzpicture}

%% file: 3D4proof.tex
\tikzset{every picture/.style={line width=0.75pt}} 

\begin{tikzpicture}[x=0.75pt,y=0.75pt,yscale=-1,xscale=1]
	
	\draw  [color={rgb, 255:red, 248; green, 231; blue, 28 }  ,draw opacity=1 ][fill={rgb, 255:red, 248; green, 231; blue, 28 }  ,fill opacity=1 ] (140.47,434.86) .. controls (140.47,431.48) and (143.42,428.74) .. (147.06,428.74) .. controls (150.7,428.74) and (153.65,431.48) .. (153.65,434.86) .. controls (153.65,438.24) and (150.7,440.98) .. (147.06,440.98) .. controls (143.42,440.98) and (140.47,438.24) .. (140.47,434.86) -- cycle ;
	\draw  [color={rgb, 255:red, 208; green, 2; blue, 27 }  ,draw opacity=1 ][fill={rgb, 255:red, 208; green, 2; blue, 27 }  ,fill opacity=1 ] (274.75,231.53) .. controls (274.75,228.16) and (277.7,225.42) .. (281.34,225.42) .. controls (284.98,225.42) and (287.92,228.16) .. (287.92,231.53) .. controls (287.92,234.91) and (284.98,237.65) .. (281.34,237.65) .. controls (277.7,237.65) and (274.75,234.91) .. (274.75,231.53) -- cycle ;
	\draw  [color={rgb, 255:red, 248; green, 231; blue, 28 }  ,draw opacity=1 ][fill={rgb, 255:red, 248; green, 231; blue, 28 }  ,fill opacity=1 ] (243.37,231.39) .. controls (243.37,228.01) and (246.32,225.27) .. (249.95,225.27) .. controls (253.59,225.27) and (256.54,228.01) .. (256.54,231.39) .. controls (256.54,234.77) and (253.59,237.51) .. (249.95,237.51) .. controls (246.32,237.51) and (243.37,234.77) .. (243.37,231.39) -- cycle ;
	\draw  [color={rgb, 255:red, 248; green, 231; blue, 28 }  ,draw opacity=1 ][fill={rgb, 255:red, 248; green, 231; blue, 28 }  ,fill opacity=1 ] (306.35,231.39) .. controls (306.35,228.01) and (309.3,225.27) .. (312.93,225.27) .. controls (316.57,225.27) and (319.52,228.01) .. (319.52,231.39) .. controls (319.52,234.77) and (316.57,237.51) .. (312.93,237.51) .. controls (309.3,237.51) and (306.35,234.77) .. (306.35,231.39) -- cycle ;
	\draw  [color={rgb, 255:red, 184; green, 233; blue, 134 }  ,draw opacity=1 ][fill={rgb, 255:red, 184; green, 233; blue, 134 }  ,fill opacity=1 ] (211.88,231.39) .. controls (211.88,228.01) and (214.83,225.27) .. (218.46,225.27) .. controls (222.1,225.27) and (225.05,228.01) .. (225.05,231.39) .. controls (225.05,234.77) and (222.1,237.51) .. (218.46,237.51) .. controls (214.83,237.51) and (211.88,234.77) .. (211.88,231.39) -- cycle ;
	\draw  [color={rgb, 255:red, 184; green, 233; blue, 134 }  ,draw opacity=1 ][fill={rgb, 255:red, 184; green, 233; blue, 134 }  ,fill opacity=1 ] (337.63,231.68) .. controls (337.63,228.3) and (340.58,225.56) .. (344.21,225.56) .. controls (347.85,225.56) and (350.8,228.3) .. (350.8,231.68) .. controls (350.8,235.06) and (347.85,237.8) .. (344.21,237.8) .. controls (340.58,237.8) and (337.63,235.06) .. (337.63,231.68) -- cycle ;
	\draw  [color={rgb, 255:red, 184; green, 233; blue, 134 }  ,draw opacity=1 ][fill={rgb, 255:red, 184; green, 233; blue, 134 }  ,fill opacity=1 ] (274.65,289.69) .. controls (274.65,286.31) and (277.6,283.57) .. (281.23,283.57) .. controls (284.87,283.57) and (287.82,286.31) .. (287.82,289.69) .. controls (287.82,293.07) and (284.87,295.81) .. (281.23,295.81) .. controls (277.6,295.81) and (274.65,293.07) .. (274.65,289.69) -- cycle ;
	\draw    (218.46,231.39) -- (344.21,231.68) ;
	\draw [shift={(344.21,231.68)}, rotate = 0.13] [color={rgb, 255:red, 0; green, 0; blue, 0 }  ][fill={rgb, 255:red, 0; green, 0; blue, 0 }  ][line width=0.75]      (0, 0) circle [x radius= 3.35, y radius= 3.35]   ;
	\draw [shift={(281.34,231.53)}, rotate = 0.13] [color={rgb, 255:red, 0; green, 0; blue, 0 }  ][fill={rgb, 255:red, 0; green, 0; blue, 0 }  ][line width=0.75]      (0, 0) circle [x radius= 3.35, y radius= 3.35]   ;
	\draw [shift={(218.46,231.39)}, rotate = 0.13] [color={rgb, 255:red, 0; green, 0; blue, 0 }  ][fill={rgb, 255:red, 0; green, 0; blue, 0 }  ][line width=0.75]      (0, 0) circle [x radius= 3.35, y radius= 3.35]   ;
	\draw    (249.95,231.39) ;
	\draw [shift={(249.95,231.39)}, rotate = 0] [color={rgb, 255:red, 0; green, 0; blue, 0 }  ][fill={rgb, 255:red, 0; green, 0; blue, 0 }  ][line width=0.75]      (0, 0) circle [x radius= 3.35, y radius= 3.35]   ;
	\draw    (312.93,231.39) ;
	\draw [shift={(312.93,231.39)}, rotate = 0] [color={rgb, 255:red, 0; green, 0; blue, 0 }  ][fill={rgb, 255:red, 0; green, 0; blue, 0 }  ][line width=0.75]      (0, 0) circle [x radius= 3.35, y radius= 3.35]   ;
	\draw    (218.46,231.39) -- (195.42,212.18) ;
	\draw [shift={(195.42,212.18)}, rotate = 219.81] [color={rgb, 255:red, 0; green, 0; blue, 0 }  ][fill={rgb, 255:red, 0; green, 0; blue, 0 }  ][line width=0.75]      (0, 0) circle [x radius= 3.35, y radius= 3.35]   ;
	\draw    (218.46,231.39) -- (195.42,250.59) ;
	\draw [shift={(195.42,250.59)}, rotate = 140.19] [color={rgb, 255:red, 0; green, 0; blue, 0 }  ][fill={rgb, 255:red, 0; green, 0; blue, 0 }  ][line width=0.75]      (0, 0) circle [x radius= 3.35, y radius= 3.35]   ;
	\draw    (281.23,289.69) -- (260.55,311.09) ;
	\draw [shift={(260.55,311.09)}, rotate = 134.02] [color={rgb, 255:red, 0; green, 0; blue, 0 }  ][fill={rgb, 255:red, 0; green, 0; blue, 0 }  ][line width=0.75]      (0, 0) circle [x radius= 3.35, y radius= 3.35]   ;
	\draw    (281.23,289.69) -- (301.91,311.09) ;
	\draw [shift={(301.91,311.09)}, rotate = 45.98] [color={rgb, 255:red, 0; green, 0; blue, 0 }  ][fill={rgb, 255:red, 0; green, 0; blue, 0 }  ][line width=0.75]      (0, 0) circle [x radius= 3.35, y radius= 3.35]   ;
	\draw    (344.21,231.68) -- (367.25,250.89) ;
	\draw [shift={(367.25,250.89)}, rotate = 39.81] [color={rgb, 255:red, 0; green, 0; blue, 0 }  ][fill={rgb, 255:red, 0; green, 0; blue, 0 }  ][line width=0.75]      (0, 0) circle [x radius= 3.35, y radius= 3.35]   ;
	\draw    (344.21,231.68) -- (367.25,212.47) ;
	\draw [shift={(367.25,212.47)}, rotate = 320.19] [color={rgb, 255:red, 0; green, 0; blue, 0 }  ][fill={rgb, 255:red, 0; green, 0; blue, 0 }  ][line width=0.75]      (0, 0) circle [x radius= 3.35, y radius= 3.35]   ;
	\draw    (281.23,289.69) ;
	\draw [shift={(281.23,289.69)}, rotate = 0] [color={rgb, 255:red, 0; green, 0; blue, 0 }  ][fill={rgb, 255:red, 0; green, 0; blue, 0 }  ][line width=0.75]      (0, 0) circle [x radius= 3.35, y radius= 3.35]   ;
	
	\draw    (281.47,192.51) -- (281.47,155.49) ;
	\draw [shift={(281.47,153.49)}, rotate = 90] [color={rgb, 255:red, 0; green, 0; blue, 0 }  ][line width=0.75]    (10.93,-3.29) .. controls (6.95,-1.4) and (3.31,-0.3) .. (0,0) .. controls (3.31,0.3) and (6.95,1.4) .. (10.93,3.29)   ;
	\draw  [color={rgb, 255:red, 208; green, 2; blue, 27 }  ,draw opacity=1 ][fill={rgb, 255:red, 208; green, 2; blue, 27 }  ,fill opacity=1 ] (274.75,46.15) .. controls (274.75,42.77) and (277.7,40.03) .. (281.34,40.03) .. controls (284.98,40.03) and (287.92,42.77) .. (287.92,46.15) .. controls (287.92,49.53) and (284.98,52.27) .. (281.34,52.27) .. controls (277.7,52.27) and (274.75,49.53) .. (274.75,46.15) -- cycle ;
	\draw  [color={rgb, 255:red, 248; green, 231; blue, 28 }  ,draw opacity=1 ][fill={rgb, 255:red, 248; green, 231; blue, 28 }  ,fill opacity=1 ] (243.37,46.01) .. controls (243.37,42.63) and (246.32,39.89) .. (249.95,39.89) .. controls (253.59,39.89) and (256.54,42.63) .. (256.54,46.01) .. controls (256.54,49.38) and (253.59,52.12) .. (249.95,52.12) .. controls (246.32,52.12) and (243.37,49.38) .. (243.37,46.01) -- cycle ;
	\draw  [color={rgb, 255:red, 184; green, 233; blue, 134 }  ,draw opacity=1 ][fill={rgb, 255:red, 184; green, 233; blue, 134 }  ,fill opacity=1 ] (211.88,46.01) .. controls (211.88,42.63) and (214.83,39.89) .. (218.46,39.89) .. controls (222.1,39.89) and (225.05,42.63) .. (225.05,46.01) .. controls (225.05,49.38) and (222.1,52.12) .. (218.46,52.12) .. controls (214.83,52.12) and (211.88,49.38) .. (211.88,46.01) -- cycle ;
	\draw  [color={rgb, 255:red, 184; green, 233; blue, 134 }  ,draw opacity=1 ][fill={rgb, 255:red, 184; green, 233; blue, 134 }  ,fill opacity=1 ] (337.63,46.3) .. controls (337.63,42.92) and (340.58,40.18) .. (344.21,40.18) .. controls (347.85,40.18) and (350.8,42.92) .. (350.8,46.3) .. controls (350.8,49.68) and (347.85,52.42) .. (344.21,52.42) .. controls (340.58,52.42) and (337.63,49.68) .. (337.63,46.3) -- cycle ;
	\draw  [color={rgb, 255:red, 184; green, 233; blue, 134 }  ,draw opacity=1 ][fill={rgb, 255:red, 184; green, 233; blue, 134 }  ,fill opacity=1 ] (274.65,104.31) .. controls (274.65,100.93) and (277.6,98.19) .. (281.23,98.19) .. controls (284.87,98.19) and (287.82,100.93) .. (287.82,104.31) .. controls (287.82,107.69) and (284.87,110.42) .. (281.23,110.42) .. controls (277.6,110.42) and (274.65,107.69) .. (274.65,104.31) -- cycle ;
	\draw    (249.95,46.01) ;
	\draw [shift={(249.95,46.01)}, rotate = 0] [color={rgb, 255:red, 0; green, 0; blue, 0 }  ][fill={rgb, 255:red, 0; green, 0; blue, 0 }  ][line width=0.75]      (0, 0) circle [x radius= 3.35, y radius= 3.35]   ;
	\draw    (218.46,46.01) -- (195.42,26.8) ;
	\draw [shift={(195.42,26.8)}, rotate = 219.81] [color={rgb, 255:red, 0; green, 0; blue, 0 }  ][fill={rgb, 255:red, 0; green, 0; blue, 0 }  ][line width=0.75]      (0, 0) circle [x radius= 3.35, y radius= 3.35]   ;
	\draw    (218.46,46.01) -- (195.42,65.21) ;
	\draw [shift={(195.42,65.21)}, rotate = 140.19] [color={rgb, 255:red, 0; green, 0; blue, 0 }  ][fill={rgb, 255:red, 0; green, 0; blue, 0 }  ][line width=0.75]      (0, 0) circle [x radius= 3.35, y radius= 3.35]   ;
	\draw    (281.23,104.31) -- (260.55,125.71) ;
	\draw [shift={(260.55,125.71)}, rotate = 134.02] [color={rgb, 255:red, 0; green, 0; blue, 0 }  ][fill={rgb, 255:red, 0; green, 0; blue, 0 }  ][line width=0.75]      (0, 0) circle [x radius= 3.35, y radius= 3.35]   ;
	\draw    (281.23,104.31) -- (301.91,125.71) ;
	\draw [shift={(301.91,125.71)}, rotate = 45.98] [color={rgb, 255:red, 0; green, 0; blue, 0 }  ][fill={rgb, 255:red, 0; green, 0; blue, 0 }  ][line width=0.75]      (0, 0) circle [x radius= 3.35, y radius= 3.35]   ;
	\draw    (344.21,46.3) -- (367.25,65.5) ;
	\draw [shift={(367.25,65.5)}, rotate = 39.81] [color={rgb, 255:red, 0; green, 0; blue, 0 }  ][fill={rgb, 255:red, 0; green, 0; blue, 0 }  ][line width=0.75]      (0, 0) circle [x radius= 3.35, y radius= 3.35]   ;
	\draw    (344.21,46.3) -- (367.25,27.09) ;
	\draw [shift={(367.25,27.09)}, rotate = 320.19] [color={rgb, 255:red, 0; green, 0; blue, 0 }  ][fill={rgb, 255:red, 0; green, 0; blue, 0 }  ][line width=0.75]      (0, 0) circle [x radius= 3.35, y radius= 3.35]   ;
	\draw    (281.23,104.31) ;
	\draw [shift={(281.23,104.31)}, rotate = 0] [color={rgb, 255:red, 0; green, 0; blue, 0 }  ][fill={rgb, 255:red, 0; green, 0; blue, 0 }  ][line width=0.75]      (0, 0) circle [x radius= 3.35, y radius= 3.35]   ;
	\draw    (281.08,46.1) -- (218.46,46.01) ;
	\draw [shift={(218.46,46.01)}, rotate = 180.09] [color={rgb, 255:red, 0; green, 0; blue, 0 }  ][fill={rgb, 255:red, 0; green, 0; blue, 0 }  ][line width=0.75]      (0, 0) circle [x radius= 3.35, y radius= 3.35]   ;
	\draw [shift={(249.77,46.05)}, rotate = 180.09] [color={rgb, 255:red, 0; green, 0; blue, 0 }  ][fill={rgb, 255:red, 0; green, 0; blue, 0 }  ][line width=0.75]      (0, 0) circle [x radius= 3.35, y radius= 3.35]   ;
	\draw [shift={(281.08,46.1)}, rotate = 180.09] [color={rgb, 255:red, 0; green, 0; blue, 0 }  ][fill={rgb, 255:red, 0; green, 0; blue, 0 }  ][line width=0.75]      (0, 0) circle [x radius= 3.35, y radius= 3.35]   ;
	\draw    (344.21,46.3) ;
	\draw [shift={(344.21,46.3)}, rotate = 0] [color={rgb, 255:red, 0; green, 0; blue, 0 }  ][fill={rgb, 255:red, 0; green, 0; blue, 0 }  ][line width=0.75]      (0, 0) circle [x radius= 3.35, y radius= 3.35]   ;
	\draw    (357.81,354.23) -- (386.78,379.98) ;
	\draw [shift={(388.27,381.31)}, rotate = 221.63] [color={rgb, 255:red, 0; green, 0; blue, 0 }  ][line width=0.75]    (10.93,-3.29) .. controls (6.95,-1.4) and (3.31,-0.3) .. (0,0) .. controls (3.31,0.3) and (6.95,1.4) .. (10.93,3.29)   ;
	\draw    (209.02,354.23) -- (180.04,379.98) ;
	\draw [shift={(178.55,381.31)}, rotate = 318.37] [color={rgb, 255:red, 0; green, 0; blue, 0 }  ][line width=0.75]    (10.93,-3.29) .. controls (6.95,-1.4) and (3.31,-0.3) .. (0,0) .. controls (3.31,0.3) and (6.95,1.4) .. (10.93,3.29)   ;
	\draw  [color={rgb, 255:red, 208; green, 2; blue, 27 }  ,draw opacity=1 ][fill={rgb, 255:red, 208; green, 2; blue, 27 }  ,fill opacity=1 ] (109.17,434.86) .. controls (109.17,431.48) and (112.12,428.74) .. (115.75,428.74) .. controls (119.39,428.74) and (122.34,431.48) .. (122.34,434.86) .. controls (122.34,438.23) and (119.39,440.97) .. (115.75,440.97) .. controls (112.12,440.97) and (109.17,438.23) .. (109.17,434.86) -- cycle ;
	\draw  [color={rgb, 255:red, 184; green, 233; blue, 134 }  ,draw opacity=1 ][fill={rgb, 255:red, 184; green, 233; blue, 134 }  ,fill opacity=1 ] (46.55,434.76) .. controls (46.55,431.38) and (49.5,428.64) .. (53.14,428.64) .. controls (56.78,428.64) and (59.73,431.38) .. (59.73,434.76) .. controls (59.73,438.14) and (56.78,440.88) .. (53.14,440.88) .. controls (49.5,440.88) and (46.55,438.14) .. (46.55,434.76) -- cycle ;
	\draw  [color={rgb, 255:red, 184; green, 233; blue, 134 }  ,draw opacity=1 ][fill={rgb, 255:red, 184; green, 233; blue, 134 }  ,fill opacity=1 ] (109.32,493.06) .. controls (109.32,489.68) and (112.27,486.94) .. (115.91,486.94) .. controls (119.55,486.94) and (122.5,489.68) .. (122.5,493.06) .. controls (122.5,496.44) and (119.55,499.18) .. (115.91,499.18) .. controls (112.27,499.18) and (109.32,496.44) .. (109.32,493.06) -- cycle ;
	\draw    (53.14,434.76) -- (30.1,415.55) ;
	\draw [shift={(30.1,415.55)}, rotate = 219.81] [color={rgb, 255:red, 0; green, 0; blue, 0 }  ][fill={rgb, 255:red, 0; green, 0; blue, 0 }  ][line width=0.75]      (0, 0) circle [x radius= 3.35, y radius= 3.35]   ;
	\draw    (53.14,434.76) -- (30.1,453.96) ;
	\draw [shift={(30.1,453.96)}, rotate = 140.19] [color={rgb, 255:red, 0; green, 0; blue, 0 }  ][fill={rgb, 255:red, 0; green, 0; blue, 0 }  ][line width=0.75]      (0, 0) circle [x radius= 3.35, y radius= 3.35]   ;
	\draw    (115.91,493.06) -- (95.23,514.46) ;
	\draw [shift={(95.23,514.46)}, rotate = 134.02] [color={rgb, 255:red, 0; green, 0; blue, 0 }  ][fill={rgb, 255:red, 0; green, 0; blue, 0 }  ][line width=0.75]      (0, 0) circle [x radius= 3.35, y radius= 3.35]   ;
	\draw    (115.91,493.06) -- (136.59,514.46) ;
	\draw [shift={(136.59,514.46)}, rotate = 45.98] [color={rgb, 255:red, 0; green, 0; blue, 0 }  ][fill={rgb, 255:red, 0; green, 0; blue, 0 }  ][line width=0.75]      (0, 0) circle [x radius= 3.35, y radius= 3.35]   ;
	\draw    (115.91,493.06) ;
	\draw [shift={(115.91,493.06)}, rotate = 0] [color={rgb, 255:red, 0; green, 0; blue, 0 }  ][fill={rgb, 255:red, 0; green, 0; blue, 0 }  ][line width=0.75]      (0, 0) circle [x radius= 3.35, y radius= 3.35]   ;
	\draw    (201.93,415.84) ;
	\draw [shift={(201.93,415.84)}, rotate = 0] [color={rgb, 255:red, 0; green, 0; blue, 0 }  ][fill={rgb, 255:red, 0; green, 0; blue, 0 }  ][line width=0.75]      (0, 0) circle [x radius= 3.35, y radius= 3.35]   ;
	\draw    (201.93,454.26) ;
	\draw [shift={(201.93,454.26)}, rotate = 0] [color={rgb, 255:red, 0; green, 0; blue, 0 }  ][fill={rgb, 255:red, 0; green, 0; blue, 0 }  ][line width=0.75]      (0, 0) circle [x radius= 3.35, y radius= 3.35]   ;
	\draw  [color={rgb, 255:red, 248; green, 231; blue, 28 }  ,draw opacity=1 ][fill={rgb, 255:red, 248; green, 231; blue, 28 }  ,fill opacity=1 ] (77.86,434.81) .. controls (77.86,431.43) and (80.81,428.69) .. (84.45,428.69) .. controls (88.08,428.69) and (91.03,431.43) .. (91.03,434.81) .. controls (91.03,438.19) and (88.08,440.92) .. (84.45,440.92) .. controls (80.81,440.92) and (77.86,438.19) .. (77.86,434.81) -- cycle ;
	\draw    (115.75,434.86) -- (53.14,434.76) ;
	\draw [shift={(53.14,434.76)}, rotate = 180.09] [color={rgb, 255:red, 0; green, 0; blue, 0 }  ][fill={rgb, 255:red, 0; green, 0; blue, 0 }  ][line width=0.75]      (0, 0) circle [x radius= 3.35, y radius= 3.35]   ;
	\draw [shift={(84.45,434.81)}, rotate = 180.09] [color={rgb, 255:red, 0; green, 0; blue, 0 }  ][fill={rgb, 255:red, 0; green, 0; blue, 0 }  ][line width=0.75]      (0, 0) circle [x radius= 3.35, y radius= 3.35]   ;
	\draw [shift={(115.75,434.86)}, rotate = 180.09] [color={rgb, 255:red, 0; green, 0; blue, 0 }  ][fill={rgb, 255:red, 0; green, 0; blue, 0 }  ][line width=0.75]      (0, 0) circle [x radius= 3.35, y radius= 3.35]   ;
	\draw  [color={rgb, 255:red, 208; green, 2; blue, 27 }  ,draw opacity=1 ][fill={rgb, 255:red, 208; green, 2; blue, 27 }  ,fill opacity=1 ] (442.7,435.25) .. controls (442.7,431.87) and (445.65,429.13) .. (449.28,429.13) .. controls (452.92,429.13) and (455.87,431.87) .. (455.87,435.25) .. controls (455.87,438.62) and (452.92,441.36) .. (449.28,441.36) .. controls (445.65,441.36) and (442.7,438.62) .. (442.7,435.25) -- cycle ;
	\draw  [color={rgb, 255:red, 248; green, 231; blue, 28 }  ,draw opacity=1 ][fill={rgb, 255:red, 248; green, 231; blue, 28 }  ,fill opacity=1 ] (411.31,435.1) .. controls (411.31,431.72) and (414.26,428.98) .. (417.9,428.98) .. controls (421.54,428.98) and (424.49,431.72) .. (424.49,435.1) .. controls (424.49,438.48) and (421.54,441.22) .. (417.9,441.22) .. controls (414.26,441.22) and (411.31,438.48) .. (411.31,435.1) -- cycle ;
	\draw  [color={rgb, 255:red, 248; green, 231; blue, 28 }  ,draw opacity=1 ][fill={rgb, 255:red, 248; green, 231; blue, 28 }  ,fill opacity=1 ] (474.29,435.1) .. controls (474.29,431.72) and (477.24,428.98) .. (480.88,428.98) .. controls (484.52,428.98) and (487.47,431.72) .. (487.47,435.1) .. controls (487.47,438.48) and (484.52,441.22) .. (480.88,441.22) .. controls (477.24,441.22) and (474.29,438.48) .. (474.29,435.1) -- cycle ;
	\draw  [color={rgb, 255:red, 184; green, 233; blue, 134 }  ,draw opacity=1 ][fill={rgb, 255:red, 184; green, 233; blue, 134 }  ,fill opacity=1 ] (379.82,435.1) .. controls (379.82,431.72) and (382.77,428.98) .. (386.41,428.98) .. controls (390.05,428.98) and (393,431.72) .. (393,435.1) .. controls (393,438.48) and (390.05,441.22) .. (386.41,441.22) .. controls (382.77,441.22) and (379.82,438.48) .. (379.82,435.1) -- cycle ;
	\draw  [color={rgb, 255:red, 184; green, 233; blue, 134 }  ,draw opacity=1 ][fill={rgb, 255:red, 184; green, 233; blue, 134 }  ,fill opacity=1 ] (505.57,435.39) .. controls (505.57,432.01) and (508.52,429.27) .. (512.16,429.27) .. controls (515.8,429.27) and (518.75,432.01) .. (518.75,435.39) .. controls (518.75,438.77) and (515.8,441.51) .. (512.16,441.51) .. controls (508.52,441.51) and (505.57,438.77) .. (505.57,435.39) -- cycle ;
	\draw  [color={rgb, 255:red, 184; green, 233; blue, 134 }  ,draw opacity=1 ][fill={rgb, 255:red, 184; green, 233; blue, 134 }  ,fill opacity=1 ] (442.59,493.4) .. controls (442.59,490.02) and (445.54,487.28) .. (449.18,487.28) .. controls (452.82,487.28) and (455.77,490.02) .. (455.77,493.4) .. controls (455.77,496.78) and (452.82,499.52) .. (449.18,499.52) .. controls (445.54,499.52) and (442.59,496.78) .. (442.59,493.4) -- cycle ;
	\draw    (386.41,435.1) -- (512.16,435.39) ;
	\draw [shift={(512.16,435.39)}, rotate = 0.13] [color={rgb, 255:red, 0; green, 0; blue, 0 }  ][fill={rgb, 255:red, 0; green, 0; blue, 0 }  ][line width=0.75]      (0, 0) circle [x radius= 3.35, y radius= 3.35]   ;
	\draw [shift={(449.28,435.25)}, rotate = 0.13] [color={rgb, 255:red, 0; green, 0; blue, 0 }  ][fill={rgb, 255:red, 0; green, 0; blue, 0 }  ][line width=0.75]      (0, 0) circle [x radius= 3.35, y radius= 3.35]   ;
	\draw [shift={(386.41,435.1)}, rotate = 0.13] [color={rgb, 255:red, 0; green, 0; blue, 0 }  ][fill={rgb, 255:red, 0; green, 0; blue, 0 }  ][line width=0.75]      (0, 0) circle [x radius= 3.35, y radius= 3.35]   ;
	\draw    (417.9,435.1) ;
	\draw [shift={(417.9,435.1)}, rotate = 0] [color={rgb, 255:red, 0; green, 0; blue, 0 }  ][fill={rgb, 255:red, 0; green, 0; blue, 0 }  ][line width=0.75]      (0, 0) circle [x radius= 3.35, y radius= 3.35]   ;
	\draw    (480.88,435.1) ;
	\draw [shift={(480.88,435.1)}, rotate = 0] [color={rgb, 255:red, 0; green, 0; blue, 0 }  ][fill={rgb, 255:red, 0; green, 0; blue, 0 }  ][line width=0.75]      (0, 0) circle [x radius= 3.35, y radius= 3.35]   ;
	\draw    (386.41,435.1) -- (363.37,415.89) ;
	\draw [shift={(363.37,415.89)}, rotate = 219.81] [color={rgb, 255:red, 0; green, 0; blue, 0 }  ][fill={rgb, 255:red, 0; green, 0; blue, 0 }  ][line width=0.75]      (0, 0) circle [x radius= 3.35, y radius= 3.35]   ;
	\draw    (449.18,493.4) -- (428.5,514.8) ;
	\draw [shift={(428.5,514.8)}, rotate = 134.02] [color={rgb, 255:red, 0; green, 0; blue, 0 }  ][fill={rgb, 255:red, 0; green, 0; blue, 0 }  ][line width=0.75]      (0, 0) circle [x radius= 3.35, y radius= 3.35]   ;
	\draw    (449.18,493.4) -- (469.86,514.8) ;
	\draw [shift={(469.86,514.8)}, rotate = 45.98] [color={rgb, 255:red, 0; green, 0; blue, 0 }  ][fill={rgb, 255:red, 0; green, 0; blue, 0 }  ][line width=0.75]      (0, 0) circle [x radius= 3.35, y radius= 3.35]   ;
	\draw    (512.16,435.39) -- (535.2,454.6) ;
	\draw [shift={(535.2,454.6)}, rotate = 39.81] [color={rgb, 255:red, 0; green, 0; blue, 0 }  ][fill={rgb, 255:red, 0; green, 0; blue, 0 }  ][line width=0.75]      (0, 0) circle [x radius= 3.35, y radius= 3.35]   ;
	\draw    (512.16,435.39) -- (535.2,416.19) ;
	\draw [shift={(535.2,416.19)}, rotate = 320.19] [color={rgb, 255:red, 0; green, 0; blue, 0 }  ][fill={rgb, 255:red, 0; green, 0; blue, 0 }  ][line width=0.75]      (0, 0) circle [x radius= 3.35, y radius= 3.35]   ;
	\draw    (449.18,493.4) ;
	\draw [shift={(449.18,493.4)}, rotate = 0] [color={rgb, 255:red, 0; green, 0; blue, 0 }  ][fill={rgb, 255:red, 0; green, 0; blue, 0 }  ][line width=0.75]      (0, 0) circle [x radius= 3.35, y radius= 3.35]   ;
	
	\draw    (147.06,434.86) -- (115.75,434.86) ;
	\draw [shift={(115.75,434.86)}, rotate = 180.01] [color={rgb, 255:red, 0; green, 0; blue, 0 }  ][fill={rgb, 255:red, 0; green, 0; blue, 0 }  ][line width=0.75]      (0, 0) circle [x radius= 3.35, y radius= 3.35]   ;
	\draw [shift={(147.06,434.86)}, rotate = 180.01] [color={rgb, 255:red, 0; green, 0; blue, 0 }  ][fill={rgb, 255:red, 0; green, 0; blue, 0 }  ][line width=0.75]      (0, 0) circle [x radius= 3.35, y radius= 3.35]   ;

\end{tikzpicture}

%% file: Gamma_proof_1.tex
\tikzset{every picture/.style={line width=0.75pt}} 

\begin{tikzpicture}[x=0.75pt,y=0.75pt,yscale=-1,xscale=1]
	
	\draw  [color={rgb, 255:red, 184; green, 233; blue, 134 }  ,draw opacity=1 ][fill={rgb, 255:red, 184; green, 233; blue, 134 }  ,fill opacity=1 ] (419.75,133.45) .. controls (419.75,126.52) and (425.37,120.9) .. (432.3,120.9) .. controls (439.23,120.9) and (444.85,126.52) .. (444.85,133.45) .. controls (444.85,140.38) and (439.23,146) .. (432.3,146) .. controls (425.37,146) and (419.75,140.38) .. (419.75,133.45) -- cycle ;
	\draw  [color={rgb, 255:red, 184; green, 233; blue, 134 }  ,draw opacity=1 ][fill={rgb, 255:red, 184; green, 233; blue, 134 }  ,fill opacity=1 ] (299.95,133.15) .. controls (299.95,126.22) and (305.57,120.6) .. (312.5,120.6) .. controls (319.43,120.6) and (325.05,126.22) .. (325.05,133.15) .. controls (325.05,140.08) and (319.43,145.7) .. (312.5,145.7) .. controls (305.57,145.7) and (299.95,140.08) .. (299.95,133.15) -- cycle ;
	\draw  [color={rgb, 255:red, 184; green, 233; blue, 134 }  ,draw opacity=1 ][fill={rgb, 255:red, 184; green, 233; blue, 134 }  ,fill opacity=1 ] (180.15,132.85) .. controls (180.15,125.92) and (185.77,120.3) .. (192.7,120.3) .. controls (199.63,120.3) and (205.25,125.92) .. (205.25,132.85) .. controls (205.25,139.78) and (199.63,145.4) .. (192.7,145.4) .. controls (185.77,145.4) and (180.15,139.78) .. (180.15,132.85) -- cycle ;
	\draw  [color={rgb, 255:red, 184; green, 233; blue, 134 }  ,draw opacity=1 ][fill={rgb, 255:red, 184; green, 233; blue, 134 }  ,fill opacity=1 ] (60.35,132.55) .. controls (60.35,125.62) and (65.97,120) .. (72.9,120) .. controls (79.83,120) and (85.45,125.62) .. (85.45,132.55) .. controls (85.45,139.48) and (79.83,145.1) .. (72.9,145.1) .. controls (65.97,145.1) and (60.35,139.48) .. (60.35,132.55) -- cycle ;
	\draw    (13,132.4) -- (132.8,132.7) ;
	\draw [shift={(132.8,132.7)}, rotate = 0.14] [color={rgb, 255:red, 0; green, 0; blue, 0 }  ][fill={rgb, 255:red, 0; green, 0; blue, 0 }  ][line width=0.75]      (0, 0) circle [x radius= 3.35, y radius= 3.35]   ;
	\draw [shift={(72.9,132.55)}, rotate = 0.14] [color={rgb, 255:red, 0; green, 0; blue, 0 }  ][fill={rgb, 255:red, 0; green, 0; blue, 0 }  ][line width=0.75]      (0, 0) circle [x radius= 3.35, y radius= 3.35]   ;
	\draw [shift={(13,132.4)}, rotate = 0.14] [color={rgb, 255:red, 0; green, 0; blue, 0 }  ][fill={rgb, 255:red, 0; green, 0; blue, 0 }  ][line width=0.75]      (0, 0) circle [x radius= 3.35, y radius= 3.35]   ;
	\draw    (132.8,132.7) -- (252.6,133) ;
	\draw [shift={(252.6,133)}, rotate = 0.14] [color={rgb, 255:red, 0; green, 0; blue, 0 }  ][fill={rgb, 255:red, 0; green, 0; blue, 0 }  ][line width=0.75]      (0, 0) circle [x radius= 3.35, y radius= 3.35]   ;
	\draw [shift={(192.7,132.85)}, rotate = 0.14] [color={rgb, 255:red, 0; green, 0; blue, 0 }  ][fill={rgb, 255:red, 0; green, 0; blue, 0 }  ][line width=0.75]      (0, 0) circle [x radius= 3.35, y radius= 3.35]   ;
	\draw [shift={(132.8,132.7)}, rotate = 0.14] [color={rgb, 255:red, 0; green, 0; blue, 0 }  ][fill={rgb, 255:red, 0; green, 0; blue, 0 }  ][line width=0.75]      (0, 0) circle [x radius= 3.35, y radius= 3.35]   ;
	\draw    (252.6,133) -- (372.4,133.3) ;
	\draw [shift={(372.4,133.3)}, rotate = 0.14] [color={rgb, 255:red, 0; green, 0; blue, 0 }  ][fill={rgb, 255:red, 0; green, 0; blue, 0 }  ][line width=0.75]      (0, 0) circle [x radius= 3.35, y radius= 3.35]   ;
	\draw [shift={(312.5,133.15)}, rotate = 0.14] [color={rgb, 255:red, 0; green, 0; blue, 0 }  ][fill={rgb, 255:red, 0; green, 0; blue, 0 }  ][line width=0.75]      (0, 0) circle [x radius= 3.35, y radius= 3.35]   ;
	\draw [shift={(252.6,133)}, rotate = 0.14] [color={rgb, 255:red, 0; green, 0; blue, 0 }  ][fill={rgb, 255:red, 0; green, 0; blue, 0 }  ][line width=0.75]      (0, 0) circle [x radius= 3.35, y radius= 3.35]   ;
	\draw    (372.4,133.3) -- (432.3,133.45) ;
	\draw [shift={(432.3,133.45)}, rotate = 0.14] [color={rgb, 255:red, 0; green, 0; blue, 0 }  ][fill={rgb, 255:red, 0; green, 0; blue, 0 }  ][line width=0.75]      (0, 0) circle [x radius= 3.35, y radius= 3.35]   ;
	\draw    (312.5,133.15) -- (312.52,73.25) ;
	\draw [shift={(312.52,73.25)}, rotate = 270.02] [color={rgb, 255:red, 0; green, 0; blue, 0 }  ][fill={rgb, 255:red, 0; green, 0; blue, 0 }  ][line width=0.75]      (0, 0) circle [x radius= 3.35, y radius= 3.35]   ;
	\draw    (192.7,132.85) -- (192.72,72.95) ;
	\draw [shift={(192.72,72.95)}, rotate = 270.02] [color={rgb, 255:red, 0; green, 0; blue, 0 }  ][fill={rgb, 255:red, 0; green, 0; blue, 0 }  ][line width=0.75]      (0, 0) circle [x radius= 3.35, y radius= 3.35]   ;
	\draw    (252.6,11) ;
	\draw [shift={(252.6,11)}, rotate = 0] [color={rgb, 255:red, 0; green, 0; blue, 0 }  ][fill={rgb, 255:red, 0; green, 0; blue, 0 }  ][line width=0.75]      (0, 0) circle [x radius= 3.35, y radius= 3.35]   ;
	\draw    (432.3,133.45) -- (492.2,133.6) ;
	\draw [shift={(492.2,133.6)}, rotate = 0.14] [color={rgb, 255:red, 0; green, 0; blue, 0 }  ][fill={rgb, 255:red, 0; green, 0; blue, 0 }  ][line width=0.75]      (0, 0) circle [x radius= 3.35, y radius= 3.35]   ;
	
	\draw (428,158.4) node [anchor=north west][inner sep=0.75pt]  [font=\footnotesize] [align=left] {1};
	\draw (369,158.4) node [anchor=north west][inner sep=0.75pt]  [font=\footnotesize] [align=left] {9};
	\draw (309,158.4) node [anchor=north west][inner sep=0.75pt]  [font=\footnotesize] [align=left] {6};
	\draw (246,158.4) node [anchor=north west][inner sep=0.75pt]  [font=\footnotesize] [align=left] {15};
	\draw (189,158.4) node [anchor=north west][inner sep=0.75pt]  [font=\footnotesize] [align=left] {5};
	\draw (126,158.4) node [anchor=north west][inner sep=0.75pt]  [font=\footnotesize] [align=left] {14};
	\draw (69,158.4) node [anchor=north west][inner sep=0.75pt]  [font=\footnotesize] [align=left] {4};
	\draw (6,158.4) node [anchor=north west][inner sep=0.75pt]  [font=\footnotesize] [align=left] {11};
	\draw (152,68.4) node [anchor=north west][inner sep=0.75pt]  [font=\footnotesize] [align=left] {13};
	\draw (488,158.4) node [anchor=north west][inner sep=0.75pt]  [font=\footnotesize] [align=left] {7};
	\draw (215,5.4) node [anchor=north west][inner sep=0.75pt]  [font=\footnotesize] [align=left] {10};
	\draw (272,68.4) node [anchor=north west][inner sep=0.75pt]  [font=\footnotesize] [align=left] {12};
\end{tikzpicture}

%% file: Gamma_proof_2.tex
\tikzset{every picture/.style={line width=0.75pt}} 

\begin{tikzpicture}[x=0.75pt,y=0.75pt,yscale=-1,xscale=1]
	
	\draw  [color={rgb, 255:red, 184; green, 233; blue, 134 }  ,draw opacity=1 ][fill={rgb, 255:red, 184; green, 233; blue, 134 }  ,fill opacity=1 ] (539.55,389.85) .. controls (539.55,382.92) and (545.17,377.3) .. (552.1,377.3) .. controls (559.03,377.3) and (564.65,382.92) .. (564.65,389.85) .. controls (564.65,396.78) and (559.03,402.4) .. (552.1,402.4) .. controls (545.17,402.4) and (539.55,396.78) .. (539.55,389.85) -- cycle ;
	\draw  [color={rgb, 255:red, 184; green, 233; blue, 134 }  ,draw opacity=1 ][fill={rgb, 255:red, 184; green, 233; blue, 134 }  ,fill opacity=1 ] (419.75,192.45) .. controls (419.75,185.52) and (425.37,179.9) .. (432.3,179.9) .. controls (439.23,179.9) and (444.85,185.52) .. (444.85,192.45) .. controls (444.85,199.38) and (439.23,205) .. (432.3,205) .. controls (425.37,205) and (419.75,199.38) .. (419.75,192.45) -- cycle ;
	\draw  [color={rgb, 255:red, 184; green, 233; blue, 134 }  ,draw opacity=1 ][fill={rgb, 255:red, 184; green, 233; blue, 134 }  ,fill opacity=1 ] (299.95,192.15) .. controls (299.95,185.22) and (305.57,179.6) .. (312.5,179.6) .. controls (319.43,179.6) and (325.05,185.22) .. (325.05,192.15) .. controls (325.05,199.08) and (319.43,204.7) .. (312.5,204.7) .. controls (305.57,204.7) and (299.95,199.08) .. (299.95,192.15) -- cycle ;
	\draw  [color={rgb, 255:red, 184; green, 233; blue, 134 }  ,draw opacity=1 ][fill={rgb, 255:red, 184; green, 233; blue, 134 }  ,fill opacity=1 ] (180.15,191.85) .. controls (180.15,184.92) and (185.77,179.3) .. (192.7,179.3) .. controls (199.63,179.3) and (205.25,184.92) .. (205.25,191.85) .. controls (205.25,198.78) and (199.63,204.4) .. (192.7,204.4) .. controls (185.77,204.4) and (180.15,198.78) .. (180.15,191.85) -- cycle ;
	\draw  [color={rgb, 255:red, 184; green, 233; blue, 134 }  ,draw opacity=1 ][fill={rgb, 255:red, 184; green, 233; blue, 134 }  ,fill opacity=1 ] (180.19,72.05) .. controls (180.19,65.12) and (185.81,59.5) .. (192.74,59.5) .. controls (199.68,59.5) and (205.29,65.12) .. (205.29,72.05) .. controls (205.29,78.98) and (199.68,84.6) .. (192.74,84.6) .. controls (185.81,84.6) and (180.19,78.98) .. (180.19,72.05) -- cycle ;
	\draw  [color={rgb, 255:red, 184; green, 233; blue, 134 }  ,draw opacity=1 ][fill={rgb, 255:red, 184; green, 233; blue, 134 }  ,fill opacity=1 ] (60.35,191.55) .. controls (60.35,184.62) and (65.97,179) .. (72.9,179) .. controls (79.83,179) and (85.45,184.62) .. (85.45,191.55) .. controls (85.45,198.48) and (79.83,204.1) .. (72.9,204.1) .. controls (65.97,204.1) and (60.35,198.48) .. (60.35,191.55) -- cycle ;
	\draw    (13,191.4) -- (132.8,191.7) ;
	\draw [shift={(132.8,191.7)}, rotate = 0.14] [color={rgb, 255:red, 0; green, 0; blue, 0 }  ][fill={rgb, 255:red, 0; green, 0; blue, 0 }  ][line width=0.75]      (0, 0) circle [x radius= 3.35, y radius= 3.35]   ;
	\draw [shift={(72.9,191.55)}, rotate = 0.14] [color={rgb, 255:red, 0; green, 0; blue, 0 }  ][fill={rgb, 255:red, 0; green, 0; blue, 0 }  ][line width=0.75]      (0, 0) circle [x radius= 3.35, y radius= 3.35]   ;
	\draw [shift={(13,191.4)}, rotate = 0.14] [color={rgb, 255:red, 0; green, 0; blue, 0 }  ][fill={rgb, 255:red, 0; green, 0; blue, 0 }  ][line width=0.75]      (0, 0) circle [x radius= 3.35, y radius= 3.35]   ;
	\draw    (132.8,191.7) -- (252.6,192) ;
	\draw [shift={(252.6,192)}, rotate = 0.14] [color={rgb, 255:red, 0; green, 0; blue, 0 }  ][fill={rgb, 255:red, 0; green, 0; blue, 0 }  ][line width=0.75]      (0, 0) circle [x radius= 3.35, y radius= 3.35]   ;
	\draw [shift={(192.7,191.85)}, rotate = 0.14] [color={rgb, 255:red, 0; green, 0; blue, 0 }  ][fill={rgb, 255:red, 0; green, 0; blue, 0 }  ][line width=0.75]      (0, 0) circle [x radius= 3.35, y radius= 3.35]   ;
	\draw [shift={(132.8,191.7)}, rotate = 0.14] [color={rgb, 255:red, 0; green, 0; blue, 0 }  ][fill={rgb, 255:red, 0; green, 0; blue, 0 }  ][line width=0.75]      (0, 0) circle [x radius= 3.35, y radius= 3.35]   ;
	\draw    (252.6,192) -- (372.4,192.3) ;
	\draw [shift={(372.4,192.3)}, rotate = 0.14] [color={rgb, 255:red, 0; green, 0; blue, 0 }  ][fill={rgb, 255:red, 0; green, 0; blue, 0 }  ][line width=0.75]      (0, 0) circle [x radius= 3.35, y radius= 3.35]   ;
	\draw [shift={(312.5,192.15)}, rotate = 0.14] [color={rgb, 255:red, 0; green, 0; blue, 0 }  ][fill={rgb, 255:red, 0; green, 0; blue, 0 }  ][line width=0.75]      (0, 0) circle [x radius= 3.35, y radius= 3.35]   ;
	\draw [shift={(252.6,192)}, rotate = 0.14] [color={rgb, 255:red, 0; green, 0; blue, 0 }  ][fill={rgb, 255:red, 0; green, 0; blue, 0 }  ][line width=0.75]      (0, 0) circle [x radius= 3.35, y radius= 3.35]   ;
	\draw    (372.4,192.3) -- (432.3,192.45) ;
	\draw [shift={(432.3,192.45)}, rotate = 0.14] [color={rgb, 255:red, 0; green, 0; blue, 0 }  ][fill={rgb, 255:red, 0; green, 0; blue, 0 }  ][line width=0.75]      (0, 0) circle [x radius= 3.35, y radius= 3.35]   ;
	\draw    (312.5,192.15) -- (312.52,132.25) ;
	\draw [shift={(312.52,132.25)}, rotate = 270.02] [color={rgb, 255:red, 0; green, 0; blue, 0 }  ][fill={rgb, 255:red, 0; green, 0; blue, 0 }  ][line width=0.75]      (0, 0) circle [x radius= 3.35, y radius= 3.35]   ;
	\draw    (192.7,191.85) -- (192.72,131.95) ;
	\draw [shift={(192.72,131.95)}, rotate = 270.02] [color={rgb, 255:red, 0; green, 0; blue, 0 }  ][fill={rgb, 255:red, 0; green, 0; blue, 0 }  ][line width=0.75]      (0, 0) circle [x radius= 3.35, y radius= 3.35]   ;
	\draw    (192.72,131.95) -- (192.74,72.05) ;
	\draw [shift={(192.74,72.05)}, rotate = 270.02] [color={rgb, 255:red, 0; green, 0; blue, 0 }  ][fill={rgb, 255:red, 0; green, 0; blue, 0 }  ][line width=0.75]      (0, 0) circle [x radius= 3.35, y radius= 3.35]   ;
	\draw    (192.74,72.05) -- (192.77,12.15) ;
	\draw [shift={(192.77,12.15)}, rotate = 270.02] [color={rgb, 255:red, 0; green, 0; blue, 0 }  ][fill={rgb, 255:red, 0; green, 0; blue, 0 }  ][line width=0.75]      (0, 0) circle [x radius= 3.35, y radius= 3.35]   ;
	\draw  [color={rgb, 255:red, 184; green, 233; blue, 134 }  ,draw opacity=1 ][fill={rgb, 255:red, 184; green, 233; blue, 134 }  ,fill opacity=1 ] (419.75,389.55) .. controls (419.75,382.62) and (425.37,377) .. (432.3,377) .. controls (439.23,377) and (444.85,382.62) .. (444.85,389.55) .. controls (444.85,396.48) and (439.23,402.1) .. (432.3,402.1) .. controls (425.37,402.1) and (419.75,396.48) .. (419.75,389.55) -- cycle ;
	\draw  [color={rgb, 255:red, 184; green, 233; blue, 134 }  ,draw opacity=1 ][fill={rgb, 255:red, 184; green, 233; blue, 134 }  ,fill opacity=1 ] (299.95,389.25) .. controls (299.95,382.32) and (305.57,376.7) .. (312.5,376.7) .. controls (319.43,376.7) and (325.05,382.32) .. (325.05,389.25) .. controls (325.05,396.18) and (319.43,401.8) .. (312.5,401.8) .. controls (305.57,401.8) and (299.95,396.18) .. (299.95,389.25) -- cycle ;
	\draw  [color={rgb, 255:red, 184; green, 233; blue, 134 }  ,draw opacity=1 ][fill={rgb, 255:red, 184; green, 233; blue, 134 }  ,fill opacity=1 ] (180.15,388.95) .. controls (180.15,382.02) and (185.77,376.4) .. (192.7,376.4) .. controls (199.63,376.4) and (205.25,382.02) .. (205.25,388.95) .. controls (205.25,395.88) and (199.63,401.5) .. (192.7,401.5) .. controls (185.77,401.5) and (180.15,395.88) .. (180.15,388.95) -- cycle ;
	\draw  [color={rgb, 255:red, 184; green, 233; blue, 134 }  ,draw opacity=1 ][fill={rgb, 255:red, 184; green, 233; blue, 134 }  ,fill opacity=1 ] (60.35,388.65) .. controls (60.35,381.72) and (65.97,376.1) .. (72.9,376.1) .. controls (79.83,376.1) and (85.45,381.72) .. (85.45,388.65) .. controls (85.45,395.58) and (79.83,401.2) .. (72.9,401.2) .. controls (65.97,401.2) and (60.35,395.58) .. (60.35,388.65) -- cycle ;
	\draw    (13,388.5) -- (132.8,388.8) ;
	\draw [shift={(132.8,388.8)}, rotate = 0.14] [color={rgb, 255:red, 0; green, 0; blue, 0 }  ][fill={rgb, 255:red, 0; green, 0; blue, 0 }  ][line width=0.75]      (0, 0) circle [x radius= 3.35, y radius= 3.35]   ;
	\draw [shift={(72.9,388.65)}, rotate = 0.14] [color={rgb, 255:red, 0; green, 0; blue, 0 }  ][fill={rgb, 255:red, 0; green, 0; blue, 0 }  ][line width=0.75]      (0, 0) circle [x radius= 3.35, y radius= 3.35]   ;
	\draw [shift={(13,388.5)}, rotate = 0.14] [color={rgb, 255:red, 0; green, 0; blue, 0 }  ][fill={rgb, 255:red, 0; green, 0; blue, 0 }  ][line width=0.75]      (0, 0) circle [x radius= 3.35, y radius= 3.35]   ;
	\draw    (132.8,388.8) -- (252.6,389.1) ;
	\draw [shift={(252.6,389.1)}, rotate = 0.14] [color={rgb, 255:red, 0; green, 0; blue, 0 }  ][fill={rgb, 255:red, 0; green, 0; blue, 0 }  ][line width=0.75]      (0, 0) circle [x radius= 3.35, y radius= 3.35]   ;
	\draw [shift={(192.7,388.95)}, rotate = 0.14] [color={rgb, 255:red, 0; green, 0; blue, 0 }  ][fill={rgb, 255:red, 0; green, 0; blue, 0 }  ][line width=0.75]      (0, 0) circle [x radius= 3.35, y radius= 3.35]   ;
	\draw [shift={(132.8,388.8)}, rotate = 0.14] [color={rgb, 255:red, 0; green, 0; blue, 0 }  ][fill={rgb, 255:red, 0; green, 0; blue, 0 }  ][line width=0.75]      (0, 0) circle [x radius= 3.35, y radius= 3.35]   ;
	\draw    (252.6,389.1) -- (372.4,389.4) ;
	\draw [shift={(372.4,389.4)}, rotate = 0.14] [color={rgb, 255:red, 0; green, 0; blue, 0 }  ][fill={rgb, 255:red, 0; green, 0; blue, 0 }  ][line width=0.75]      (0, 0) circle [x radius= 3.35, y radius= 3.35]   ;
	\draw [shift={(312.5,389.25)}, rotate = 0.14] [color={rgb, 255:red, 0; green, 0; blue, 0 }  ][fill={rgb, 255:red, 0; green, 0; blue, 0 }  ][line width=0.75]      (0, 0) circle [x radius= 3.35, y radius= 3.35]   ;
	\draw [shift={(252.6,389.1)}, rotate = 0.14] [color={rgb, 255:red, 0; green, 0; blue, 0 }  ][fill={rgb, 255:red, 0; green, 0; blue, 0 }  ][line width=0.75]      (0, 0) circle [x radius= 3.35, y radius= 3.35]   ;
	\draw    (372.4,389.4) -- (432.3,389.55) ;
	\draw [shift={(432.3,389.55)}, rotate = 0.14] [color={rgb, 255:red, 0; green, 0; blue, 0 }  ][fill={rgb, 255:red, 0; green, 0; blue, 0 }  ][line width=0.75]      (0, 0) circle [x radius= 3.35, y radius= 3.35]   ;
	\draw    (312.5,389.25) -- (312.52,329.35) ;
	\draw [shift={(312.52,329.35)}, rotate = 270.02] [color={rgb, 255:red, 0; green, 0; blue, 0 }  ][fill={rgb, 255:red, 0; green, 0; blue, 0 }  ][line width=0.75]      (0, 0) circle [x radius= 3.35, y radius= 3.35]   ;
	\draw    (432.3,389.55) -- (492.2,389.7) ;
	\draw [shift={(492.2,389.7)}, rotate = 0.14] [color={rgb, 255:red, 0; green, 0; blue, 0 }  ][fill={rgb, 255:red, 0; green, 0; blue, 0 }  ][line width=0.75]      (0, 0) circle [x radius= 3.35, y radius= 3.35]   ;
	\draw    (492.2,389.7) -- (552.1,389.85) ;
	\draw [shift={(552.1,389.85)}, rotate = 0.14] [color={rgb, 255:red, 0; green, 0; blue, 0 }  ][fill={rgb, 255:red, 0; green, 0; blue, 0 }  ][line width=0.75]      (0, 0) circle [x radius= 3.35, y radius= 3.35]   ;
	\draw    (552.1,389.85) -- (612,390) ;
	\draw [shift={(612,390)}, rotate = 0.14] [color={rgb, 255:red, 0; green, 0; blue, 0 }  ][fill={rgb, 255:red, 0; green, 0; blue, 0 }  ][line width=0.75]      (0, 0) circle [x radius= 3.35, y radius= 3.35]   ;
	
	\draw (428,217.4) node [anchor=north west][inner sep=0.75pt]  [font=\footnotesize] [align=left] {1};
	\draw (369,217.4) node [anchor=north west][inner sep=0.75pt]  [font=\footnotesize] [align=left] {9};
	\draw (309,217.4) node [anchor=north west][inner sep=0.75pt]  [font=\footnotesize] [align=left] {6};
	\draw (246,217.4) node [anchor=north west][inner sep=0.75pt]  [font=\footnotesize] [align=left] {15};
	\draw (189,217.4) node [anchor=north west][inner sep=0.75pt]  [font=\footnotesize] [align=left] {5};
	\draw (126,217.4) node [anchor=north west][inner sep=0.75pt]  [font=\footnotesize] [align=left] {13};
	\draw (69,217.4) node [anchor=north west][inner sep=0.75pt]  [font=\footnotesize] [align=left] {2};
	\draw (6,217.4) node [anchor=north west][inner sep=0.75pt]  [font=\footnotesize] [align=left] {10};
	\draw (152,127.4) node [anchor=north west][inner sep=0.75pt]  [font=\footnotesize] [align=left] {14};
	\draw (154,67.4) node [anchor=north west][inner sep=0.75pt]  [font=\footnotesize] [align=left] {4};
	\draw (151,7.4) node [anchor=north west][inner sep=0.75pt]  [font=\footnotesize] [align=left] {11};
	\draw (272,127.4) node [anchor=north west][inner sep=0.75pt]  [font=\footnotesize] [align=left] {12};
	\draw (428,414.5) node [anchor=north west][inner sep=0.75pt]  [font=\footnotesize] [align=left] {1};
	\draw (369,414.5) node [anchor=north west][inner sep=0.75pt]  [font=\footnotesize] [align=left] {9};
	\draw (309,414.5) node [anchor=north west][inner sep=0.75pt]  [font=\footnotesize] [align=left] {6};
	\draw (246,414.5) node [anchor=north west][inner sep=0.75pt]  [font=\footnotesize] [align=left] {15};
	\draw (189,414.5) node [anchor=north west][inner sep=0.75pt]  [font=\footnotesize] [align=left] {5};
	\draw (126,414.5) node [anchor=north west][inner sep=0.75pt]  [font=\footnotesize] [align=left] {13};
	\draw (69,414.5) node [anchor=north west][inner sep=0.75pt]  [font=\footnotesize] [align=left] {2};
	\draw (6,414.5) node [anchor=north west][inner sep=0.75pt]  [font=\footnotesize] [align=left] {10};
	\draw (490,414.5) node [anchor=north west][inner sep=0.75pt]  [font=\footnotesize] [align=left] {8};
	\draw (548,414.5) node [anchor=north west][inner sep=0.75pt]  [font=\footnotesize] [align=left] {4};
	\draw (605,414.5) node [anchor=north west][inner sep=0.75pt]  [font=\footnotesize] [align=left] {11};
	\draw (272,324.5) node [anchor=north west][inner sep=0.75pt]  [font=\footnotesize] [align=left] {12};
\end{tikzpicture}

%% file: 3folds.bbl
\providecommand{\bysame}{\leavevmode\hbox to3em{\hrulefill}\thinspace}
\providecommand{\MR}{\relax\ifhmode\unskip\space\fi MR }
\providecommand{\MRhref}[2]{%
  \href{http://www.ams.org/mathscinet-getitem?mr=#1}{#2}
}
\providecommand{\href}[2]{#2}
\begin{thebibliography}{CMGHL21}

\bibitem[AGLV98]{AGLV}
V.I. Arnold, V.V. Goryunov, O.V. Lyashko, and V.A. Vasil'ev, \emph{Singularity
  theory {I}}, 1 ed., Springer, 1998.

\bibitem[AGZV85]{Arnold2012}
V.I. Arnold, S.M. Gusein-Zade, and A.N. Varchenko, \emph{Singularities of
  differentiable maps}, 1 ed., vol.~1, Birkh\"{a}user, 1985.

\bibitem[All03]{allcock1}
Daniel Allcock, \emph{The moduli space of cubic threefolds}, J. Algebraic Geom.
  \textbf{12} (2003), no.~2, 201--223.

\bibitem[Bea96]{beauville}
Arnaud Beauville, \emph{Complex algebraic surfaces}, 2 ed., London Math. Soc.
  Student Texts, Cambridge University Press, 1996.

\bibitem[Bri79]{Brieskorn}
E.~Brieskorn, \emph{Die {H}ierarchie der {$1$}-modularen {S}ingularit\"{a}ten},
  Manuscripta Math. \textbf{27} (1979), no.~2, 183--219.

\bibitem[BW79]{BW}
J.~W. Bruce and C.~T.~C. Wall, \emph{On the classification of cubic surfaces},
  J. London Math. Soc. (2) \textbf{19} (1979), 245--256.

\bibitem[CG72]{CG}
C.~Herbert Clemens and Phillip~A. Griffiths, \emph{The intermediate {J}acobian
  of the cubic threefold}, Ann. of Math. (2) \textbf{95} (1972), 281--356.

\bibitem[CMGHL21]{CMGHL}
Sebastian Casalaina-Martin, Samuel Grushevsky, Klaus Hulek, and Radu Laza,
  \emph{Complete moduli of cubic threefolds and their intermediate
  {J}acobians}, Proc. Lond. Math. Soc. (3) \textbf{122} (2021), no.~2,
  259--316.

\bibitem[Dol83]{Dolgachev}
Igor Dolgachev, \emph{Integral quadratic forms: applications to algebraic
  geometry}, S\'eminaire N. Bourbaki \textbf{1982/83} (1983), no.~611,
  251--278.

\bibitem[Dol16]{DolgSegre}
\bysame, \emph{Corrado {S}egre and nodal cubic threefolds}, From classical to
  modern algebraic geometry, Trends Hist. Sci., Birkh\"{a}user, 2016,
  pp.~429--450.

\bibitem[dPW00a]{DPWI}
A.~A. du~Plessis and C.~T.~C. Wall, \emph{Hypersurfaces in {${\rm P}^n({\bf
  C})$} with one-parameter symmetry groups}, R. Soc. Lond. Proc. Ser. A Math.
  Phys. Eng. Sci. \textbf{456} (2000), no.~2002, 2515--2541.

\bibitem[dPW00b]{DPW00}
\bysame, \emph{Singular hypersurfaces, versality, and {G}orenstein algebras},
  J. Algebraic Geom. \textbf{9} (2000), no.~2, 309--322.

\bibitem[dPW08]{DPW08}
\bysame, \emph{Hypersurfaces with isolated singularities with symmetry},
  Contemp. Math. \textbf{459} (2008), 147--164.

\bibitem[Ebe19]{Ebeling2019}
Wolfgang Ebeling, \emph{Distinguished bases and monodromy of complex
  hypersurface singularities}, arXiv:1905.12435 [math.AG] (2019), 48 pp.

\bibitem[GH12]{GrHulek}
Samuel Grushevsky and Klaus Hulek, \emph{The class of the locus of intermediate
  {J}acobians of cubic threefolds}, Invent. Math. \textbf{190} (2012), no.~1,
  119--168.

\bibitem[Loo84]{loo}
E.~J.~N. Looijenga, \emph{Isolated singular points on complete intersections},
  1 ed., London Mathematical Society Lecture Note Series, vol.~77, Cambridge
  University Press, 1984.

\bibitem[LPZ18]{LPZ}
Radu Laza, Gregory Pearlstein, and Zheng Zhang, \emph{On the moduli space of
  pairs consisting of a cubic threefold and a hyperplane}, Adv. Math.
  \textbf{340} (2018), 684--722.

\bibitem[LSV17]{LSV}
Radu Laza, Giulia Sacc\`a, and Claire Voisin, \emph{A hyper-{K}\"{a}hler
  compactification of the intermediate {J}acobian fibration associated with a
  cubic 4-fold}, Acta Math. \textbf{218} (2017), no.~1, 55--135.

\bibitem[Moe14]{Moe}
Torgunn~Karoline Moe, \emph{Rational cuspidal curves with four cusps on
  {H}irzebruch surfaces}, Le Matematiche \textbf{69} (2014), no.~2, 295--318.

\bibitem[Mon13]{Mongardi}
Giovanni Mongardi, \emph{Automorphisms of {H}yperk\"{a}hler manifolds}, Ph.D.
  thesis, Universit\`a degli Studi di RomaTRE, 2013.

\bibitem[Nik80]{Nikulin}
V.V. Nikulin, \emph{Integral symmetric bilinear forms and some of their
  applications}, Mathematics of the USSR-Izvestiya \textbf{14} (1980), no.~1,
  103--167.

\bibitem[Sch63]{Schlafli}
Ludwig Schl\"{a}fli, \emph{On the distribution of surfaces of the third order
  into species}, Philos. Trans. Roy. Soc. \textbf{153} (1863), 193--241.

\bibitem[Seg87]{Segre}
Corrado Segre, \emph{Sulla variet\`a cubica con dieci punti doppi dello spazio
  a quattro dimensioni}, Atti R. Acc. Sci. Torino \textbf{22} (1887), 791--801.

\bibitem[Slo80]{slodowy}
Peter Slodowy, \emph{Simple singularities and simple algebraic groups}, 1 ed.,
  Lecture Notes in Mathematics, vol. 815, Springer-Verlag, 1980.

\bibitem[Ste20]{Stegmann}
Ann-Kathrin Stegmann, \emph{Cubic fourfolds with {ADE} singularities and {K}3
  surfaces}, Ph.D. thesis, Gottfried Wilhelm Leibniz Universit\"{a}t Hannover,
  2020.

\bibitem[Ura87]{urabe}
Tohsuke Urabe, \emph{Elementary transformations of {D}ynkin graphs and
  singularities on quartic surfaces}, Invent. Math. \textbf{87} (1987), no.~3,
  549--572.

\bibitem[Wal99]{W99}
C.~T.~C. Wall, \emph{Sextic curves and quartic surfaces with higher
  singularities}, preprint (1999), 32 pp.

\bibitem[Yan94]{Yang}
Jin-gen Yang, \emph{Rational double points on a normal quintic {$K3$} surface},
  Acta Math. Sinica (N.S.) \textbf{10} (1994), no.~4, 348--361.

\bibitem[Yok02]{yokoyama}
Mutsumi Yokoyama, \emph{Stability of cubic 3-folds}, Tokyo J. Math. \textbf{25}
  (2002), no.~1, 85--105.

\end{thebibliography}
